\documentclass[10pt]{amsart}
\usepackage{amsmath,amsthm,amsfonts,amssymb,amscd,euscript,epsfig}
\usepackage{enumerate,color,soul}

\usepackage{tikz}
\usetikzlibrary{automata,arrows}

\newlength{\defbaselineskip}
\theoremstyle{plain}
\newtheorem{theorem}{Theorem}[section]
\newtheorem{proposition}[theorem]{Proposition}

\newtheorem{lemma}[theorem]{Lemma}

\theoremstyle{definition}
\newtheorem{definition}[theorem]{Definition}

\newtheorem{remark}[theorem]{Remark}
\newtheorem{example}[theorem]{Example}

\DeclareMathOperator{\Ant}{Ant}
\DeclareMathOperator{\area}{area}
\DeclareMathOperator{\arm}{arm}

\DeclareMathOperator{\Cat}{Cat}
\DeclareMathOperator{\Def}{Def}
\DeclareMathOperator{\Des}{Des}
\DeclareMathOperator{\defc}{defc}

\DeclareMathOperator{\dinv}{dinv}
\DeclareMathOperator{\ftype}{ftype}

\DeclareMathOperator{\leg}{leg}
\DeclareMathOperator{\len}{len}

\DeclareMathOperator{\tail}{\textsc{tail}}
\DeclareMathOperator{\TI}{TI}
\newcommand{\arNU}{\stackrel{\text{\sc nu}_1}{\rightarrow}}
\newcommand{\NU}{\text{\sc nu}}
\newcommand{\ND}{\text{\sc nd}}

\newcommand{\qdvmap}{\textsc{qdv}}

\newcommand{\mind}{\mathrm{min}_{\Delta}}

\newcommand{\CC}{\mathfrak{C}}
\newcommand{\C}{\mathcal{C}}

\newcommand{\mcP}{\mathcal{P}}

\newcommand{\calS}{\mathcal{S}}
\newcommand{\calT}{\mathcal{T}}
\newcommand{\ptn}[1]{\langle #1\rangle}

\numberwithin{equation}{section}

\begin{document}
\title{Chain decompositions of $q,t$-Catalan numbers: 
 \\ tail extensions and flagpole partitions}

\date{\today} 

\author{Seongjune Han}
\address{Department of Mathematics \\
University of Alabama \\
Tuscaloosa, AL 35401}
\email{shan25@crimson.ua.edu}

\author{Kyungyong Lee}
\address{Department of Mathematics \\
University of Alabama \\
Tuscaloosa, AL 35401; 
and Korea Institute for Advanced Study \\
Seoul 02455, Republic of Korea}
\email{kyungyong.lee@ua.edu; klee1@kias.re.kr}
\thanks{
The second author was supported by NSF grant DMS 2042786, the Korea Institute 
for Advanced Study (KIAS), and the University of Alabama.}

\author{Li Li}
\address{Department of Mathematics and Statistics \\
 Oakland University \\ 
 Rochester, MI 48309}
\email{li2345@oakland.edu} 

\author{Nicholas A. Loehr}
\address{Department of Mathematics \\
 Virginia Tech \\
 Blacksburg, VA 24061-0123} %
\email{nloehr@vt.edu} 
\thanks{This work was supported by a grant
from the Simons Foundation/SFARI (Grant \#633564 to N.A.L.).}


\begin{abstract} 
This article is part of an ongoing investigation of the combinatorics of
$q,t$-Catalan numbers $\Cat_n(q,t)$. We develop a structure theory
for integer partitions based on the partition statistics dinv, deficit,
and minimum triangle height. Our goal is to decompose the infinite set
of partitions of deficit $k$ into a disjoint union of chains $\C_{\mu}$ 
indexed by partitions of size $k$. Among other structural properties, these 
chains can be paired to give refinements of the famous symmetry property 
$\Cat_n(q,t)=\Cat_n(t,q)$. Previously, we introduced a map that
builds the tail part of each chain $\C_{\mu}$. Our first main contribution
here is to extend this map to construct larger second-order tails for each 
chain.  Second, we introduce new classes of partitions called
flagpole partitions and generalized flagpole partitions.
Third, we describe a recursive construction for building the
chain $\C_{\mu}$ for a (generalized) flagpole partition $\mu$, 
assuming that the chains indexed by certain specific
smaller partitions (depending on $\mu$) are already known.
We also give some enumerative and asymptotic results for
flagpole partitions and their generalized versions.
\end{abstract}

\maketitle

\section{Introduction}
\label{sec:intro}

This article is the third in a series of papers developing the combinatorics
of the $q,t$-Catalan numbers $\Cat_n(q,t)$. We refer readers to
Haglund's monograph~\cite{hag-book} for background and references 
on $q,t$-Catalan numbers. 
Our motivating problem~\cite{gill-blog} is to find
a purely combinatorial proof of the symmetry property
$\Cat_n(q,t)=\Cat_n(t,q)$. It turns out that this symmetry is just one facet
of an elaborate new structure theory for integer partitions. Each partition
$\gamma$ has a \emph{size} $|\gamma|$, a \emph{diagonal inversion count}
$\dinv(\gamma)$, a \emph{deficit} $\defc(\gamma)$, and a 
\emph{minimum triangle height} $\mind(\gamma)$.
Here, $\mind(\gamma)$ is the smallest $n$ such that the Ferrers
diagram of $\gamma$ is contained in the diagram of 
$\Delta_n=\ptn{n-1,n-2,\ldots,3,2,1,0}$, $\dinv(\gamma)$
counts certain boxes in the diagram of $\gamma$, and $\defc(\gamma)$
counts the remaining boxes (see~\S\ref{subsec:area-dinv-defc}
for details).  The $q,t$-Catalan numbers can be defined combinatorially as
\begin{equation}\label{eq:Cat_n}
\Cat_n(q,t)=\sum_{\gamma:\,\mind(\gamma)\leq n}
              q^{\binom{n}{2}-|\gamma|}t^{\dinv(\gamma)}.
\end{equation}

To explain the term ``deficit,'' note that
$|\gamma|=\dinv(\gamma)+\defc(\gamma)$ holds by definition.
So the monomials in~\eqref{eq:Cat_n} indexed by partitions $\gamma$ with
a given deficit $k$ are precisely the monomials of degree $\binom{n}{2}-k$ 
in $\Cat_n(q,t)$.  To prove symmetry of the full polynomial $\Cat_n(q,t)$, 
it suffices to prove the symmetry of each homogeneous component of degree 
$\binom{n}{2}-k$.  Fixing $k$ and letting $n$ vary, we are led to study the 
collection $\Def(k)$ of all integer partitions $\gamma$ with $\defc(\gamma)=k$.
Informally, the structural complexity of these collections 
(stratified by the dinv statistic) grows exponentially with $k$;
see Theorem~\ref{thm:defc-count} for a more precise version of this remark.

\subsection{Global Chains}
\label{subsec:intro-global}

The first paper in our series~\cite{LLL18} introduced the idea of
\emph{global chain decompositions} for the collections $\Def(k)$.
One might observe that each term of degree $\binom{n}{2}-k$
in $\Cat_n(q,t)$ has coefficient at most $p(k)$, the number of 
integer partitions of $k$ (cf. Theorem~1.3 of~\cite{LLL13}
and Theorem~\ref{thm:defc-count} below).  This suggests the 
possibility of decomposing each set $\Def(k)$ into a disjoint union of
\emph{global chains} $\C_{\mu}$ indexed by partitions $\mu$ with $|\mu|=k$,
where each $\C_{\mu}$ should be an infinite sequence of partitions having 
constant deficit $k$ and consecutive dinv values. In other words, 
each conjectural global chain should have the form
$\C_{\mu}=(c_{\mu}(i): i\geq i_0)$, where $\defc(c_{\mu}(i))=k=|\mu|$
and $\dinv(c_{\mu}(i))=i$ for all integers $i$ starting at
some value $i_0$ that depends on $\mu$. The sequence
of $\mind$-values $(\mind(c_{\mu}(i)): i\geq i_0)$, 
which we call the \emph{$\mind$-profile of the chain $\C_{\mu}$},
has intricate combinatorial structure that is crucial to understanding
the symmetry of $q,t$-Catalan numbers.

More specifically, we conjectured in~\cite{LLL18} that the chains
$\C_{\mu}$ (satisfying the above conditions) could be chosen to satisfy
the following \emph{opposite property}. Given a proposed chain
$\C_{\mu}$, define
\begin{equation}\label{eq:Cat_n,mu}
\Cat_{n,\mu}(q,t)=\sum_{\gamma\in\C_{\mu}:\,\mind(\gamma)\leq n}
              q^{\binom{n}{2}-|\gamma|}t^{\dinv(\gamma)},
\end{equation}
which is the sum of only those terms in~\eqref{eq:Cat_n}
arising from partitions $\gamma$ in the chain $\C_{\mu}$.
We conjecture there is a size-preserving involution $\mu\mapsto\mu^*$ 
(defined on the set of all integer partitions)
such that for all integers $n>0$,
$\Cat_{n,\mu^*}(q,t)=\Cat_{n,\mu}(t,q)$. Each pair $\{\mu,\mu^*\}$
yields new small slices of the full $q,t$-Catalan polynomials 
(namely $\Cat_{n,\mu}(q,t)+\Cat_{n,\mu^*}(q,t)$ if $\mu\neq\mu^*$,
or $\Cat_{n,\mu}(q,t)$ if $\mu=\mu^*$) that are symmetric in $q$ and $t$.


\begin{example}\label{ex:C61} 
We constructed the global chains for $\mu=\ptn{6,1}$
and $\mu^*=\ptn{3,3,1}$ in~\cite[Appendix 4.3]{HLLL20}.
All partitions in these chains have deficit $k=7$.
The $\mind$-profiles for $\C_{\mu}$ and $\C_{\mu^*}$ are shown here: 
\[ \left[\begin{array}{rrrrrrrrrrrrrrrrrrrrrrr}
\dinv:&3 & 4 & 5 & 6 & 7 & 8 & 9 & 10 & 11 & 12 & 13 & 14 & 15 & 16 & 17 & 18 
& \cdots & 25 & 26 & 27 & \cdots\\
\mind:&7 & 8 & 7 & 7 & 7 & 8 & 7 &  7 &  8 &  8 &  8 &  8 &  8 &  8 &  8 & 9
& \cdots & 9 & 10 & 9 & \cdots
\end{array}\right];  \] 
\[ \left[\begin{array}{rrrrrrrrrrrrrrrrrrr}
\dinv:&2&3&4 & 5 & 6 & 7 & 8 & 9 & 10 & 11 & 12 & 13 & 14 & 15 & 16 & 17 & 18 
&\cdots\\
\mind:&9&10&7 & 7 & 8 & 7 & 7 &  7 &  8 &  7 &  8 &  8 &  8 &  8 &  8 & 8 & 8
&\cdots
\end{array}\right].  \] 
In both cases, all values of $\mind$ not shown are at least $9$. Taking $n=7$, 
we find $\binom{n}{2}-k=14$ and
\begin{align*}
\Cat_{7,\mu}(q,t)& =q^{11}t^3+q^9t^5+q^8t^6+q^7t^7+q^5t^9+q^4t^{10}; \\
\Cat_{7,\mu^*}(q,t) &=q^{10}t^4+q^9t^5+q^7t^7+q^6t^8+q^5t^9+q^3t^{11}
=\Cat_{7,\mu}(t,q).
\end{align*} 
Despite the apparent irregularity of these profiles, the same opposite
property holds for all $n$. The value $n=9$ is especially striking:
here $\Cat_{9,\mu^*}(q,t)$ and $\Cat_{9,\mu}(t,q)$ both equal
$\sum_{d=2}^{26} q^{29-d}t^d$ with $q^{26}t^3$ omitted, due to the 
two displayed $10$s in the $\mind$-profiles.
\end{example}
 
\subsection{Local Chains}
\label{subsec:intro-local}

The second paper in our series~\cite{HLLL20} introduced important
technical tools called \emph{local chains}, which guide our construction
of global chains and greatly simplify the task of verifying the
opposite property for given $\C_{\mu}$ and $\C_{\mu^*}$. The main idea is
that each global chain should be the union of certain overlapping local chains
whose $\mind$-profiles have special relationships. In fact, we showed
that any proposed global chain $\C_{\mu}$ can be decomposed into local
chains in at most one way. The $\mind$-profiles of the global chain and
its local constituents can be distilled into lists of integers called
the \emph{$amh$-vectors}. We proved that chains $\C_{\mu}$ and $\C_{\mu^*}$
have the opposite property if the $amh$-vectors for these chains
satisfy three easily checkable conditions (illustrated in the next example
and fully explained in~\S\ref{sec:review-local}).
These ideas enabled us to build all global chains and verify the opposite 
property for all deficit values $k$ up to $11$. These 195 chains are listed 
in~\cite[Appendix]{LLL18} for $0\leq k\leq 6$,
in~\cite[Appendix 4.3]{HLLL20} for $k=7,8,9$, 
and in the extended appendix~\cite{HLLL20ext} for $k=10,11$.
Our first paper~\cite[\S2 and \S4]{LLL18} also constructs two infinite 
families of chains, namely $\C_{\ptn{k}}$ for every $k\geq 0$ 
and $\C_{\ptn{ab-a-1,a-1}}$ for every $a,b\geq 2$. We showed 
in~\cite[\S3 and \S5]{LLL18} that the opposite property holds
if we set $\ptn{k}^*=\ptn{k}$ and $\ptn{ab-a-1,a-1}^*=\ptn{ab-b-1,b-1}$.

\begin{example}
Continuing Example~\ref{ex:C61}, the $amh$-vectors for $\C_{\ptn{61}}$ 
(from~\cite[App. 4.3]{HLLL20}) are $a=(3,5,9,27)$, $m=(0,2,1,0)$, and
$h=(7,7,7,9)$. The $amh$-vectors for $\C_{\ptn{331}}$ are
$a^*=(2,4,7,11)$, $m^*=(0,1,2,0)$, and $h^*=(9,7,7,7)$.
The opposite property for $\C_{\ptn{61}}$ and $\C_{\ptn{331}}$ is
verified by noting that $m^*$ is the reversal of $m$, $h^*$ is the 
reversal of $h$, and $a_i+m_i+7+a^*_{5-i}=\binom{h_i}{2}$ for $i=1,2,3,4$,
where $7=|\ptn{61}|$ is the deficit value.
\end{example}

\subsection{The Successor Map}
\label{subsec:succ-map}

One important tool for building the chains $\C_{\mu}$ is
the \emph{successor map} (called $\NU_1$ in this paper,
and called $\nu$ in~\cite{HLLL20,LLL18}).
For each partition $\gamma$ with deficit $k$ and dinv $i$,
$\NU_1(\gamma)$ (if defined) is a partition with deficit $k$ 
and dinv $i+1$. We would like to build the entire global chain
$\C_{\mu}$ by repeatedly applying $\NU_1$ to some starting partition
$c_{\mu}(i_0)$. The trouble is that $\NU_1$ is not defined for all
partitions. There is a known set of \emph{$\NU_1$-initial objects}
where $\NU_1^{-1}$ is undefined, and there is a known set of
\emph{$\NU_1$-final objects} where $\NU_1$ is undefined
(see~\S\ref{subsec:review-NU} for details). Given
any $\NU_1$-initial partition $\gamma$ of deficit $k$, we obtain
the \emph{$\NU_1$-segment $\NU_1^*(\gamma)$} by starting at $\gamma$
and applying $\NU_1$ as many times as possible. Each $\NU_1$-segment is
either infinite or terminates after finitely many steps at a $\NU_1$-final 
object. (Note that a $\NU_1$-segment cannot cycle back on itself,
since dinv increases with each application of $\NU_1$.)

For each partition $\mu$ of size $k$, we have constructed
a specific $\NU_1$-initial object $\TI(\mu)$, called the \emph{tail-initiator
partition indexed by $\mu$}. The partition $\TI(\mu)$
has deficit $k$ and generates an infinite $\NU_1$-segment 
$\tail(\mu)$, called the \emph{$\NU_1$-tail indexed by $\mu$}.  
Section~\ref{subsec:TI-NU-tail} reviews
the definitions and properties of $\TI(\mu)$ and $\tail(\mu)$ in more detail.
The remaining $\NU_1$-segments consist of finite chains 
called \emph{$\NU_1$-fragments}. For each deficit value $k$,
the challenge is to assemble the huge number of $\NU_1$-fragments and
$\NU_1$-tails of deficit $k$ to produce $p(k)$ global chains $\C_{\mu}$ 
satisfying the opposite property.  

\begin{example}
For the partitions $\mu=\ptn{61}$ and $\mu^*=\ptn{331}$
from Example~\ref{ex:C61}, $\TI(\mu)=\ptn{77654311}$
and $\TI(\mu^*)=\ptn{544311}$. The chain $\C_{\mu}$
is the union of the fragments $\NU_1^*(\ptn{511111})$,
$\NU_1^*(\ptn{3333})$, $\NU_1^*(\ptn{44422})$,
$\NU_1^*(\ptn{554421})$, and $\tail(\ptn{61})=\NU_1^*(\ptn{77654311})$.
The chain $\C_{\mu^*}$ is
the union of fragments $\NU_1^*(\ptn{21111111})$, $\NU_1^*(\ptn{32222})$,
$\NU_1^*(\ptn{43331})$, and $\tail(\ptn{331})=\NU_1^*(\ptn{544311})$.
These chains come from~\cite[App. 4.3]{HLLL20}.
\end{example}

\subsection{Extending the Successor Map.}
\label{subsec:extend-succ}

In the first part of this paper, we extend the map $\NU_1$
by defining a new map $\NU_2$ that acts on certain $\NU_1$-final partitions.
This extension causes many $\NU_1$-fragments to coalesce, making
it easier to assemble global chains. In particular, each original
$\NU_1$-tail starting at $\TI(\mu)$ may now extend backwards to
a new starting object $\TI_2(\mu)$, called the \emph{second-order
tail-initiator indexed by $\mu$}. These generate longer 
\emph{second-order tails} called $\tail_2(\mu)$.  
We proceed to define and study \emph{flagpole partitions}, which are 
partitions $\mu$ satisfying $|\mu|\leq 2\mind(\TI_2(\mu))-8$.
We use the term ``flagpole'' because, informally,
flagpole partitions must end in many parts equal to $1$
(see Remark~\ref{rem:flag-parts} for a precise statement).
So the Ferrers diagram of a flagpole partition 
(drawn in the English style, with the largest part on top)
looks like a flag flying on a pole. A generalized version
of flagpole partitions is introduced later (\S\ref{subsec:def-gen-flag}).

The principal results in the first half of this paper are as follows.
\begin{itemize}
\item Theorems~\ref{thm:tail-profile},~\ref{thm:BDV-tail}, 
and~\ref{thm:find-tail} give detailed structural information 
about the $\NU_1$-tails, including the $\mind$-profile of each tail
and a precise description of the objects in each tail having a 
given value of $\mind$. Theorem~\ref{thm:big-dinv} shows that
for each $k$, all partitions of deficit $k$ with sufficiently
large dinv belong to one of these tails.

\item Theorem~\ref{thm:extend-NU} shows that the maps $\NU_1$ and $\NU_2$ 
assemble to give a bijection $\NU$ (defined on a specific subcollection
of integer partitions) that preserves deficit and increases dinv by $1$.

\item Theorem~\ref{thm:tail2} gives specific characterizations of
which partitions appear in the second-order tails and 
which partitions are second-order tail initiators.

\item Theorem~\ref{thm:NU2-chains} explicitly computes the second-order tails
 that start from objects $\TI_2(\mu)$ having a particular form. Among
 other structural facts, we show that the $\NU_1$-fragments comprising
 such tails are demarcated by the descents in the $\mind$-profile.
 
\item Theorem~\ref{thm:flag-init} characterizes flagpole partitions $\mu$
 based on the form of the Dyck vector representing $\TI_2(\mu)$.
 This leads to an exact enumeration of flagpole partitions as a sum
 of partition numbers (Theorem~\ref{thm:count-flag}).
\end{itemize}

We also did computer experiments to obtain an empirical comparison
of the sizes of $\tail(\mu)$ and $\tail_2(\mu)$. Specifically,
for $k\leq 30$, we divided the number of partitions in $\Def(k)$ 
not in any $\tail_2(\mu)$ by the number of partitions in $\Def(k)$
not in any $\tail(\mu)$. This ratio is always less than $0.38$ and
quite close to that value for $20\leq k\leq 30$. This means that 
introducing $\NU_2$ causes more than $62\%$ of the objects
in the $\NU_1$-fragments to be absorbed into the second-order tails.

\subsection{A Recursive Construction of Certain Chains.}
\label{subsec:rec-construction}

Our ultimate goal in this series of papers
is to build all global chains $\C_{\mu}$ by an elaborate
recursive construction using induction on the deficit value $k$.
Here is an outline of the construction we envisage.
For the base case, all chains $\C_{\mu}$ with (say) $|\mu|\leq 5$
have already been defined successfully. For the induction step,
we consider a particular fixed value of $k\geq 6$.
As the induction hypothesis, we assume that all chains $\C_{\lambda}$ 
with $|\lambda|<k$ are already defined and satisfy various 
technical conditions (including the opposite property of $\C_{\lambda}$ 
and $\C_{\lambda^*}$ and other requirements on the $amh$-vectors).
This information is used to build the chains $\C_{\mu}$ for all partitions
$\mu$ of size $k$ and to verify the corresponding technical conditions
for these chains. 

At this time, we cannot execute the entire recursive construction
just outlined. However, for each particular flagpole partition $\mu$, 
we identify a specific list of \emph{needed partitions} for $\mu$
(Definition~\ref{def:needed}). Assuming that $\C_{\rho}$ and $\C_{\rho^*}$
exist and have required structural properties for each needed partition
$\rho$, we can explicitly define $\mu^*$, $\C_{\mu}$, and $\C_{\mu^*}$
and prove that the new chains satisfy the same properties.
Section~\ref{sec:make-flag-chains} provides all the details of this chain
construction, and Theorem~\ref{thm:extend-mu} proves that the construction
works. Starting with a particular base collection
of known chains, we can apply this construction repeatedly
to augment the given collection with many new chains.
Theorem~\ref{thm:extend-all} describes this process precisely.
Section~\ref{sec:gen-flagpole} extends these results to generalized 
flagpole partitions. Theorem~\ref{thm:gflag-asym} gives a lower bound
on the number of generalized flagpole partitions of size $k$
and an asymptotic estimate based on this bound.

\section{Preliminaries on Dyck Classes, Deficit, $\NU_1$, and Tail Initiators}
\label{sec:background}

This section covers needed background material on quasi-Dyck vectors,
the $\dinv$ and deficit statistics, the original $\NU_1$ map, and
the tail-initiators $\TI(\mu)$. Some new ingredients not found in
earlier papers include: the representation of integer partitions by 
equivalence classes of quasi-Dyck vectors (Proposition~\ref{prop:dyck-class}); 
useful formulas for the deficit statistic (Proposition~\ref{prop:defc-pairs} 
and Lemma~\ref{lem:defc-ineq}); an explicit description of how iterations
of $\NU_1$ act on binary Dyck vectors (Proposition~\ref{prop:bin-NU}); and a 
detailed characterization of the Dyck classes belonging to each $\NU_1$-tail
(Theorems~\ref{thm:tail-profile},~\ref{thm:BDV-tail}, and~\ref{thm:find-tail}).

\subsection{Quasi-Dyck Vectors and Dyck Classes}
\label{subsec:QDV}

A \emph{quasi-Dyck vector} (abbreviated \emph{QDV}) is a sequence
of integers $(v_1,v_2,\ldots,v_n)$ such that $v_1=0$ and $v_{i+1}\leq v_i+1$
for $1\leq i<n$. A \emph{Dyck vector} is a quasi-Dyck vector where
$v_i\geq 0$ for all $i$. We often use word notation for QDVs,
writing $v_1v_2\cdots v_n$ instead of $(v_1,v_2,\ldots,v_n)$. 
The notation $i^c$ always indicates $c$ copies of the symbol $i$, 
as opposed to exponentiation.  For example, $0^312^2(-1)^301^20$ stands 
for the QDV $(0,0,0,1,2,2,-1,-1,-1,0,1,1,0)$.  

A \emph{binary Dyck vector} (abbreviated \emph{BDV}) is a Dyck
vector with all entries in $\{0,1\}$.
A \emph{ternary Dyck vector} (abbreviated \emph{TDV}) is a Dyck
vector with all entries in $\{0,1,2\}$.
For any integer $a$, write $a^-=a-1$ and $a^+=a+1$.
For any list of integers $A=a_1\cdots a_n$,
write $A^-=a_1^-\cdots a_n^-$ and $A^+=a_1^+\cdots a_n^+$.

Let $\lambda$ be an integer partition with $\ell=\ell(\lambda)$
positive parts $\lambda_1\geq\lambda_2\geq\cdots\geq\lambda_{\ell}$.
The \emph{size} of $\lambda$ is $|\lambda|=\sum_{i=1}^{\ell} \lambda_i$.
We define $\lambda_i=0$ for all $i>\ell$.  
For each integer $n>\ell$, we associate with $\lambda$ 
a quasi-Dyck vector of length $n$ by setting 
\begin{equation}\label{eq:qdvmap}
\qdvmap_n(\lambda)=(0-\lambda_n,1-\lambda_{n-1},2-\lambda_{n-2},
 \ldots,i-\lambda_{n-i},\ldots,n-1-\lambda_1). 
\end{equation}
Visually, we obtain this QDV by trying to embed the diagram of
$\lambda$ in the diagram of $\Delta_n=\ptn{n-1,\ldots,2,1,0}$
and counting the number of boxes of $\Delta_n$ in each row
(from bottom to top) that are not in the diagram of $\lambda$.
However, we allow $\lambda$ to protrude outside $\Delta_n$,
which leads to negative entries in the QDV. We define the 
\emph{minimum triangle height} $\mind(\lambda)$ to be the least 
$n>\ell(\lambda)$ such that the diagram of $\lambda$ does fit inside the
diagram of $\Delta_n$. This is also the least $n$ such that
$\qdvmap_n(\lambda)$ has all entries nonnegative.

\begin{example}\label{ex:ptn-tri}
Let $\lambda=\ptn{5441}$. Figure~\ref{fig:qdv} shows the
diagrams of $\lambda$ and $\Delta_n$ for $n=5,6,7,8$.
We have $\qdvmap_5(\lambda)=00(-2)(-1)(-1)$,
 $\qdvmap_6(\lambda)=011(-1)00$,
 $\qdvmap_7(\lambda)=0122011$, 
 $\qdvmap_8(\lambda)=01233122$, and $\mind(\lambda)=7$.  
\end{example}

\begin{figure}
\begin{center}
\epsfig{file=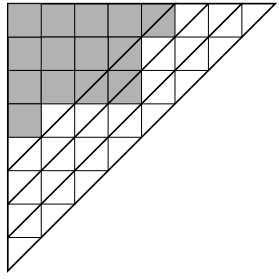,scale=0.7}
\end{center}
\caption{Embedding a partition in various triangles.}
\label{fig:qdv}
\end{figure}

Suppose $n>\ell(\lambda)$ and $\qdvmap_n(\lambda)=v_1v_2\cdots v_n$.
Then $\qdvmap_{n+1}(\lambda) =0(v_1v_2\cdots v_n)^+$. 
Define $\sim$ to be the equivalence
relation on the set of all QDVs generated by the relations
$v_1v_2\cdots v_n\sim 0(v_1v_2\cdots v_n)^+$. So, 
for all QDVs $y=y_1\cdots y_n$ and $z=z_1\cdots z_{n+k}$, 
$y\sim z$ if and only if $z=(0,1,2,\ldots,k-1,y_1+k,y_2+k,\ldots,y_n+k)$.
Each equivalence class of $\sim$ is called a \emph{Dyck class}.
Let $[v]$ denote the Dyck class containing the QDV $v$.

\begin{proposition}\label{prop:dyck-class}
The map sending each partition $\lambda$ to
the Dyck class $\{\qdvmap_n(\lambda):n>\ell(\lambda)\}$
is a bijection from the set of all integer partitions
onto the set of all Dyck classes.
\end{proposition}
\begin{proof}
First, the set $\{\qdvmap_n(\lambda):n>\ell(\lambda)\}$ is a Dyck class.
This holds since $\qdvmap_{m+1}(\lambda)=0\qdvmap_m(\lambda)^+$ for all
$m>\ell(\lambda)$, but the shortest QDV in this set (of length 
$\ell(\lambda)+1$) cannot be equivalent to any shorter QDV, as
the first symbol of a QDV must be $0$. So we have a function $F$ mapping 
integer partitions to Dyck classes. We must show this function is bijective.
Given the Dyck class $[v]$, where $v$ is a QDV of length $n$ representing
the equivalence class, define $\lambda_1=n-1-v_n$, $\lambda_2=n-2-v_{n-1}$, 
etc., and $\lambda_m=0$ for all $m>n$. It is routine to check that
$\lambda$ does not depend on the representative chosen, and that $\lambda$
is the unique integer partition satisfying $F(\lambda)=[v]$.
\end{proof}

Henceforth, we make no distinction between the partition 
$\lambda$ and its associated Dyck class, regarding the list of parts
$\ptn{\lambda_1,\ldots,\lambda_{\ell}}$ and the Dyck class $[v]=F(\lambda)$
as two notations for the same underlying object.
For $n=\mind(\lambda)$, the Dyck vector $\qdvmap_n(\lambda)$ is
called the \emph{reduced Dyck vector for $\lambda$}.
The \emph{reduction} of a QDV $w$ is the unique reduced Dyck vector $v$
with $[w]=[v]$. For example, the reduction of $012\cdots d$ is $0$
for any $d\geq 0$; here $[0]$ is the Dyck class representing the
zero partition $\ptn{0}$, which has $\mind(\ptn{0})=1$.

\subsection{Area, Dinv, and Deficit for Quasi-Dyck Vectors}
\label{subsec:area-dinv-defc}
 
Let $v=v_1v_2\cdots v_n$ be a quasi-Dyck vector.
Define $\len(v)=n$ (the length of the list $v$)
and $\area(v)=v_1+v_2+\cdots+v_n$. 
If $\lambda$ is a partition and $n>\ell(\lambda)$, then~\eqref{eq:qdvmap}
shows that $|\lambda|+\area(\qdvmap_n(\lambda)) =|\Delta_n|=\binom{n}{2}$.

The \emph{diagonal inversion statistic} for a Dyck vector $v$, 
written $\dinv(v)$, is the number of pairs $(i,j)$ 
with $1\leq i<j\leq n$ and $v_i-v_j\in\{0,1\}$.
To generalize this definition to all QDVs $v$, we define $v_k=k-1$
for all $k\leq 0$ and then set $\dinv(v)$ to be the number 
of pairs of integers $(i,j)$ with $i<j\leq n$ and $v_i-v_j\in\{0,1\}$.
Visually, we compute $\dinv(v)$ by looking at the infinite word
$\cdots(-3)(-2)(-1)v_1v_2\cdots v_n$ and counting all pairs of
symbols $\cdots b\cdots b\cdots$ or $\cdots(b+1)\cdots b\cdots$.  
Suppose we replace $v=v_1\cdots v_n$ by the equivalent QDV
$w=w_1w_2\cdots w_{n+1}=0v_1^+\cdots v_n^+$. The infinite word
for $w$ is obtained from the infinite word for $v$ by incrementing
every entry, and therefore $\dinv(w)=\dinv(v)$. It follows that
for all QDVs $v$ and $z$, $v\sim z$ implies $\dinv(v)=\dinv(z)$.
Thus $\dinv$ is constant on Dyck classes. 

The \emph{deficit statistic} for a QDV $v=v_1v_2\cdots v_n$ is
$\defc(v)=\binom{\len(v)}{2}-\area(v)-\dinv(v)$. 
Replacing $v$ by $w$ as above, $\len$ increases from $n$ to $n+1$,
$\binom{\len}{2}$ increases by $n$, $\area$ increases by $n$,
and $\dinv$ does not change. Therefore $\defc(w)=\defc(v)$,
so $\defc$ is constant on Dyck classes.

For a partition $\lambda$ corresponding to a Dyck class $[v]$, 
we set $\dinv(\lambda)=\dinv([v])=\dinv(v)$ and
$\defc(\lambda)=\defc([v])=\defc(v)$. 
Note that $\area$ is \emph{not} constant on Dyck classes.
For $n>\ell(\lambda)$, we define $\area_n(\lambda)=\area(\qdvmap_n(\lambda))$.
For any such $n$, $\dinv(\lambda)+\defc(\lambda)+\area_n(\lambda)
 =\binom{n}{2}=|\lambda|+\area_n(\lambda)$, and hence
$\dinv(\lambda)+\defc(\lambda)=|\lambda|$ for all $\lambda$.  
(It can be shown that $\dinv(\lambda)$ is the number of cells $c$ 
in the diagram of $\lambda$ with $\arm(c)-\leg(c)\in\{0,1\}$, 
and $\defc(\lambda)$ counts the remaining cells,
but we do not need these formulas here.)
Define $\area_{\Delta}(\lambda)=\area_n(\lambda)$ where $n=\mind(\lambda)$; so 
$\area_{\Delta}(\lambda)$ is the area of the reduced Dyck vector for $\lambda$.
In Example~\ref{ex:ptn-tri}, $|\lambda|=14$, 
$\dinv(\lambda)=\dinv(0122011)=10$, 
$\defc(\lambda)=\defc(0122011)=4$, $\mind(\lambda)=7$,
$\area_5(\lambda)=-4$, $\area_6(\lambda)=1$,
$\area_7(\lambda)=7=\area_{\Delta}(\lambda)$, and $\area_8(\lambda)=14$.

The next proposition gives a convenient alternate formula 
for computing $\defc(v)$.

\begin{proposition}\label{prop:defc-pairs}
For any Dyck vector $v=v_1v_2\cdots v_n$, $\defc(v)$ is the number of pairs
$(i,j)$ such that $1\leq i<j\leq n$ and either $v_i-v_j\geq 2$ 
or there exists $k<i$ with $v_k=v_i<v_j$. 
In other words, $\defc(v)$ is the number of pairs
of letters $\cdots b\cdots c\cdots$ in the word $v$ such that either
$b\geq c+2$, or $b<c$ and the displayed $b$ is not the leftmost occurrence
of $b$ in $v$. 
\end{proposition} 
\begin{proof}
On one hand, we know $\binom{n}{2}=\dinv(v)+\defc(v)+\area(v)$.
On the other hand, there are $\binom{n}{2}$ pairs $(i,j)$ with
$1\leq i<j\leq n$. Each such pair satisfies exactly one of the
following conditions: (a)~$v_i-v_j\in\{0,1\}$; 
(b)~$v_i-v_j\geq 2$; (c)~for some $k<i$, $v_k=v_i<v_j$;
(d)~for all $k<i$, $v_k\neq v_i<v_j$. Pairs satisfying (a)
are counted by $\dinv(v)$, while pairs satisfying (b) or (c)
are the pairs mentioned in the lemma statement. So it suffices
to prove that the number of pairs $(i,j)$ satisfying (d) is
$\area(v)=\sum_{j=1}^n v_j$. Consider a fixed $j$ with $v_j=c\neq 0$.
The Dyck vector $v$ has nonnegative integer entries, begins with $0$,
and consecutive entries may increase by at most $1$ reading left to right.
Therefore, $c>0$, and each symbol $0,1,2,\ldots,c-1$ must occur
at least once to the left of position $j$ in $v$.
The leftmost occurrence of each symbol $0,1,\ldots,c-1$ pairs
with $v_j=c$ to give a pair $(i,j)$ of type (d).
Thus, we get exactly $c=v_j$ type (d) pairs $(i,j)$ for this fixed $j$.
The total number of type (d) pairs is $\sum_{j: v_j\neq 0} v_j=\area(v)$, 
as needed.  
%
\end{proof}


\begin{example}\label{ex:defc-calc}
For all integers $n,q\geq 0$, we claim 
$\defc(0^312^n1^q)=2(n+q+1)$. Here there are no pairs of symbols
$b\cdots c$ with $b\geq c+2$. We ignore the leftmost zero;
the next $0$ pairs with $1+n+q$ larger symbols to its right.
The same is true of the third $0$. Ignoring the leftmost $1$,
the remaining $1$s do not pair with any larger symbols to their right
(similarly for the $2$s).  So the total contribution to deficit is $2(1+n+q)$.
\end{example}


\subsection{Some Deficit Calculations}
\label{subsec:defc-ineq}

For any finite list $A$, let $\len(A)$ be the length of $A$.

\begin{lemma}\label{lem:defc-ineq}
(a)~Let $v=AB12^n$ be a Dyck vector where $n\geq 1$ and $A$ either
has at least three $0$s or has two $0$s and at least two $1$s.
Then $\defc(v)\geq 2\len(B)+\defc(A12^n)$.
\\ (b)~Let $v=00A0B12^n1^q$ be a Dyck vector where $n\geq 1$, $q\geq 0$,
 and $A,B$ are lists that might be empty.
Then $\defc(v)\geq 2\len(A)+2\len(B)+2(n+q)+1$.
\end{lemma}
\begin{proof}
(a)~We use the formula in Proposition~\ref{prop:defc-pairs} to justify the
stated lower bound on $\defc(v)$. We first show that each symbol in $B$
contributes at least $2$ to $\defc(v)$.
\emph{Case~1:} Assume $A$ has at least three $0$s.  Each occurrence of $0$ 
in $B$ is not the leftmost $0$ in $v$ and contributes at least $2$ (in fact,
at least $1+n$) to $\defc(v)$ by pairing with one of the symbols in the 
suffix $12^n$. Each occurrence of a symbol $c>0$ in $B$ pairs with the 
second and third $0$s in $A$ to contribute at least $2$ to $\defc(v)$.  
\emph{Case~2:} Assume $A$ has two $0$s and at least two $1$s.
Each $0$ in $B$ contributes at least $2$ to $\defc(v)$, as in Case~1.
Each $1$ in $B$ (which is not the leftmost $1$ in $v$) 
pairs with the second $0$ in $A$ and with each $2$ in
the suffix $12^n$ to contribute at least $2$ to $\defc(v)$.
Each symbol $c\geq 2$ in $B$ pairs with the second $0$ in $A$
and the second $1$ in $A$ to contribute at least $2$ to $\defc(v)$.

So far we have found at least $2\len(B)$ contributions to $\defc(v)$
coming from symbols in $B$ pairing with other symbols. On the other hand,
the subword $A12^n$ of $v$ is a Dyck vector. Any pair of symbols in
this subword contributing to $\defc(A12^n)$ also contributes to $\defc(v)$.
This proves (a).

(b)~Arguing as in Case~1 of (a), we see that each symbol in $B$ 
 and each $0$ in $A$ contributes at least $2$ to $\defc(v)$. 
 Each symbol $c\geq 2$ in $A$ pairs with the zero just before $A$ 
 and the zero just after $A$ in $v$. Finally, 
 each $1$ in $A$ (if any) pairs with the second $0$ in $v$,
 while each $1$ in $A$ except the leftmost $1$ pairs with the rightmost
 $2$ in $v$. So far, symbols in $A$ and $B$ account for at least
 $2\len(A)-1+2\len(B)$ contributions to $\defc(v)$.
 When we delete $A$ and $B$ from $v$, we get the subword $0^312^n1^q$.
 By Example~\ref{ex:defc-calc}, this subword has deficit $2(n+q+1)$,
 and all pairs contributing to this deficit also contribute to $\defc(v)$.
\end{proof}

The next two lemmas will be used later to define the \emph{antipode map},
which interchanges area and dinv for a restricted class of Dyck vectors.

\begin{lemma}\label{lem:stat-0w1}
Let $E$ be a ternary Dyck vector and $S=0E1$.  
Then $\len(S)=\len(E)+2$, $\area(S)=\area(E)+1$,
 $\dinv(S)=\dinv(E)+\len(E)$, and $\defc(S)=\defc(E)+\len(E)>\defc(E)$.  
\end{lemma}
\begin{proof}
The formulas for $\len(S)$ and $\area(S)$ are clear.
Each pair of symbols in $E$ that contribute to $\dinv(E)$
also contribute to $\dinv(S)$. We get additional contributions
to $\dinv(S)$ from the initial $0$ pairing with each $0$ in $E$,
and from the final $1$ pairing with each $1$ and $2$ in $E$.
Since $E$ is ternary, there are exactly $\len(E)$ such pairs.
So $\dinv(S)=\dinv(E)+\len(E)$. The formula for $\defc(S)$
follows from the previous formulas using $\area+\dinv+\defc=\binom{\len}{2}$,
or by an argument based on Proposition~\ref{prop:defc-pairs}.  
\end{proof}

\begin{lemma}\label{lem:semi-stats}
Let $v$ and $z$ be Dyck vectors such that $v=00z^+$.
Then $\len(v)=\len(z)+2$, $\area(v)=\area(z)+\len(z)$,
$\dinv(v)=\dinv(z)+1$, and $\defc(v)=\defc(z)+\len(z)>\defc(z)$.  
\end{lemma}
\begin{proof}
The formulas for $\len(v)$ and $\area(v)$ are clear.
We know $\dinv(0z^+)=\dinv(z)$ since $z\sim 0z^+$. Preceding $0z^+$ with
one more $0$ adds $1$ to $\dinv$, since all symbols in $z^+$ are positive
and the new $0$ only pairs with the $0$ immediately following it. This
proves $\dinv(v)=\dinv(z)+1$. As in Lemma~\ref{lem:stat-0w1}, the formula 
for $\defc(v)$ follows from the definition or via 
Proposition~\ref{prop:defc-pairs}.  
\end{proof}

\subsection{The Maps $\NU_1$ and $\ND_1$}
\label{subsec:review-NU}

We now review the definition and basic properties of the original 
{\sc next-up} map $\NU_1$ (called $\nu$ in~\cite{HLLL20,LLL18}).
For any integer partition $\gamma$, recall $\gamma_1$ is the first
(longest) part of $\gamma$, and $\ell(\gamma)$ is the length
(number of positive parts) of $\gamma$. The domain of $\NU_1$ is
the set $D_1=\{\gamma: \gamma_1\leq\ell(\gamma)+2\}$. 
For $\gamma\in D_1$ of length $\ell$, define 
$\NU_1(\gamma)=\ptn{\ell^+\gamma_1^-\gamma_2^-\cdots\gamma_\ell^-}$.
The map $\NU_1$ is not defined for partitions $\gamma$ outside the set $D_1$.
Such $\gamma$ satisfy $\gamma_1>\ell(\gamma)+2$ and are called
\emph{$\NU_1$-final} objects.

We now define the {\sc next-down} map $\ND_1$ (called $\nu^{-1}$
in~\cite{HLLL20,LLL18}). The domain of $\ND_1$ is the
set $C_1=\{\gamma:\gamma_1\geq\ell(\gamma)\}$. 
For $\gamma\in C_1$ of length $\ell$, define 
$\ND_1(\gamma)=\ptn{\gamma_2^+\gamma_3^+\cdots\gamma_\ell^+1^{\gamma_1-\ell}}$.
The map $\ND_1$ is not defined for partitions $\gamma$ outside the set $C_1$.
Such $\gamma$ satisfy $\gamma_1<\ell(\gamma)$ and are called
\emph{$\NU_1$-initial} objects.

The following proposition summarizes some known properties of the maps
$\NU_1$ and $\ND_1$; see~\cite[\S2.1]{LLL18} for more details. 
Properties (a), (b), and (c) will be used frequently hereafter.

\begin{proposition}\label{prop:NU1-ND1}
 (a) The map $\NU_1:D_1\rightarrow C_1$ is a bijection with
 inverse $\ND_1:C_1\rightarrow D_1$.
\\ (b) For $\gamma\in D_1$, $\defc(\NU_1(\gamma))=\defc(\gamma)$
 and $\dinv(\NU_1(\gamma))=\dinv(\gamma)+1$.
\\ (c) For $\gamma\in C_1$, $\defc(\ND_1(\gamma))=\defc(\gamma)$
 and $\dinv(\ND_1(\gamma))=\dinv(\gamma)-1$.
\\ (d) For $\gamma\in D_1$, $\NU_1$ acts on the Ferrers diagram of $\gamma$
by removing the leftmost column (containing $\ell(\gamma)$ boxes), then adding 
a new top row with $\ell(\gamma)+1$ boxes.
\\ (e) For $\gamma\in C_1$, $\ND_1$ acts on the Ferrers diagram of $\gamma$
by removing the top row (containing $\gamma_1$ boxes), then adding a new 
leftmost column with $\gamma_1-1$ boxes.
\\ (f) For $\gamma\in D_1$, the (finite or infinite) sequence
 $\gamma,\NU_1(\gamma),\NU_1^2(\gamma),\NU_1^3(\gamma),\ldots$
 contains no repeated entries.
\end{proposition}

Note that part (f) follows from part (b), since each object in the
sequence in (f) must have a different value of dinv. Part (f) shows 
that iterating $\NU_1$ can never produce a cycle of partitions.

\begin{example}\label{ex:NU1-ND1}
Given $\gamma=\ptn{5441}$, we compute 
$\NU_1(\gamma)=\ptn{5433}$, $\NU_1^2(\gamma)=\ptn{54322}$,
 $\NU_1^3(\gamma)=\ptn{643211}$, and so on.
On the other hand, $\ND_1(\gamma)=\ptn{5521}$, $\ND_1^2(\gamma)=\ptn{6321}$,
 and $\ND_1^3(\gamma)=\ptn{43211}$, which is a $\NU_1$-initial object.
\end{example}

Since integer partitions correspond bijectively with Dyck classes
(Proposition~\ref{prop:dyck-class}), the maps $\NU_1$ and $\ND_1$
can be viewed as well-defined functions acting on certain Dyck classes
(those corresponding to the partitions in $D_1$ and $C_1$, respectively).
We now describe a convenient formula for computing $\NU_1([v])$
or $\ND_1([v])$ by acting on a representative QDV $v=v_1\cdots v_n$
for the Dyck class $[v]$.
Define the \emph{leader} of $v$ be the largest $d\geq 0$ such that
$v$ starts with the increasing sequence $012\cdots d$. 
Call this first occurrence of $d$ the \emph{leader symbol} of $v$.
With this notation, the following rule is readily verified
(Lemma~2.3 of~\cite{LLL18} proves it for Dyck vectors, and
 the proof easily extends to QDVs).

\begin{proposition}\label{prop:NU-QDV}
Let $v$ be a QDV of length $n>1$ with leader $d$ and last symbol $v_n$.
\\ (a)~Suppose $v_2\geq 0$. In the case $d>v_n+2$, $[v]$ is a $\NU_1$-final
object and $\NU_1([v])$ is not defined. In the case $d\leq v_n+2$,
 $\NU_1([v])=[z]$ where $z$ is obtained from $v$ by deleting the leader
 symbol $d$ and appending $d-1$.
\\ (b)~Suppose $v_n=s\geq -1$ and $[v]\neq [0]$. In the case $d<v_n$,
 $[v]$ is a $\NU_1$-initial object and $\ND_1([v])$ is not defined.
 In the case $d\geq v_n$, $\ND_1([v])=[z]$ where $z$ is obtained from
 $v$ by deleting $v_n$ and inserting $s+1$ immediately after the leftmost $s$
 in $v$. (When $s=-1$, this means putting a new $0$ at the front of $v$.) 
\end{proposition}

It follows that no Dyck class $[v]$ is both $\NU_1$-initial and $\NU_1$-final.

\begin{example} 
We repeat Example~\ref{ex:NU1-ND1} using Dyck vectors;
here $\gamma=\ptn{5441}=[0122011]$. In the following
computation, the leader symbol of each QDV is underlined:
\[ [0\underline{1}12222] \stackrel{\ND_1}{\leftarrow}
   [01\underline{2}2220] \stackrel{\ND_1}{\leftarrow}
   [01\underline{2}2201] \stackrel{\ND_1}{\leftarrow}
   \gamma=[01\underline{2}2011] \stackrel{\NU_1}{\rightarrow}
   [01\underline{2}0111] \stackrel{\NU_1}{\rightarrow}
   [0\underline{1}01111] \stackrel{\NU_1}{\rightarrow} 
   [\underline{0}011110] \stackrel{\NU_1}{\rightarrow} \cdots. \]
\end{example}

We frequently need the fact that a reduced Dyck vector $v$ has 
$\mind([v])=\len(v)$. Using this and Proposition~\ref{prop:NU-QDV}, 
we obtain the following.

\begin{proposition}\label{prop:NU-seg2}
(a)~Let $v$ be a reduced Dyck vector.
If $v$ starts with $00$, then $\NU_1([v])$ is defined and
$\mind(\NU_1([v]))=\mind([v])+1$. If $v$ starts with $01$ and $\NU_1([v])$
is defined, then $\mind(\NU_1([v]))=\mind([v])$.
(b)~Suppose a Dyck vector $v$ starts with $0012$ and ends with a positive 
symbol.  Then $[v]$ is a $\NU_1$-initial object, $\NU_1([v])$ is defined and is
a $\NU_1$-final object, and $\mind(\NU_1([v]))=\mind([v])+1=\len(v)+1$.
\end{proposition}

\begin{proposition}\label{prop:bin-NU}
Let $v=0v_2v_3\cdots v_n$ be a binary Dyck vector of length $n$.
Starting at $[v]$ and applying $\NU_1$ $n$ times leads to $[v0]$
via the following chain of Dyck classes:
\begin{align}\label{eq:bin-NU}
 [v] &=[0v_2v_3\cdots v_n]\arNU 
       [0v_3v_4\cdots v_nv_2^-] \arNU
       [0v_4\cdots v_nv_2^-v_3^-] \arNU \cdots \arNU
       [0v_nv_2^-v_3^-\cdots v_{n-1}^-] 
   \\ &\arNU [0v_2^-v_3^-\cdots v_n^-]=[01v_2v_3\cdots v_n] \arNU
       [0v_2v_3\cdots v_n0]=[v0]. 
\end{align}
\end{proposition}
We call the intermediate vectors $0v_{k+1}\cdots v_nv_2^-\cdots v_k^-$ 
(where $1\leq k\leq n$) \emph{cycled versions of $v$}.  
\begin{proof}
Let $w=0w_2w_3\cdots w_n$ be a QDV with all $w_i$ in $\{-1,0,1\}$.
Suppose $w_2$ is $0$ or $1$. By checking the two cases, 
we see that $\NU_1([w])=[0w_3\cdots w_nw_2^-]$. This observation
justifies the links in~\eqref{eq:bin-NU} leading to
$[0v_2^-v_3^-\cdots v_n^-]$. At this last stage, the
representative QDV might have second symbol $-1$, so we change 
to the new representative $01v_2v_3\cdots v_n$ of length $n+1$.
We then apply the observation at the start of the proof once more
to reach $[v0]$.
\end{proof}


\subsection{Tail-Initiators and $\NU_1$-Tails}
\label{subsec:TI-NU-tail}

\begin{definition}\label{def:TI-tail}
Given a nonzero partition $\mu=\ptn{r^{n_r}\cdots 2^{n_2}1^{n_1}}$
with $n_r>0$, define $B_{\mu}=01^{n_1}01^{n_2}\cdots 01^{n_r}$.
Note that every binary word starting with $0$ and ending with $1$ has
the form $B_{\mu}$ for exactly one such $\mu$. When $\mu=\ptn{0}$,
define $B_{\mu}$ to be the empty word.  For any partition $\mu$,
define the \emph{tail-initiator of $\mu$} to be the Dyck class
$\TI(\mu)=[0B_{\mu}]$. Define $\tail(\mu)$ to be the sequence of
Dyck classes reachable from $\TI(\mu)$ by applying $\NU_1$
zero or more times.
\end{definition}

For example, $\mu=\ptn{33111}=\ptn{3^22^01^3}$ has $\TI(\mu)=[001110011]$,
which is the Dyck class identified with the partition $\ptn{76653211}$.
The map $\TI$ is a bijection from the set of integer partitions
onto the set of Dyck classes $[v]$ such that the reduced representative
$v$ is either $0$ or a BDV starting with $00$ and ending with $1$.

\begin{proposition}\label{prop:TI}
For every partition $\mu$, $\TI(\mu)$ is a $\NU_1$-initial object 
with the following statistics.
\\ (a)~$\mind(\TI(\mu))=\len(0B_{\mu})=\mu_1+\ell(\mu)+1$.
\\ (b)~$\defc(\TI(\mu))=\defc(0B_{\mu})=|\mu|$.
\\ (c)~$\area_{\Delta}(\TI(\mu))=\area(0B_{\mu})=\ell(\mu)$.
\\ (d)~$\dinv(\TI(\mu))=\dinv(0B_{\mu}) 
  =\binom{\mu_1+\ell(\mu)+1}{2}-\ell(\mu)-|\mu|$.
\end{proposition}
\begin{proof}
All statements are immediately verified for $\mu=\ptn{0}$
and $\TI(\mu)=[0]$. Now consider a nonzero partition
$\mu=\ptn{r^{n_r}\cdots 2^{n_2}1^{n_1}}$ and
$B_{\mu}=01^{n_1}01^{n_2}\cdots 01^{n_r}$.
We have $r=\mu_1$ and $n_1+\cdots+n_r=\ell(\mu)$.  
Now $0B_{\mu}$ is the reduced representative of the Dyck class $\TI(\mu)$,
since $0B_{\mu}$ starts with $00$. As $0B_{\mu}$ contains $1+r$ zeroes,
we have $\mind(\TI(\mu))=\len(0B_{\mu})=\mu_1+\ell(\mu)+1$.
Using Proposition~\ref{prop:defc-pairs} to find $\defc(0B_{\mu})$, 
each $1$ in $1^{n_i}$ pairs with $i$ preceding $0$s (not including
the leftmost $0$). So $\defc(0B_{\mu})=1n_1+2n_2+\cdots+rn_r=|\mu|$.
The area of $0B_{\mu}$ is $n_1+\cdots+n_r=\ell(\mu)$.
The formula for $\dinv([0B_{\mu}])$ follows since 
$\dinv+\defc+\area=\binom{\len}{2}$.
Finally, since $0B_{\mu}$ has leader $d=0$ and last symbol $1$, $\TI(\mu)$ is 
a $\NU_1$-initial object by Proposition~\ref{prop:NU-QDV}(b).
\end{proof}

Given any sequence $\C$ of partitions, the \emph{$\mind$-profile}
of $\C$ is the numerical sequence obtained from $\C$ 
by replacing each term $\gamma$ by $\mind(\gamma)$.

\begin{theorem}\label{thm:tail-profile}
For each partition $\mu$, $\tail(\mu)$ is the infinite sequence
$(\NU_1^m(\TI(\mu)):m\geq 0)$, which consists of all $[z]$ 
such that $z$ is a cycled version of $0B_{\mu}0^c$ for some $c\geq 0$. 
Letting $n=\mind(\TI(\mu))=\mu_1+\ell(\mu)+1$, the 
$\mind$-profile of $\tail(\mu)$ is $n^1(n+1)^n(n+2)^{n+1}\cdots
 (n+c+1)^{n+c}\cdots$. All objects in $\tail(\mu)$ have deficit $|\mu|$.
\end{theorem}
\begin{proof}
The sequence $\tail(\mu)$ has first entry $\TI(\mu)=[0B_{\mu}]$, where 
the reduced representative $0B_{\mu}$ is a BDV of length $n=\mind(\TI(\mu))$.
By Proposition~\ref{prop:bin-NU}, applying $\NU_1$ $n$ times leads to
the Dyck class $[0B_{\mu}0]$. The $n$ Dyck classes following
$[0B_{\mu}]$ all have $\mind=n+1$, as we see by inspection
of~\eqref{eq:bin-NU}. We now invoke Proposition~\ref{prop:bin-NU} again,
taking $v$ there to be the BDV $0B_{\mu}0$ of length $n+1$. After $n+1$ 
applications of $\NU_1$, we reach $[0B_{\mu}00]$, where the $n+1$ Dyck classes
following $0B_{\mu}0$ all have $\mind=n+2$. We proceed similarly.
Having reached $[0B_{\mu}0^c]$ for some $c\geq 0$, 
the tail continues to $[0B_{\mu}0^{c+1}]$ in $n+c$ steps.
By Proposition~\ref{prop:bin-NU}, the Dyck classes from
$[0B_{\mu}0^c]$ (inclusive) to $[0B_{\mu}0^{c+1}]$ (exclusive)
are precisely the classes $[z]$ where $z$ is a cycled version of
$0B_{\mu}0^c$. Moreover, the Dyck classes from
$[0B_{\mu}0^c]$ (exclusive) to $[0B_{\mu}0^{c+1}]$ (inclusive)
all have $\mind=n+c+1$. The first part of the theorem follows by 
induction on $c$. All objects in $\tail(\mu)$ have deficit $|\mu|$
by Propositions~\ref{prop:TI}(b) and~\ref{prop:NU1-ND1}(b).
\end{proof}

\subsection{Further Structural Analysis of the $\NU_1$-Tails.}
\label{structure-NU1-tails}

In our later work, we need an even more detailed version
of the description of $\tail(\mu)$ in Theorem~\ref{thm:tail-profile}.
For each $j\geq 0$, let the \emph{$j$th plateau of $\tail(\mu)$} consist
of all $[z]$ in $\tail(\mu)$ with $\mind([z])=n+j$, where $n=\mind(\TI(\mu))$.
The $0$th plateau consists of $\TI(\mu)=[0B_{\mu}]$ alone. For $j>0$,
the $j$th plateau consists of $n+j-1$ objects with
consecutive dinv values, namely all objects strictly after $[0B_{\mu}0^{j-1}]$
and weakly before $[0B_{\mu}0^j]$ in $\tail(\mu)$, as we
saw in the proof of Theorem~\ref{thm:tail-profile}.
The next result explicitly lists all such objects using
reduced representatives for each Dyck class.

\begin{theorem}\label{thm:BDV-tail}
For any nonzero partition $\mu$ and $j>0$, the $j$th plateau of $\tail(\mu)$
consists of the following Dyck classes, listed in order from lowest dinv
 to highest dinv:
\begin{itemize}
\item[(a)] first, $[01Z^+1^{j-1}Y]$ where $Y$ and $Z$ are nonempty strings
 such that $B_{\mu}=YZ$, listed in order from the shortest $Y$ to the 
 longest $Y$; 
\item[(b)] second, $[01^aB_{\mu}0^b]$ where $a+b=j$, listed in order from $b=0$ 
 to $b=j$.
\end{itemize}
 For $\mu=\ptn{0}$ and $j>0$, the $j$th plateau of $\tail(\ptn{0})$
 consists of $[01^a0^b]$ where $a+b=j$ and $b>0$, listed in order 
 from $b=1$ to $b=j$.
\end{theorem}
\begin{proof}
This theorem follows from the calculation~\eqref{eq:bin-NU} applied to
$v=0B_{\mu}0^{j-1}$. Initially, the symbols in $B_{\mu}$ cycle to
the end of the list and decrement, one at a time, producing the Dyck
classes $[0Z0^{j-1}Y^-]=[01Z^+1^{j-1}Y]$ in the order listed in (a).
When all symbols in $B_{\mu}$ have cycled to the end, we have reached
$[00^{j-1}B_{\mu}^-]=[01^jB_{\mu}]$, which is the first Dyck class in (b).
Applying $\NU_1$ $j$ more times in succession gives the remaining objects
in (b) in order, ending with $[0B_{\mu}0^j]$. The special case
$\mu=\ptn{0}$ is different because the objects in (a) do not exist
and the Dyck vector $01^j$ is not reduced. Since $[01^j]=[0^j]$, 
this Dyck class belongs to plateau $j-1$, not plateau $j$.  
Applying~\eqref{eq:bin-NU} to $v=01^j$ proves the theorem in this case.
\end{proof}

We now show that every Dyck class $[w]$ represented by a binary
Dyck vector $w$ belongs to exactly one $\tail(\mu)$,
where $\mu$ can be easily deduced from $w$. This result also
holds when $w$ is a ternary Dyck vector with a particular structure.

\begin{theorem}\label{thm:find-tail} 
 (a)~For each binary Dyck vector $w$, $[w]\in\tail(\mu)$ for 
 exactly one partition $\mu$.
\\ (b)~For each non-reduced ternary Dyck vector $w$,
 $[w]\in\tail(\mu)$ for exactly one partition $\mu$.
\\ (c)~For each reduced ternary Dyck vector $w$ containing $2$,
  $[w]\in\tail(\mu)$ for some (necessarily unique) $\mu$ if and only if 
 $w_1=0$ is the only $0$ in $w$ before the last $2$ in $w$. 
\end{theorem}
\begin{proof} 
Part~(a) is true since every binary string starting with $0$ has the form
given in Theorem~\ref{thm:BDV-tail}(b)
for exactly one choice of $\mu$, $a$, and $b$.
Part~(b) follows from (a) since a non-reduced TDV $w$ has the form 
$w=0z^+$ for some BDV $z$, and $[w]=[0z^+]=[z]$.  
To prove~(c), let $w$ be a reduced TDV containing $2$.
First assume $[w]\in\tail(\mu)$. Since Theorem~\ref{thm:BDV-tail}
lists all reduced Dyck vectors representing Dyck classes in $\tail(\mu)$,
we must have $w=01Z^+1^{j-1}Y$ for some $j>0$ and some
nonempty lists $Y$ and $Z$ with $B_{\mu}=YZ$.  
Every symbol of $Z^+$ is $1$ or $2$ and the last symbol is $2$,
while every symbol of $Y$ is $0$ or $1$ and $Y$ starts with $0$.
Thus, $w$ has only one $0$ before the last $2$.  
Conversely, assume $w$ has only one $0$ before the last $2$. 
Then we can factor $w$ as $w=01Z^+1^{j-1}Y$ 
by letting the last symbol of $Z^+$ be the last $2$ in $w$, and choosing the
maximal $j>0$ to ensure $Y$ starts with $0$. This $0$ must exist, since $w$ is
reduced with only one $0$ before the last $2$.  We see that $YZ$ is a binary
vector of the form $B_{\mu}$, so that $[w]\in\tail(\mu)$ by 
Theorem~\ref{thm:BDV-tail}(a).
\end{proof}

\begin{example}\label{ex:plateau}
(a)~The BDV $w=011110101$ matches the form in Theorem~\ref{thm:BDV-tail}(b)
 with $a=4$, $b=0$, $B_{\mu}=0101$, so $\mu=\ptn{21}$. 
 Therefore $[w]$ is in plateau 4 of $\tail(\ptn{21})$.
\\ (b)~The TDV $w=01211221$ is not reduced; in fact, $[w]=[0100110]$.
 The binary representative matches Theorem~\ref{thm:BDV-tail}(b)
 with $a=b=1$, $B_{\mu}=0011$, and $\mu=\ptn{22}$.
 Therefore $[w]$ is in plateau 2 of $\tail(\ptn{22})$.
\\ (c)~The TDV $w=01122110$ is reduced with only one $0$ before the last $2$.
 This TDV matches the form in Theorem~\ref{thm:BDV-tail}(a)
 with $Z^+=122$, $j=3$, $Y=0$, $B_{\mu}=YZ=0011$, so $\mu=\ptn{22}$.
 Therefore $[w]$ is the first element in plateau 3 of $\tail(\ptn{22})$.
\end{example}

Using a hard result from~\cite{LW09}, we proved the following fact
in Remark~2.3 of~\cite{HLLL20}. 

\begin{theorem}\label{thm:defc-count}
For all $k,d\geq 0$, the number of integer partitions 
with deficit $k$ and dinv $d$ equals the number of integer partitions
of size $k$ with largest part at most $d$. Hence,
for all $d\geq k$, there are exactly $p(k)$ partitions with deficit $k$
and dinv $d$. 
\end{theorem}

As a consequence, we now show that all but finitely many partitions
of deficit $k$ belong to one of the tail sequences $\tail(\mu)$.


\begin{theorem}\label{thm:big-dinv}
For all $k\geq 0$, there exists $d_0(k)$ such that for all $d\geq d_0(k)$,
each partition with deficit $k$ and dinv $d$ appears in exactly one of
the sequences $\tail(\mu)$ as $\mu$ ranges over partitions of size $k$.
\end{theorem}
\begin{proof}
Fix $k\geq 0$. As $\mu$ ranges over all partitions of size $k$,
we obtain $p(k)$ disjoint sequences $\tail(\mu)$, where $\tail(\mu)$
starts at $\TI(\mu)$ and dinv increases by $1$ as we move along each sequence.
Let $d_0(k)$ be the maximum of $\dinv(\TI(\mu))$ over all partitions $\mu$ 
of size $k$. Fix $d\geq d_0(k)$. Then each sequence $\tail(\mu)$ contains a
partition with dinv $d$ and deficit $|\mu|=k$. By Theorem~\ref{thm:defc-count},
these sequences already account for all $p(k)$ partitions with dinv $d$
and deficit $k$. Thus each such partition must belong to one (and only one) 
of these sequences.

Here is a different proof not relying on Theorem~\ref{thm:defc-count}.
Fix $k\geq 0$ and let $d_0(k)=\binom{k+4}{2}+1$. Assume $[v]$ is a Dyck class
with deficit $k$ that belongs to none of the sequences $\tail(\mu)$.
It suffices to prove that $\dinv(v)<d_0(k)$. We may choose $v$
to be a reduced Dyck vector. By Theorem~\ref{thm:find-tail}(a), $v$ cannot be
a binary vector, so $v$ contains a $2$. As $v$ is reduced, $v$ must contain
at least two $0$s.

\emph{Case~1:} $v$ contains a $3$ to the left of the second $0$ in $v$.
Then we can write $v=0A3B0C$ where $A$ and $B$ contain no $0$s. 
We use Proposition~\ref{prop:defc-pairs} to show that $\defc(v)\geq \len(v)-4$.
Each symbol $x\geq 2$ in $A$, $B$, or $C$ pairs with the $0$ following $B$.
Each $x\leq 1$ in $B$ or $C$ pairs with the $3$ before $B$.
Each $1$ in $A$ except the leftmost $1$ pairs with the $3$ after $A$.
Thus, $k=\defc(v)\geq \len(A)-1+\len(B)+\len(C)=\len(v)-4$.

\emph{Case~2:} All symbols in $v$ before the second $0$ are at most $2$.
Here we can write $v=0A0B2C$ where every symbol in $A$ is $1$ or $2$.
Note that the displayed $2$ after $B$ must exist, either because the Dyck
vector $v$ contains a $3$ after the second $0$ or (when $v$ is ternary)
by Theorem~\ref{thm:find-tail}(c).
Here, each $x\geq 2$ in $A$ or $B$ pairs with the $0$ between $A$ and $B$.
Each $x\leq 1$ in $A$ or $B$ (except the leftmost $1$) pairs with 
the $2$ after $B$. Each $0$ in $C$ pairs with the $2$ before $C$,
while other symbols in $C$ pair with the $0$ before $B$. We again have
$k=\defc(v)\geq\len(A)+\len(B)-1+\len(C)=\len(v)-4$.

In both cases, $\dinv(v)\leq\binom{\len(v)}{2}\leq\binom{k+4}{2}<d_0(k)$.  
\end{proof}

\section{Extending the Map $\NU_1$}
\label{sec:extend-NU}

This section extends the function $\NU_1$ to act on
certain $\NU_1$-final objects. Using the inverse of this extended map, 
each infinite $\NU_1$-tail (starting at $\TI(\mu)$, say) can potentially 
be extended backward to a new starting point called $\TI_2(\mu)$.  
This leads to the concept of flagpole partitions in the next section.

\subsection{Two Rules Extending $\NU_1$}
\label{subsec:NU2}

The next definition gives two new rules that extend $\NU_1$.

\begin{definition}\label{def:NU2}
(a)~Assume $h\geq 2$ and $A=A_1\cdots A_s$ is a list of integers 
 such that $A=\emptyset$, or all $A_i\leq 2$ and $A_s\geq 0$
and $A_{i+1}\leq A_i+1$ for all $i<s$.  Define
 $\NU_2([012^hA(-1)^{h-1}])=[00^{h-1}1A1^h]$.
\\ (b)~Assume $k\geq 1$ and $B=B_1\cdots B_s$ is a list of integers such that 
$B=\emptyset$, or all $B_i\leq 2$ and $B_1\leq 1$ and $B_s\geq -1$
and $B_{i+1}\leq B_i+1$ for all $i<s$. Define 
 $\NU_2([012^kB(-1)^k])=[00^kB01^k]$.
\\ (c)~Let $D_2$ be the set of Dyck classes matching one of the input
templates $[012^hA(-1)^{h-1}]$ or $[012^kB(-1)^k]$ in (a) and (b).
Let $C_2$ be the set of Dyck classes matching one of the output
templates $[0^h1A1^h]$ or $[0^{k+1}B01^k]$ in (a) and (b).
\end{definition}

\begin{example}
$\NU_2([012222(-1)001(-1)(-1)])=[00012(-1)001111]$
by letting $h=3$ and $A=2(-1)001$ in rule (a). 
Also, $\NU_2([012211(-1)(-1)(-1)])=[00011(-1)011]$
by letting $k=2$ and $B=11(-1)$ in rule (b).
\end{example}

Recall (Proposition~\ref{prop:NU1-ND1}) that $\NU_1$ is a well-defined 
bijection from the domain $D_1$ onto the codomain $C_1$, where $D_1$
and $C_1$ are defined at the beginning of \S\ref{subsec:review-NU}.

\begin{lemma}\label{lem:NU2-ND2}
The rules in~\ref{def:NU2}(a) and (b) specify a well-defined
bijection $\NU_2:D_2\rightarrow C_2$ with an inverse called
$\ND_2:C_2\rightarrow D_2$. Moreover, $D_2$ is disjoint from $D_1$,
and $C_2$ is disjoint from $C_1$.
\end{lemma}
\begin{proof}
A given Dyck class has at most one representative $v$ ending in $-1$,
which is the only representative that rules (a) and (b) might apply to.
We claim rules (a) and (b) cannot both apply to such a $v$. 
On one hand, since $A$ cannot end in $-1$, the number of $2$s
at the start of the subword $2^hA$ is strictly greater than
the number of $-1$s at the end of $v$ when rule (a) applies.
On the other hand, since $B$ cannot start with $2$, the number
of $2$s at the start of $2^kB$ is not greater than
the number of $-1$s at the end of $v$ when rule (b) applies.
The conditions on $A$ and $B$ ensure that the outputs of the two
rules are valid Dyck classes. This shows that $\NU_2$ is a well-defined
function mapping the domain $D_2$ into the codomain $C_2$.

We define the inverse $\ND_2$ to $\NU_2$ by reversing the rules
in Definition~\ref{def:NU2}. For example, 
\[ \ND_2([0001(-1)001111])=[01222(-1)001(-1)(-1)]
\quad\mbox{ and }\quad\ND_2([000011])=[0122(-1)(-1)]. \]
Reasoning similar to the previous paragraph shows that $\ND_2$ is a 
well-defined function mapping $C_2$ into $D_2$. On one hand,
a Dyck class has at most one representative $v$ beginning with $00$.
On the other hand, the inverse of rule (a) applies only when
the number of $1$s at the end of $v$ weakly exceeds the number
of $0$s at the start of $v$, while the inverse of rule (b) applies
only when the number of initial $0$s strictly exceeds the number of final $1$s.
Thus, the two inverse rules can never both apply to the same object.
Since $\ND_2$ clearly inverts $\NU_2$, we conclude that
$\NU_2:D_2\rightarrow C_2$ is a well-defined bijection
with inverse $\ND_2:C_2\rightarrow D_2$.

Each input $[012^hA(-1)^{h-1}]$ to rule (a) is a $\NU_1$-final object,
since the leader $2$ exceeds the last symbol $-1$ by more than $2$
(Proposition~\ref{prop:NU-QDV}(a)). Similarly, 
each input to rule (b) is a $\NU_1$-final object. This shows that
$D_1$ and $D_2$ are disjoint. Next, each output $[0^h1A1^h]$ for rule (a)
is a $\NU_1$-initial object, since the leader $0$ is less than the
last symbol $1$ (Proposition~\ref{prop:NU-QDV}(b)).  
Likewise, each output for rule (b) is a $\NU_1$-initial object.
So $C_1$ and $C_2$ are disjoint.  
\end{proof}

The next lemma shows that $\NU_2$ has the required effect on 
the dinv and deficit statistics. 

\begin{lemma}\label{lem:NU2-stats}
Acting by $\NU_2$ increases dinv by $1$ and preserves deficit.
\end{lemma}
\begin{proof}
Let $v=012^hA(-1)^{h-1}$ and $v'=00^{h-1}1A1^h$ be the input and output 
representatives appearing in rule~\ref{def:NU2}(a).  For each $s$,
let $n_s(A)$ be the number of copies of $s$ in the list $A$.
We have $\len(v)=2h+1+\len(A)=\len(v')$
    and $\area(v)=h+2+\area(A)=\area(v')+1$.
Next we show $\dinv(v')=\dinv(v)+1$. 
We compute $\dinv(v)$ by starting with $\dinv(01A)$ and adding
contributions involving symbols in the subwords $2^h$ or $(-1)^{h-1}$.
Recall the convention $v_k=k-1$ for all $k\leq 0$; we must count
pairs $\cdots b\cdots b\cdots$ or $\cdots (b+1)\cdots b\cdots$ in
the extended word where one (or both) of the displayed symbols
comes from the subwords $2^h$ or $(-1)^{h-1}$. We get $\binom{h}{2}$
contributions from pairs of $2$s in $2^h$ and $\binom{h-1}{2}$
contributions from pairs of $-1$s in $(-1)^{h-1}$. Each $2$ in $2^h$
contributes nothing when compared to the earlier symbols $\cdots (-2)(-1)01$
or the later symbols $(-1)^{h-1}$. Comparing each $2$ in $2^h$ to
later symbols in $A$ gives $h$ new contributions $n_1(A)+n_2(A)$.
Next, the $h-1$ copies of $-1$ in $(-1)^{h-1}$ each contribute
$2$ (comparing to the initial $-1$ and $0$) and $n_{-1}(A)+n_0(A)$
(comparing to symbols in $A$). The total is
\[ \dinv(v)=\dinv(01A)+\binom{h}{2}+\binom{h-1}{2}
            +h(n_1(A)+n_2(A))+(h-1)(2+n_{-1}(A)+n_0(A)). \]
Similarly, isolating contributions from $0^{h-1}$ and $1^h$ in $v'$, we find 
\[ \dinv(v')=\dinv(01A)+\binom{h-1}{2}+\binom{h}{2}
            +(h-1)(1+n_{-1}(A)+n_0(A))+h(1+n_1(A)+n_2(A)). \]
Comparing the expressions, we get $\dinv(v')=\dinv(v)+1$.
It follows that
\[\defc(v')=\binom{\len(v')}{2}-\area(v')-\dinv(v')
                   =\binom{\len(v)}{2}-(\area(v)-1)-(\dinv(v)+1)=\defc(v).\]

A similar proof works for rule~\ref{def:NU2}(b).
Now $v=012^kB(-1)^k$, $v'=00^kB01^k$, $\len(v)=2k+2+\len(B)=\len(v')$,
and $\area(v)=k+1+\area(B)=\area(v')+1$. Isolating the dinv contributions
of $12^k$ and $(-1)^k$ in $v$, and the dinv contributions
of $0^k$ and $01^k$ in $v'$, we get
\[ \dinv(v')=\dinv(0B)+2\binom{k}{2}+n_0(B)+n_1(B)+k(n_1(B)+n_2(B)
 +n_0(B)+n_{-1}(B))+2k+1=\dinv(v)+1. \] 
So $\defc(v')=\defc(v)$ holds here, too.
\end{proof}

We can now combine the bijections $\NU_1$ and $\NU_2$ to obtain
the extended version of the successor map.


\begin{definition}\label{def:extend-NU}
Let $D=D_1\cup D_2$ and $C=C_1\cup C_2$.
Define the \emph{extended \textsc{next-up} map}
$\NU:D\rightarrow C$ by $\NU(\gamma)=\NU_1(\gamma)$ for $\gamma\in D_1$
and $\NU(\gamma)=\NU_2(\gamma)$ for $\gamma\in D_2$.
Define the \emph{extended \textsc{next-down} map}
$\ND:C\rightarrow D$ by $\ND(\gamma)=\ND_1(\gamma)$ for $\gamma\in C_1$
and $\ND(\gamma)=\ND_2(\gamma)$ for $\gamma\in C_2$.
\end{definition}

The next theorem summarizes the crucial properties of the extended maps.

\begin{theorem}\label{thm:extend-NU}
 (a) The map $\NU:D\rightarrow C$ is a well-defined 
bijection with inverse $\ND:C\rightarrow D$.
\\ (b) For $\gamma\in D$, $\defc(\NU(\gamma))=\defc(\gamma)$
 and $\dinv(\NU(\gamma))=\dinv(\gamma)+1$.
\\ (c) For $\gamma\in C$, $\defc(\ND(\gamma))=\defc(\gamma)$
 and $\dinv(\ND(\gamma))=\dinv(\gamma)-1$.
\\ (d) For $\gamma\in D$, the (finite or infinite) sequence
 $\gamma,\NU(\gamma),\NU^2(\gamma),\NU^3(\gamma),\ldots$
 contains no repeated entries.
\end{theorem}
\begin{proof}
Part (a) follows from Proposition~\ref{prop:NU1-ND1}(a)
and Lemma~\ref{lem:NU2-ND2}. In particular, the combination of $\NU_1$
and $\NU_2$ (resp. $\ND_1$ and $\ND_2$) is a well-defined function
since the domains $D_1$ and $D_2$ (resp. $C_1$ and $C_2$) are disjoint.
Parts (b) and (c) follow from Proposition~\ref{prop:NU1-ND1}(b) and (c)
and Lemma~\ref{lem:NU2-stats}. Part (d) follows from part (b),
since each object in the sequence in (d) must have a different value
of dinv. 
\end{proof}

Part (d) of the theorem assures us that starting at some partition $\gamma$
and iterating $\NU$ can never produce a cycle of partitions. Instead, 
we either get a finite sequence (called a \emph{$\NU$-fragment} 
if $\gamma\not\in C$) or an infinite sequence (called a \emph{$\NU$-tail} 
if $\gamma\not\in C$). A similar remark applies to iterations of $\ND$,
but here the sequence must be finite since dinv cannot be negative.

\subsection{Second-Order Tail Initiators}
\label{subsec:TI2}

\begin{definition}
Given an integer partition $\mu$, the \emph{second-order tail initiator
of $\mu$} is the Dyck class $\TI_2(\mu)$ obtained by starting at
$\TI(\mu)$ and iterating $\ND$ as many times as possible.
Since $\ND$ decreases dinv, this iteration must terminate
 in finitely many steps. The \emph{second-order tail indexed by $\mu$}
is $\tail_2(\mu)=\{\NU^m(\TI_2(\mu)):m\geq 0\}$. All objects in this
tail have deficit $|\mu|$.
\end{definition}

\begin{example}\label{ex:TI2}
The following table shows $\mu$, $\TI(\mu)$, and $\TI_2(\mu)$ for
all partitions of size $4$ or less.
{\footnotesize
\[ \left[\begin{array}{c|c|c|c|c|c|c|c|c|c|c|c|c}
 \mu & \ptn{0} & \ptn{1} & \ptn{2} & \ptn{11}
 & \ptn{3} & \ptn{21} & \ptn{111}
 & \ptn{4} & \ptn{31} & \ptn{22} & \ptn{211} & \ptn{1111}  \\
\TI(\mu) & [0] & [001] & [0001] & [0011] 
 & [00001] & [00101] & [00111]
 & [000001] & [001001] & [00011] & [001101] & [001111] \\
\TI_2(\mu) & [0] & [001] & [0012] & [0011]
 & [01012] & [00121] & [00122]
 & [00012] & [00112] & [00011] & [001222] & [001221]
\end{array}\right] \] } 
The entry for $\mu=\ptn{211}$ is computed as follows:
\[ \TI(\ptn{211})=[001101]
\stackrel{\ND_2}{\rightarrow} [01211(-1)]
\stackrel{\ND_1}{\rightarrow} [001211]
\stackrel{\ND_2}{\rightarrow} [01222(-1)]
\stackrel{\ND_1}{\rightarrow} [001222]=\TI_2(\ptn{211}). \]
Some further values (found with a computer) are:
\begin{equation}\label{eq:TI2-ex}
 \begin{array}{lll}
\TI_2(\ptn{2111})=[0012221],
&\TI_2(\ptn{11111})=[0012222], 
&\TI_2(\ptn{321})=[0012121], \\
\TI_2(\ptn{3111})=[0012212],  
&\TI_2(\ptn{322111})=[001221222], 
&\TI_2(\ptn{4321})=[001212121].   
\end{array} 
\end{equation}
\end{example}

Although $\TI(\mu)$ is built from $\mu$ by a simple explicit formula
(see Definition~\ref{def:TI-tail}),
we do not know any analogous formula for $\TI_2(\mu)$. 
However, we can characterize the set of all Dyck classes $\TI_2(\mu)$ 
as $\mu$ ranges over all partitions. We also prove an explicit criterion
for when a Dyck class belongs to some second-order tail $\tail_2(\mu)$.

\begin{definition}\label{def:cyc-TDV}
A QDV $v$ is a \emph{cycled ternary Dyck vector} if and only if
$v$ is a ternary Dyck vector or $v=A(B^-)$ for some 
ternary Dyck vectors $A$, $B$.  Equivalently, a QDV $v$ is a cycled TDV if 
and only if every $v_i$ is in $\{-1,0,1,2\}$ and there do not exist $j<k$ with
$v_j=-1$ and $v_k=2$.
\end{definition}

\begin{theorem}\label{thm:tail2}
(a)~A Dyck class $[w]$ belongs to $\tail_2(\mu)$ for some partition $\mu$
if and only if $[w]=[v]$ for some cycled ternary Dyck vector $v$.
\\ (b)~A Dyck class $[w]$ has the form $\TI_2(\mu)$ for some partition $\mu$
if and only if $[w]=[v]$ for some ternary Dyck vector $v$ matching one of 
these forms:
\begin{itemize}
\item Type 1: $v=01^m0X2^n$ where $n\geq 1$ and $0\leq m\leq n$ and
 $X$ does not end in $2$.
\item Type 2: $v=0^nY21^m$ where $n\geq 2$ and $0<m<n$ and
 $Y$ does not begin with $0$.
\item Type 3: $v=0^n1^n$ or $v=0^n1^{n-1}$ where $n\geq 2$, or $v=0$.
\end{itemize} 
\end{theorem} 
\begin{proof} 
Let $\calT$ be the set of cycled ternary Dyck vectors,
and let $\calS$ be the set of vectors of type 1, 2, 3 described above.
Note that $\calS\subseteq\calT$.

\emph{Step~1:} We show that for all $v\in\calT$, there exists $z\in\calT$
 with $\NU([v])=[z]$. Fix $v\in\calT$, and consider cases based on the 
 initial symbols in $v$. In the case $v=00R$, Proposition~\ref{prop:NU-QDV}(a)
 gives $\NU([v])=[z]$, where $z=0R(-1)$ is in $\calT$. In the case $v=01R$
 where $R$ does not start with $2$, $\NU([v])=[z]$ where $z=0R0$ is in $\calT$.
 In the case $v=0(-1)R$, we must have every entry of $R$ in $\{-1,0,1\}$.
 Then $[v]=[010R^+]$ where the new input representative is a TDV satsifying 
 the previous case, so the result holds. In the case $v=012R$ where the 
 last symbol of $R$ is at least $0$, $\NU([v])=[z]$ where $z=01R1$ 
 is in $\calT$. The final case is that $v=012^aR(-1)^b$ for some $a,b>0$,
 where we can choose $a$ and $b$ so $R$ does not begin with $2$ and does
 not end with $-1$. If $a>b$, then rule~\ref{def:NU2}(a) applies with
 $h=b+1\leq a$ and $A=2^{a-b-1}R$. We get $\NU([v])=[z]$ for
 $z=0^{b+1}12^{a-b-1}R1^{b+1}$, which is easily seen to be in $\calT$.
 If $a\leq b$, then rule~\ref{def:NU2}(b) applies with
 $k=a$ and $B=R(-1)^{b-a}$. Here we get $\NU([v])=[z]$ for
 $z=0^{a+1}R(-1)^{b-a}01^a$, which is also in $\calT$. 
 
\emph{Step 2:} We show that for all $v\in\calS$, $\ND([v])$ is not defined.  
 Fix $v\in\calS$. Since $\ND([0])$ is undefined, we may assume $v\neq 0$.
 Checking each type, we see that the leader of $v$
 is always less than the last symbol, so $\ND_1([v])$ is not defined
 (Proposition~\ref{prop:NU-QDV}(b)). Next consider the $\ND_2$ rules.
 If $v$ is type 1 with $m=0$, neither rule in~\ref{def:NU2} applies 
 because no representative of $[v]$ starts with $00$ and ends with $1$.
 If $v$ is type 1 with $m>0$, note that $[v]=[0^m(-1)X^-1^n]$. Rule~(a)
 does not apply since $0^m$ is not followed by $1$, while rule~(b) does 
 not apply since $m\leq n$. If $v$ is type 2, rule~(a) does not apply 
 since $n>m$, while rule~(b) does not apply since the final $1$s in $v$
 are preceded by $2$, not $0$.  If $v$ is type 3 with $v\neq 0$, rule~(a)
 does not apply because $v$ starts with too many $0$s, while rule~(b)
 does not apply because $v$ starts with too few $0$s. (Observe that
 when $A$ or $B$ is empty, the inputs to the two $\ND_2$ rules 
 are $[0^h1^{h+1}]$ and $[0^{k+2}1^k]$.)
 
\emph{Step 3:} We show that for all $v\in\calT$, either $[v]=[v']$
 for some $v'\in\calS$ or $\ND([v])=[z]$ for some $z\in\calT$. 
 Fix $v\in\calT$ and consider cases based on the last symbol of $v$.  
 The conclusion holds if $[v]=[0]$ since $0\in\calS$, so assume $[v]\neq [0]$.  
 In the case $v=0R(-1)$, $\ND_1([v])=[z]$ where $z=00R$ is in $\calT$.
 In the case $v=0R0$, $\ND_1([v])=[z]$ where $z=01R$ is in $\calT$.
 In the case $v=01R1$, $\ND_1([v])=[z]$ where $z=012R$ is in $\calT$.
 In the case $v=0(-1)R1$, $[v]=[v']$ where $v'=010R^+2$ 
 is a type $1$ vector in $\calS$ with $m=1$ (note $R$ cannot contain $2$ here).
 In the case $v=00\cdots 1$, we can write $v=0^aR1^b$ where $a\geq 2$, 
$b\geq 1$, $R$ does not start with $0$ or $2$, 
       and $R$ does not end with $1$ or $-1$.  
If $a\leq b$ and $R$ starts with $1$, then rule~\ref{def:NU2}(a) 
 for $\ND_2$ applies and yields an output representative in $\calT$.
If $a>b$ and $R$ ends with $0$, then the same outcome holds using 
 rule~\ref{def:NU2}(b). 
If $a\leq b$ and $R$ starts with $-1$, then $[v]=[v']$ where
 $v'=01^aR^+2^b$ is a type 1 vector in $\calS$ with $m=a$
 (note $2$ cannot appear in $R$).
If $a>b$ and $R$ ends with $2$, then $v$ is a type 2 TDV in $\calS$
 (note $-1$ cannot appear in $R$).
If $R$ is empty, then rule~\ref{def:NU2}(a) applies if $a<b$,
 rule~\ref{def:NU2}(b) applies if $a>b+1$, and $v$ is type 3
 if $a=b$ or $a=b+1$.
In the final case where $v$ ends in $2$, $v$ must be a TDV.
 If $v=01R2$, then we reduce to a previous case by noting
 $[v]=[w]$ where $w=0R^-1\in\calT$.
 If $v=00R2$, then $v$ is a type 1 vector in $\calS$ with $m=0$.

\emph{Step 4:} We prove the ``if'' parts of Theorem~\ref{thm:tail2}.
 Fix arbitrary $v\in\calT$.
 By iteration of Step~1, the $\NU$-segment $U=\{\NU^m([v]):m\geq 0\}$ 
 is infinite, and every Dyck class in $U$ is represented by something 
 in $\calT$. Because $U$ is infinite, it contains Dyck classes
 with arbitrarily large dinv. By Theorem~\ref{thm:big-dinv},
 $U$ must overlap one of the original tails $\tail(\mu)$ for some $\mu$.
 This forces $U$ to be a subsequence of the new tail $\tail_2(\mu)$.
 If the $v$ we started with is in $\calS$, then Step~2 forces $[v]$ to be 
 the initial object in $\tail_2(\mu)$, namely $\TI_2(\mu)$. 
 
\emph{Step 5:} We prove the ``only if'' parts of Theorem~\ref{thm:tail2}.
 Fix a partition $\mu$. Note that $\TI(\mu)$ is a Dyck class represented
 by a binary Dyck vector $v$, which belongs to $\calT$. Applying
 $\ND$ to $[v]$ repeatedly, we get a finite sequence ending
 at $\TI_2(\mu)$. By Steps~2 and 3, we must have $\TI_2(\mu)=[u]$
 for some $u\in\calS$. Now by Step~1, every Dyck class in $\tail_2(\mu)$
 is represented by something in $\calT$.  
\end{proof}

\subsection{Computation of Some $\NU$-Chains}
\label{subsec:compute-NU}

In this section, we compute detailed information about the
$\NU$-chain obtained by iterating $\NU$ starting at a Dyck class
with reduced representative of the form $v=00A^+B$ for some binary 
vectors $A$ and $B$. Since this $\NU$-chain is already understood
if $v$ itself is binary (see Theorem~\ref{thm:find-tail}(a)
and Theorem~\ref{thm:BDV-tail}), we assume here that $A^+$ ends in $2$.
In particular, the results of this section let us explicitly
compute the $\NU$-chain leading from $\TI_2(\mu)$ to $\TI(\mu)$ in the case
where $\TI_2(\mu)$ has reduced representative $v$ of the indicated form.
Referring to the classification in Theorem~\ref{thm:tail2}(b), 
our analysis applies to 
type $1$ vectors $v$ such that $m=0$ and $X$ does not contain $0$, as well as
to type $2$ vectors $v$ such that $n=2$, $m=1$, and $Y$ does not contain $0$.
(Type $3$ vectors are binary and thus present no problems.)
Theorem~\ref{thm:tail2}(b) also shows that not every Dyck class $[00A^+B]$ 
has the form $\TI_2(\mu)$; for simple counterexamples, consider
$v=00Y21^m$ where $Y$ contains only $1$s and $2$s and $m\geq 2$.

\begin{definition}\label{def:Sj}
Let $v=00A^+B$ be a Dyck vector where $A$ and $B$ are binary vectors
with $A^+$ ending in $2$. Define $S_0(v)=v$. Let $S_1(v),S_2(v),\ldots,
S_J(v)$ be the reduced Dyck vectors for all the $\NU_1$-initial objects 
appearing in the $\NU$-chain $(\NU^i([v]):i>0)$, listed in the order
they appear in the chain (in increasing order of dinv). 
Note that $J$ must be finite, since the $\NU$-chain either terminates
or enters some infinite $\NU_1$-tail $\tail(\mu)$
(Theorem~\ref{thm:big-dinv}).
Let $L_j(v)=\len(S_j(v))=\mind([S_j(v)])$ for $0\leq j\leq J$.
We write $L_j$ for $L_j(v)$ when $v$ is understood from context.
\end{definition}

We are most interested in the case where $[v]$ itself is $\NU_1$-initial,
which occurs if and only if $B$ is empty or ends in $1$
(Proposition~\ref{prop:NU-QDV}(b)). In this case, the $\NU$-chain
starting from $[v]$ consists of $\NU_1$-chains starting at each $[S_j(v)]$,
linked together by $\NU_2$-steps arriving at each $[S_j(v)]$ with $j>0$.
The next lemmas show how to compute $S_j(v)$, $L_j(v)$, and the 
$\mind$-profile of the $\NU$-chain starting at $[v]$.


\begin{lemma}\label{lem:m1=012}
Let $v=0012^{m_1}12^{m_2}\cdots 12^{m_s}B$ be a Dyck vector of length 
$L=\mind(v)$ where $s\geq 1$, $m_i\geq 0$ for all $i$, $m_s>0$, 
and $B$ is a binary vector.
\begin{itemize}
\item[(a)] If $m_1=0$, then $\NU^L([v])=[0012^{m_2}\cdots 12^{m_s}1B0]$,
which is reached by applying $\NU_1$ $L$ times.
The $\mind$-profile of $[v],\NU^1([v]),\ldots,\NU^L([v])$ is $L(L+1)^L$.
\item[(b)] If $m_1=1$, then $\NU^2([v])=[0012^{m_2}\cdots 12^{m_s}B01]$,
which is reached by applying $\NU_1$ and then $\NU_2$.
The $\mind$-profile of $[v],\NU([v]),\NU^2([v])$ is $L,L+1,L$.
\item[(c)] If $m_1\geq 2$, then $\NU^2([v])
=[0012^{m_1-2}12^{m_2}\cdots 12^{m_s}B11]$,
which is reached by applying $\NU_1$ and then $\NU_2$.
The $\mind$-profile of $[v],\NU([v]),\NU^2([v])$ is $L,L+1,L$.
\end{itemize}
\end{lemma}
\begin{proof}
(a) We can apply $\NU_1$ to $[v]=[00112^{m_2}\cdots 12^{m_s}B]$ 
repeatedly, using Proposition~\ref{prop:NU-QDV}(a). The first two steps give
 $\NU_1([v])=[0112^{m_2}\cdots 12^{m_s}B(-1)]$ and
 $\NU_1^2([v])=[012^{m_2}\cdots 12^{m_s}B(-1)0]$.
The next $m_2$ applications of $\NU_1$ remove the $m_2$ copies of $2$
from $12^{m_2}$, one at a time, and put $m_2$ copies of $1$ at the end.
Next, the $1$ from $12^{m_2}$ is removed and a $0$ is added to the end.
At this point, $\NU_1^{3+m_2}([v])=[012^{m_3}\cdots 12^{m_s}B(-1)01^{m_2}0]$.
This pattern now continues: in the next $m_3+1$ steps, $\NU_1$ 
gradually removes $12^{m_3}$ from the front and adds $1^{m_3}0$ to the end.
Eventually, we reach $[0B(-1)01^{m_2}01^{m_3}0\cdots 1^{m_s}0]$.
Next, $\NU_1$ removes each symbol of $B$ and puts the corresponding
decremented symbol at the end. 
Since $L-1$ symbols of $v$ have now cycled to the end, we have reached
$\NU^{L-1}([v])=[0(-1)01^{m_2}01^{m_3}0\cdots 1^{m_s}0B^-]$.
All powers $\NU^i([v])$ computed so far have representatives of length $L$ 
with smallest entry $-1$, which implies $\mind(\NU^i([v]))=L+1$ 
for $1\leq i<L$. The reduced representatives for $\NU^i([v])$ 
all begin with $01$ and have length $L+1$.
In particular, $\NU^{L-1}([v])=[01012^{m_2}12^{m_3}1\cdots 2^{m_s}1B]$.
Using this representative, we can do $\NU_1$ one more time to reach 
$\NU^L([v])=[0012^{m_2}12^{m_3}\cdots 12^{m_s}1B0]$,
which also has $\mind$ equal to $L+1$.

(b) Given $m_1=1$, $\NU([v])=\NU_1([v])=[01212^{m_2}\cdots 12^{m_s}B(-1)]$.
As in (a), this Dyck class has $\mind$ equal to $L+1$. 
We continue by applying $\NU_2$ (namely, rule~\ref{def:NU2}(b) with $k=1$) 
to get $\NU^2([v])=[0012^{m_2}\cdots 12^{m_s}B01]$, which has length
and $\mind$ equal to $L$. 

(c) Given $m_1\geq 2$, 
$\NU([v])=\NU_1([v])=[012^{m_1}12^{m_2}\cdots 12^{m_s}B(-1)]$,
which has $\mind$ equal to $L+1$. 
We continue by applying $\NU_2$ (rule~\ref{def:NU2}(a) with $h=2$) to get
$\NU^2([v])=[0012^{m_1-2}12^{m_2}\cdots 12^{m_s}B1^2]$, which has
length and $\mind$ equal to $L$. 
\end{proof}

\begin{remark}\label{rem:mind-descents}
For each application of $\NU$ in Lemma~\ref{lem:m1=012},
doing $\NU_1$ weakly increases the value of $\mind$, 
while doing $\NU_2$ strictly decreases the value of $\mind$.
This remark allows us to identify the $\NU_1$-initial objects
in $\NU$-chains built from iteration of Lemma~\ref{lem:m1=012},
simply by finding descents in the $\mind$-profile of such a $\NU$-chain.
More precisely, for all objects $\NU^i([v])$ mentioned in the lemma
with $i>0$, $\NU^i([v])$ is $\NU_1$-initial
if and only if $\NU^i([v])$ is reached by a $\NU_2$-step
if and only if $\mind(\NU^{i-1}([v]))>\mind(\NU^i([v]))$.
\end{remark}

\begin{example}\label{ex:alg-v*}
Consider the input $v=0012221122$, which matches the template
in Lemma~\ref{lem:m1=012} with $m_1=3$, $m_2=0$, $m_3=2$, $B=\emptyset$,
and $L=10$.  By Lemma~\ref{lem:m1=012}(c), $\NU^2([v])=[0012112211]$.
Now Lemma~\ref{lem:m1=012}(b) applies to input $0012112211$,
giving $\NU^4([v])=[0011221101]$.
Now Lemma~\ref{lem:m1=012}(a) applies to input $0011221101$ of length $10$,
giving $\NU^{14}([v])=[00122111010]$.
Finally, one more application of Lemma~\ref{lem:m1=012}(c)
gives $\NU^{16}([v])=[00111101011]=[0B_{\mu}]=\TI(\mu)$ for $\mu=\ptn{3321^4}$.
The $\mind$-profile of the chain from $[v]$ to $\TI(\mu)$ is
$10,11,10,11,10,11^{10},12,11$. Based on the descents in this profile,
the $\NU_1$-initial objects in this chain are $\NU^i([v])$ for $i=0,2,4,16$.
So $S_0(v)=v$, $S_1(v)=0012112211$, $S_2(v)=0011221101$,
$S_3(v)=00111101011$, $L_0(v)=L_1(v)=L_2(v)=10$, and $L_3(v)=11$.
\end{example}

Iterating Lemma~\ref{lem:m1=012} gives the following result.

\begin{lemma}\label{lem:cycle-m1}
Let $v=0012^{m_1}12^{m_2}\cdots 12^{m_s}B$ be a Dyck vector 
of length $L=\mind(v)$
where $s\geq 1$, $m_i\geq 0$ for all $i$, $m_s>0$, and $B$ is a binary vector.
\begin{itemize}
\item[(a)] If $m_1$ is odd, then $\NU^{m_1+1}([v])
 =[0012^{m_2}\cdots 12^{m_s}B1^{m_1-1}01]$. 
The $\mind$-profile of $(\NU^i([v]):0\leq i\leq m_1+1)$
is $L$ followed by $(m_1+1)/2$ copies of $L+1,L$.
\item[(b)] If $m_1$ is even and $s=1$, then
 $\NU^{m_1}([v])=[001B1^{m_1}]$, which is $\TI(\mu)$ for some $\mu$.
The $\mind$-profile of $(\NU^i([v]):0\leq i\leq m_1)$
is $L$ followed by $m_1/2$ copies of $L+1,L$.
\item[(c)] If $m_1$ is even and $s>1$, then
 $\NU^{m_1+L}([v])=[0012^{m_2}\cdots 12^{m_s}1B1^{m_1}0]$.
The $\mind$-profile of $(\NU^i([v]):0\leq i\leq m_1+L)$ is
$L$ followed by $m_1/2$ copies of $L+1,L$, followed by $(L+1)^L$.
\item[(d)] For all $\NU^i([v])$ mentioned in (a) through (c) 
 with $i>0$, $\NU^i([v])$ is $\NU_1$-initial if and only if
 $\mind(\NU^{i-1}([v]))>\mind(\NU^i([v]))$, which occurs 
 precisely for the even $i\leq m_1+1$. 
\item[(e)] The $\NU_1$-initial objects in (d) are the powers
 $\NU^i([v])=[0012^{m_1-i}12^{m_2}\cdots 12^{m_s}B1^i]$ for 
 even $i\leq m_1$ in cases (a) through (c), along with 
 $\NU^{m_1+1}([v])$ in case (a).
\end{itemize}
\end{lemma}
\begin{proof}
To prove~(a), first apply Lemma~\ref{lem:m1=012}(c) a total of $(m_1-1)/2$
times. The net effect is to remove $m_1-1$ copies of $2$ from $12^{m_1}$ 
and add $1^{m_1-1}$ to the end. Now $m_1$ has been reduced to $1$, 
so Lemma~\ref{lem:m1=012}(b) applies. We do $\NU$ twice more,
removing $12$ from the front and adding $01$ to the end. 
The claims about $\mind$ also follow from Lemma~\ref{lem:m1=012}. 
Part~(b) follows similarly, by applying Lemma~\ref{lem:m1=012}(c) $m_1/2$
times. At this point, all $2$s have been removed (since $s=1$), 
so we have reached a binary Dyck vector representing some $\TI(\mu)$.  
In part~(c), we find $\NU([v]),\ldots,\NU^{m_1}([v])$ using 
Lemma~\ref{lem:m1=012}(c). Since $m_1$ has now been reduced to $0$
but another $2$ still remains,
we can find the next $L$ powers using Lemma~\ref{lem:m1=012}(a).  
Part~(d) follows from Remark~\ref{rem:mind-descents}, since all $\NU$-powers
computed in parts (a) through (c) were found using Lemma~\ref{lem:m1=012}.
Part~(e) follows from part (d) and iteration of Lemma~\ref{lem:m1=012}(c).
\end{proof}

\begin{example}
Given $v=001222212211121221$, let us compute the $\NU$-chain starting
at $[v]$, the $\mind$-profile of this chain, and the reduced vectors $S_j(v)$.
To start, apply Lemma~\ref{lem:cycle-m1} to input $v$, so $L=18$, $s=6$,
$m_1=4$, $m_2=2$, $m_3=m_4=0$, $m_5=1$, $m_6=2$, and $B=1$.
Part (c) of the lemma says $\NU^{22}([v])=[0012211121221^60]$,
and the $\mind$-profile of the $\NU$-chain from $[v]$ to $\NU^{22}([v])$
is $18,19,18,19,18,19^{18}$.  
To continue, apply the lemma to input $0012211121221^60$, so now $L=19$, $s=5$, 
$m_1=2$, $m_2=m_3=0$, $m_4=1$, $m_5=2$, and $B=1^60$.
Part (c) applies again, giving $\NU^{22+21}([v])=[0011121221^70110]$,
and the $\mind$-profile from $\NU^{23}([v])$ through $\NU^{43}([v])$ is 
$20,19,20^{19}$.
At the next stage, the input to the lemma is $0011121221^70110$,
with $L=20$, $s=4$, $m_1=m_2=0$, $m_3=1$, $m_4=2$, and $B=1^70110$.
We reach $\NU^{63}([v])=[001121221^801100]$ and append
$21^{20}$ to the $\mind$-profile.
Next we continue to $\NU^{84}([v])=[00121221^9011000]$ and
append $22^{21}$ to the $\mind$-profile.
At the next step, Lemma~\ref{lem:cycle-m1}(a) applies, taking us
to $\NU^{86}([v])=[001221^901100001]$ and appending $23,22$ to the 
$\mind$-profile. Finally, we use part (b) of the lemma, reaching
$\NU^{88}([v])=[0011^{9}0110000111]$ and appending $23,22$ to the 
$\mind$-profile. We have reached the Dyck class $[0B_{\mu}]=\TI(\mu)$
where $\mu=\ptn{6^32^21^{10}}$.

The full $\mind$-profile from $[v]$ to $\TI(\mu)$ is
$\underline{18},19,\underline{18},19,\underline{18},19^{18},20,
 \underline{19},20^{19},21^{20},22^{21},23,\underline{22},23,\underline{22}$, 
where the underlined values are descents in the $\mind$-profile
corresponding to the $S_j(v)$.  Theorem~\ref{thm:tail-profile} gives the 
rest of the $\mind$-profile as $23^{22},24^{23},25^{24},\ldots$.  
We find the $S_j(v)$ by applying Lemma~\ref{lem:cycle-m1}(e) 
to each of the input vectors used in the previous paragraph.
We obtain $S_0(v)=v$, $S_1(v)=001221221112122111$, $S_2(v)=001122111212211111$, 
$S_3(v)=00111121221^6011$, $S_4(v)=001221^901100001$,
 and $S_5(v)=0011^{9}0110000111$.
\end{example}

The next theorem uses the following notation for sublists of a list.
Let $1^{2\lfloor n/2\rfloor}01^{n\bmod 2}$ denote the
list $1^n0$ for $n$ even and $1^{n-1}01$ for $n$ odd.
For any list $A$, let $(A)^m$ be the concatenation of $m$ copies
of $A$, which is the empty list if $m=0$.

\begin{theorem}\label{thm:NU2-chains}
Let $v=0012^{n_0}12^{n_1}\cdots 12^{n_r}C$ be a Dyck vector 
where $r\geq 0$, $n_i\geq 0$ for all $i$, $n_r>0$, and $C$ is a binary vector.
For $0\leq i\leq r+1$, let $p_i$ be the number of even integers in the list 
$n_0,\ldots,n_{i-1}$.
\begin{itemize}
\item[(a)] The $\NU$-chain starting at $[v]$ contains all Dyck classes
 $[v^{(0)}]=[v],[v^{(1)}],\ldots,[v^{(r+1)}]$, where
\begin{equation}\label{eq:loop-iter}
 v^{(i)}=0012^{n_i}\cdots 12^{n_r}1^{p_i}C
 1^{2\lfloor n_0/2\rfloor} 0 1^{n_0\bmod 2}\cdots
 1^{2\lfloor n_{i-1}/2\rfloor} 0 1^{n_{i-1}\bmod 2}
\quad\mbox{ for $0\leq i\leq r$, and}
\end{equation}
\begin{equation}\label{eq:loop-end}
 v^{(r+1)}=001^{p_{r+1}}C
  1^{2\lfloor n_0/2\rfloor}01^{n_0\bmod 2} \cdots
  1^{2\lfloor n_{r-1}/2\rfloor}01^{n_{r-1}\bmod 2}
  1^{2\lfloor n_r/2\rfloor}(01)^{n_r\bmod 2}.
\end{equation}
We have $[v^{(r+1)}]=\TI(\mu)$ for some partition $\mu$.
\item[(b)] Put $L=\len(v^{(i)})=\mind(v^{(i)})$.
For $i\leq r$, the $\mind$-profile of the part of the $\NU$-chain 
from $[v^{(i)}]$ through $[v^{(i+1)}]$ always starts
$L,(L+1,L)^{\lceil n_i/2\rceil}$; this is followed by $(L+1)^L$ 
if $n_i$ is even and $i<r$.  For $i=r+1$, the $\mind$-profile of the 
$\NU$-chain from $\TI(\mu)$ onward is $L(L+1)^L(L+2)^{L+1}\cdots$.
\item[(c)] For all $i>0$, $\NU^i([v])$ is a $\NU_1$-initial object
 if and only if $\mind(\NU^{i-1}([v]))>\mind(\NU^i([v]))$.
\item[(d)] The reduced vectors $S_j(v)$ representing the $\NU_1$-initial
 objects in the $\NU$-chain starting at $[v]$ are: 
 $v$ itself, if $C$ is empty or ends in $1$;
 $v^{(i)}$, for each $i>0$ such that $n_{i-1}$ is odd; and all vectors 
\begin{equation}\label{eq:tail2-init}
 0012^{n_i-2c}12^{n_{i+1}}\cdots 12^{n_r}1^{p_i}C
 1^{2\lfloor n_0/2\rfloor} 0 1^{n_0\bmod 2}\cdots
 1^{2\lfloor n_{i-1}/2\rfloor} 0 1^{n_{i-1}\bmod 2} 1^{2c}. 
\end{equation}
where $0\leq i\leq r$ and $0<2c\leq n_i$.
The final vector~\eqref{eq:loop-end} is $S_J(v)$.
\item[(e)] Each $S_j(v)$ except $S_J(v)$ contains a $2$.
 For $j>0$, each $S_j(v)$ is a reduced Dyck vector of the form
 $00X1$ with $X$ ternary. No other reduced representatives 
 of classes in the chain from $\NU([v])$ to $[S_J(v)]$ have this form.
 For any Dyck class $\delta$ in the $\NU$-chain starting at $[v]$,
 $\mind(\delta)<\mind(\NU(\delta))$ if and only if the
 reduced Dyck vector for $\delta$ begins with $00$. 
\item[(f)] The $\mind$-profile from $[S_j(v)]$ to just before
 $[S_{j+1}(v)]$ is a prefix of  
 $L_j(L_j+1)^{L_j}(L_j+2)^{L_j+1}\cdots$. Each $[S_j(v)]$ is immediately
 preceded (if $j>0$) and followed by an object with $\mind$ equal to $L_j+1$.
 Hence, $L_0\leq L_1\leq\cdots\leq L_J$.
\item[(g)] For $0\leq j<J$, $\area(S_{j+1}(v))=\area(S_j(v))-2$.
\end{itemize}
\end{theorem}
\begin{proof}
We prove (a) and (b) by induction on $i$, by iterating
Lemma~\ref{lem:cycle-m1}. The $\NU$-chain starts at $[v]=[v^{(0)}]$.
For the induction step, fix $i$ with $0\leq i\leq r$, and assume
the $\NU$-chain has reached $[v^{(i)}]$. Lemma~\ref{lem:cycle-m1}
applies taking $v$ there to be $v^{(i)}$, taking $m_1,\ldots,m_s$
there to be $n_i,\ldots,n_r$, and taking $B$ there to be the part
of $v^{(i)}$ starting with $1^{p_i}$. The lemma
describes how repeated action by $\NU$ removes $12^{n_i}$ from the
front of $v^{(i)}$ and adds new symbols further right. Assume $i<r$ first.
If $n_i$ is odd, then we add $1^{n_i-1}01$ on the right end
(Lemma~\ref{lem:cycle-m1}(a)).
If $n_i$ is even (possibly zero), then we add a $1$
immediately after $2^{n_r}$ and add $1^{n_i}0$ on the right end
(Lemma~\ref{lem:cycle-m1}(c)). Since $p_{i+1}=p_i$ for $n_i$ odd
and $p_{i+1}=p_i+1$ for $n_i$ even, we obtain the correct vector
$v^{(i+1)}$ in both cases. The analysis for $i=r$ is similar,
but now we need Lemma~\ref{lem:cycle-m1}(b) if $n_r$ is even.
In that case, we still add a new $1$ before $1^{p_r}$ 
(in agreement with $p_{r+1}=p_r+1$), but we only add $1^{n_r}$
(not $1^{n_r}0$) at the right end. This explains the term
$(01)^{n_r\bmod 2}$ in~\eqref{eq:loop-end}.
Part (b) of the theorem follows (in all cases just discussed)
from the descriptions of the $\mind$-profiles in Lemma~\ref{lem:cycle-m1}.
Since $v^{(r+1)}$ is a binary vector starting with $00$ and ending with $1$,
$[v^{(r+1)}]=\TI(\mu)$ for some partition $\mu$. The $\mind$-profile
of the $\NU$-tail starting here is given by Theorem~\ref{thm:tail-profile}.

Parts~(c) and~(d) of the theorem follow from parts~(d) and~(e)
of Lemma~\ref{lem:cycle-m1}. The first two statements in part~(e)
of the theorem follow from the explicit description in~\ref{thm:NU2-chains}(d).
The next claim in~(e) follows by checking that the output of
each $\NU_1$-step in the proof of Lemma~\ref{lem:m1=012}
never has reduced representative $00X1$ with $X$ ternary.
The last assertion in~(e) is verified similarly, also using
Theorem~\ref{thm:BDV-tail} to check those $\delta$ in $\tail(\mu)$.
Part~(f) follows from parts~(b) and~(c), recalling that
$L_j=\len(S_j(v))=\mind([S_j(v)])$ since $S_j(v)$ is reduced.

Part~(g) also follows from the computations in Lemma~\ref{lem:m1=012}.
If part~(b) or part~(c) of that lemma applies to input $S_j(v)$,
then $S_{j+1}(v)$ is the reduced representative of $\NU^2([S_j(v)])$.
By that lemma and Theorem~\ref{thm:extend-NU}, $L_{j+1}=L_j$, 
$\dinv(S_{j+1}(v))=\dinv(S_j(v))+2$, $\defc(S_{j+1}(v))=\defc(S_j(v))$,
and $\area(S_{j+1}(v))=\area(S_j(v))-2$.
On the other hand, suppose Lemma~\ref{lem:m1=012}(a) 
applies to input $[S_j(v)]$ for $c>0$ successive times,
which happens when an even $m_i$ has been reduced to zero
 and is followed by $m_{i+1}=\cdots=m_{i+c-1}=0<m_{i+c}$.
In this situation, the $\mind$-profile mentioned in (f)
is the prefix $L_j(L_j+1)^{L_j}\cdots (L_j+c)^{L_j+c-1}(L_j+c+1)$, 
and the next object has reduced representative $S_{j+1}(v)$ 
with length $L_{j+1}=L_j+c$. 
Going from $S_j(v)$ to $S_{j+1}(v)$, we see that the deficit has not changed,
the length has increased from $L_j$ to $L_j+c$, and dinv has increased by 
$L_j+(L_j+1)+\cdots+(L_j+c-1)+2=\binom{L_j+c}{2}-\binom{L_j}{2}+2$. 
Using $\area+\dinv+\defc=\binom{\len}{2}$, 
it follows that $\area(S_{j+1}(v))=\area(S_j(v))-2$.
\end{proof}

We work through a detailed example illustrating Theorem~\ref{thm:NU2-chains}
in~\S\ref{subsec:init-tail2}.

\section{Flagpole Partitions}
\label{sec:flagpole}

The following strange-looking definition will be explained
by Lemma~\ref{lem:flag-init}.

\begin{definition}\label{def:flag}
A \emph{flagpole partition} is an integer partition $\mu$ such that
$|\mu|+8\leq 2\mind(\TI_2(\mu))$.  
\end{definition}

For example, $\mu=\ptn{322111}$ is a flagpole partition since
(from~\eqref{eq:TI2-ex}) $\TI_2(\mu)=[001221222]$, $|\mu|+8=18$,
and $\mind(\TI_2(\mu))=9$. But $\mu=\ptn{22}$ is not a flagpole
partition since $|\mu|+8=12$ while $\TI_2(\mu)=[00011]$ has 
$\mind(\TI_2(\mu))=5$.
The flagpole partitions of size at most $7$ are $\ptn{211}$, $\ptn{1^4}$,
$\ptn{2111}$, $\ptn{1^5}$, $\ptn{321}$, $\ptn{3111}$, $\ptn{21^4}$, 
$\ptn{1^6}$, $\ptn{3211}$, $\ptn{31^4}$, $\ptn{21^5}$, and $\ptn{1^7}$
(cf. Example~\ref{ex:TI2}). 

\subsection{Flagpole Initiators}
\label{subsec:flag-init}

We are going to characterize the second-order tail initiators of
flagpole partitions. This requires the following notation.
Recall that for $\lambda=\ptn{r^{n_r}\cdots 2^{n_2}1^{n_1}}$ with $n_r>0$,
$B_{\lambda}=01^{n_1}01^{n_2}\cdots 01^{n_r}$
and $\len(B_{\lambda})=\lambda_1+\ell(\lambda)$.
Let $a_0(\lambda)=|\lambda|-\lambda_1-\ell(\lambda)+3$.
Note that $a_0(\lambda)\geq 2$ for all partitions $\lambda$,
and equality holds if and only if $\lambda=\ptn{b,1^c}$ is a nonzero hook.
Define $v(\lambda,a,0)=0012^aB_{\lambda}^+$
 and $v(\lambda,a,1)=0012^{a-1}B_{\lambda}^+1$ for all partitions $\lambda$
and integers $a\geq 2$. For example,
\[ v(\ptn{33111},3,0)=00122212221122\quad\mbox{ and }\quad
   v(\ptn{4421},3,1)=00122121211221. \]

\begin{lemma}\label{lem:flag-init}
For all Dyck vectors $v$ listed in Theorem~\ref{thm:tail2}(b),
the following conditions are equivalent:
\\(a)~$\defc(v)+8\leq 2\len(v)$;
\\(b)~there exists a partition $\lambda$ and an integer $a\geq a_0(\lambda)$
 with $v=v(\lambda,a,0)$ or $v=v(\lambda,a,1)$.
\end{lemma}
\begin{proof}
We look at seven cases based on the possible forms of $v$.

\emph{Case~1.} $v=01^m0X2^n$ is type $1$ where $0<m\leq n$
and $X$ has at least two $1$s. Then $v$ has the form $AB12^n$,
where $A=01^m0^r1$ has at least two $0$s and at least two $1$s.
By Lemma~\ref{lem:defc-ineq}(a), $\defc(v)\geq 2\len(B)+\defc(A12^n)$.
By Proposition~\ref{prop:defc-pairs}, $\defc(A12^n)=\defc(01^m0^r112^n)
 =r(2+n)+(m+1)n\geq 2r+2n+mn$. Since $m\leq n$, we get
\[ \defc(v)+8>2\len(B)+2r+2n+6+(m^2+1)\geq 2[\len(B)+r+n+3+m]=2\len(v). \]
Thus condition~\ref{lem:flag-init}(a) is false for $v$, 
and condition~\ref{lem:flag-init}(b) is also false since $v$ 
does not begin with $00$.

\emph{Case~2.} $v=01^m0X2^n$ is type $1$ where $0<m\leq n$
and $X$ has only one $1$. Then $v$ has the form $01^m0^r12^n$,
$\len(v)=m+r+n+2$, and $\defc(v)=r(1+n)+mn$. Here $\defc(v)+8-2\len(v)$
simplifies to $(m-2)(n-2)+(n-1)r$.  If $m\geq 2$, then (using
$n\geq m\geq 2$) we get $(m-2)(n-2)+(n-1)r>0$. If $m=1$, this expression 
becomes $(n-1)(r-1)+1>0$, which is also positive. Thus (a) is false for $v$,
and (b) is false since $v$ does not begin with $00$.

\emph{Case~3.} $v=00X2^n$ is type $1$ where $m=0$ and $X$ has at least one $0$.
Then $v$ has the form $00A0B12^n$, $\len(v)=\len(A)+\len(B)+n+4$,
and Lemma~\ref{lem:defc-ineq}(b) gives $\defc(v)\geq 2\len(A)+2\len(B)+2n+1$.
So $\defc(v)+8>2\len(v)$, and (a) is false for $v$.
Condition~(b) is also false since $v$ has too many $0$s.

\emph{Case~4.} $v=00X2^n$ is type $1$ where $m=0$ and $X$ contains no $0$.
Then there exist a partition $\lambda$ and positive integers $c,a$
such that $v=001^c2^aB_{\lambda}^+$. We compute $\len(v)=2+c+a+\lambda_1
 +\ell(\lambda)$. Using Proposition~\ref{prop:defc-pairs} to find $\defc(v)$,
the second $0$ in $v$ contributes $\len(v)-2$,
each $1$ in $1^c$ except the first contributes $a+\ell(\lambda)$,
and the $1$s in $B_{\lambda}^+$ pair with later $2$s to
contribute $n_1+2n_2+\cdots+rn_r=|\lambda|$.  In total, we get 
\begin{equation}\label{eq:flag-case4}
 8+\defc(v) = \len(v)+6+(c-1)(a+\ell(\lambda))+|\lambda|. 
\end{equation} 
Consider the subcase $c\geq 2$. Here, condition (b) is false for $v$ since $v$
begins with $0011$. On the other hand, because $a\geq 1$, we have $(c-2)a>c-4$
and hence $6+(c-1)a>2+c+a$. By~\eqref{eq:flag-case4},
\begin{equation}\label{eq:flag4a}
 8+\defc(v)>\len(v)+2+c+a+\ell(\lambda)+|\lambda|\geq 2\len(v), 
\end{equation}
so condition (a) is false for $v$.  
In the subcase $c=1$, (b) is true for $v$ iff $a\geq a_0(\lambda)$.
On the other hand, using~\eqref{eq:flag-case4} with $c=1$, (a) is true for
$v$ iff $6+|\lambda|\leq\len(v)$ iff 
$6+|\lambda|\leq 3+a+\lambda_1+\ell(\lambda)$ iff $a\geq a_0(\lambda)$.
Thus, (a) and (b) are equivalent in this subcase.
   
\emph{Case~5.} $v=0^nY21^m$ is type $2$ and contains at least three $0$s.
Then $v=00A0B12^p1^m$ where $p,m>0$, and $\len(v)=\len(A)+\len(B)+4+p+m$.
Lemma~\ref{lem:defc-ineq}(b) gives 
$\defc(v)+8\geq 2\len(A)+2\len(B)+2(p+m)+9>2\len(v)$. So conditions (a)
and (b) are both false for $v$.

\emph{Case~6.} $v=0^nY21^m$ is type $2$ with exactly two $0$s, which
forces $n=2$ and $m=1$. So there exist a partition $\lambda$ and
integers $c\geq 1$, $a\geq 2$ with $v=001^c2^{a-1}B_{\lambda}^+1$.
Similarly to Case~4, we compute $\len(v)=2+c+a+\lambda_1+\ell(\lambda)$ and
\begin{equation}\label{eq:flag-case6}
 8+\defc(v) = \len(v)+6+(c-1)(a-1+\ell(\lambda))+|\lambda|. 
\end{equation} 
In the subcase $c\geq 2$, (b) is false for $v$ since $v$ begins with $0011$.
On the other hand, $(a-2)(c-2)\geq 0$ in this subcase, so
$6+(c-1)(a-1)>2+c+a$.  Using this in~\eqref{eq:flag-case6} 
yields~\eqref{eq:flag4a}, so (a) is false for $v$.  
In the subcase $c=1$, (b) is true for $v$ iff $a\geq a_0(\lambda)$.
By~\eqref{eq:flag-case6} with $c=1$, (a) is true for $v$ iff 
$6+|\lambda|\leq \len(v)$ iff $a\geq a_0(\lambda)$ (as in Case~4).
So (a) and (b) are equivalent in this subcase.

\emph{Case~7.} $v$ is a type 3 vector. Then condition (b) is false for $v$
 since $v$ contains no $2$.  If $v=0^n1^n$ with $n\geq 2$, then
 $\defc(v)=(n-1)n$, $\len(v)=2n$, and it is routine to check $(n-1)n+8>2n$.
 If $v=0^n1^{n-1}$ with $n\geq 2$, then $\defc(v)=(n-1)^2$,
 $\len(v)=2n-1$, and $(n-1)^2+8>2n-1$ holds. If $v=0$, then
 $\defc(v)=0$, $\len(v)=1$, and $8>2$ holds. So condition (a) is false
 for all type 3 vectors $v$.  
\end{proof}

\begin{remark}\label{rem:flag-init}
As seen in the proof, $v(\lambda,a,0)$ and $v(\lambda,a,1)$ both have
length $L=a+3+\lambda_1+\ell(\lambda)$ and deficit 
$a+1+\lambda_1+\ell(\lambda)+|\lambda|=L+|\lambda|-2$.
We also saw that $a\geq a_0(\lambda)$ if and only if $L\geq |\lambda|+6$.
Since $\area(v(\lambda,a,1))=\area(v(\lambda,a,0))-1$, it follows that
      $\dinv(v(\lambda,a,1))=\dinv(v(\lambda,a,0))+1$.
Thus, $v(\lambda,a,0)$ and $v(\lambda,a,1)$ have dinv of opposite parity.
More precisely, one readily checks that 
\begin{equation}\label{eq:TI2-stats}
 \area(v(\lambda,a,\epsilon))=2L-\lambda_1-5-\epsilon\quad\mbox{ and }\quad
\dinv(v(\lambda,a,\epsilon))=\binom{L}{2}-3L-|\lambda|+\lambda_1+7+\epsilon.
\end{equation}
\end{remark}

\begin{theorem}\label{thm:flag-init}
A partition $\mu$ is a flagpole partition if and only if there exist
a partition $\lambda$ and an integer $a\geq a_0(\lambda)$ such that
$\TI_2(\mu)=[v(\lambda,a,0)]$ or $\TI_2(\mu)=[v(\lambda,a,1)]$.
\end{theorem}
\begin{proof}
Given any partition $\mu$, we know $\TI_2(\mu)=[v]$ for some vector
$v$ listed in Theorem~\ref{thm:tail2}(b). 
Since these $v$ are all reduced, $\mind(\TI_2(\mu))=\len(v)$.
Also, $\defc(v) =\defc(\TI_2(\mu))=\defc(\TI(\mu))=|\mu|$.
The theorem now follows from Definition~\ref{def:flag}
 and Lemma~\ref{lem:flag-init}.  
\end{proof}

\begin{definition}\label{def:flag-type}
For any integer partition $\mu$ such that $\TI_2(\mu)=[v(\lambda,a,\epsilon)]$,
we call $\lambda$ the \emph{flag type} of $\mu$ and write
$\lambda=\ftype(\mu)$.
\end{definition}

\subsection{Representations of Flagpole Partitions}
\label{subsec:rep-flagpole}

Theorem~\ref{thm:flag-init} leads to some useful representations
of flagpole partitions involving the flag type and other data.

\begin{lemma}\label{lem:flag-bij1}
Let $F$ be the set of flagpole partitions, and
let $G$ be the set of triples $(\lambda,a,\epsilon)$,
where $\lambda$ is any integer partition, $a$ is an integer
with $a\geq a_0(\lambda)$, and $\epsilon$ is $0$ or $1$.
There is a bijection $\Phi:F\rightarrow G$ such that
$\Phi(\mu)=(\lambda,a,\epsilon)$ if and only if 
$\TI_2(\mu)=[v(\lambda,a,\epsilon)]$.
\end{lemma}
\begin{proof}
For a given flagpole partition $\mu$, there exists
$(\lambda,a,\epsilon)\in G$ with $\TI_2(\mu)=[v(\lambda,a,\epsilon)]$
by Theorem~\ref{thm:flag-init}. This triple is uniquely determined
by $\mu$ since no two Dyck vectors in Theorem~\ref{thm:tail2}(b) are 
equivalent.  So $\Phi$ is a well-defined function from $F$ into $G$.
To see $\Phi$ is bijective, fix $(\lambda,a,\epsilon)\in G$.
Then $[v(\lambda,a,\epsilon)]=\TI_2(\mu)$ for some partition $\mu$ by 
Theorem~\ref{thm:tail2}, and $\mu$ is a flagpole partition by
Theorem~\ref{thm:flag-init}. Thus $\Phi$ is surjective.
Since we can recover $\TI(\mu)$ and $\mu$ itself from $\TI_2(\mu)$, 
$\Phi$ is injective.  
\end{proof}

\begin{remark}\label{rem:flag-parts} 
Theorem~\ref{thm:NU2-chains} provides an explicit
formula for $\mu=\Phi^{-1}(\lambda,a,\epsilon)$. Apply the theorem
to the vector $v=v(\lambda,a,\epsilon)$ representing $\TI_2(\mu)$. 
This $v$ has the required form $v=0012^{n_0}12^{n_1}\cdots 12^{n_r}C$,
where $n_0=a-\epsilon$, $n_i$ is the number of $i$s in $\lambda$ for 
$1\leq i\leq r$, and $C=1^{\epsilon}$.  From~\eqref{eq:loop-end}, 
the reduced vector $0B_{\mu}$ for $\TI(\mu)$ is given explicitly as
\begin{equation}\label{eq:flag-parts}
  001^{p}1^{\epsilon}
  1^{2\lfloor (a-\epsilon)/2\rfloor}01^{(a-\epsilon)\bmod 2} 
  1^{2\lfloor n_1/2\rfloor}01^{n_1\bmod 2}\cdots
  1^{2\lfloor n_{r-1}/2\rfloor}01^{n_{r-1}\bmod 2}
  1^{2\lfloor n_r/2\rfloor}(01)^{n_r\bmod 2},
\end{equation}
where $p$ is the number of even integers in the list $a-\epsilon,n_1,
\ldots,n_r$.  We can then recover $\mu$ from $B_{\mu}$
via Definition~\ref{def:TI-tail}. 
In particular, the number of parts in $\mu$ equal to $1$
is the length of the first block of $1$s in~\eqref{eq:flag-parts},
namely $p+\epsilon+2\lfloor (a-\epsilon)/2\rfloor$,
which is at least $a-1\geq a_0(\lambda)-1$.
Informally, this shows that a flagpole partition $\mu$ must
end in many $1$s, so that the Ferrers diagram of $\mu$
looks like a flag flying on a pole.
\end{remark}

\begin{example}
Given $\lambda=\ptn{4433111}$, let us find $\mu=\Phi^{-1}(\lambda,10,0)$. 
Here $a=10$, $\epsilon=0$, $\TI_2(\mu)=[0012^{10}12221122122]$, 
$n_0=10$, $n_1=3$, $n_2=0$, $n_3=2$, $n_4=2$, $C=\emptyset$,
and $p=4$.  Using~\eqref{eq:flag-parts},
we get $0B_{\mu}=001^4(1^{10}0)(1101)(0)(110)(11)
=001^{14}01^20101^201^2$. So $\mu=\ptn{5^24^23^12^21^{14}}$.
When computing $\Phi^{-1}(\lambda,10,1)$, we get $n_0=9$, $C=1$, $p=3$,
$0B_{\mu}=001^31(1^801)(1101)(0)(110)(11)=001^{12}01^30101^201^2$,
and the answer is $\ptn{5^24^23^12^31^{12}}$.  
More generally, for any $a\geq a_0(\lambda)=9$, we see that
$\Phi^{-1}(\lambda,a,0)$ is $\ptn{5^24^23^12^21^{4+a}}$ when $a$ is even
 and is $\ptn{5^24^23^12^31^{2+a}}$ when $a$ is odd. Also
$\Phi^{-1}(\lambda,a,1)$ is $\ptn{5^24^23^12^21^{4+a}}$ when $a$ is odd
and is $\ptn{5^24^23^12^31^{2+a}}$ when $a$ is even.
In particular, for all $b\geq 13$, $\ptn{55443221^{b}}$
 and $\ptn{554432221^{b-2}}$ are flagpole partitions of flag type $\lambda$.
A similar pattern holds for other choices of $\lambda$.
\end{example}

Next, we use the bijection $\Phi$ to enumerate flagpole partitions.

\begin{theorem}\label{thm:count-flag}
The number of flagpole partitions of size $n$ 
is $\sum_{j=0}^{\lfloor (n-4)/2\rfloor} 2p(j)$, 
where $p(j)$ is the number of integer partitions of size $j$.
\end{theorem}
\begin{proof}
Suppose $\mu$ is a flagpole partition of size $n$
and $\Phi(\mu)=(\lambda,a,\epsilon)$.  Then $\TI_2(\mu)$ has deficit $n$ 
and is represented by $v(\lambda,a,\epsilon)$. 
By Remark~\ref{rem:flag-init}, $\defc(v(\lambda,a,\epsilon))
 =a+1+\lambda_1+\ell(\lambda)+|\lambda|$. Since $a\geq a_0(\lambda)$,
the smallest possible value of $\defc(v(\lambda,a,\epsilon))$
is $2|\lambda|+4$. Thus, $n\geq 2|\lambda|+4$ and $|\lambda|\leq (n-4)/2$.
By reversing this argument, we can construct each flagpole partition of $n$
by making the following choices. Pick an integer $j$ with $0\leq j\leq (n-4)/2$,
and pick $\lambda$ to be any of the $p(j)$ partitions of $j$. Pick the unique
integer $a\geq a_0(\lambda)$ such that $a+1+\lambda_1+\ell(\lambda)+|\lambda|
 =n$. Pick $\epsilon$ to be $0$ or $1$ (two choices). Finally, define
$\mu=\Phi^{-1}(\lambda,a,\epsilon)$.  The number of ways to make these
choices is $\sum_{0\leq j\leq (n-4)/2} 2p(j)$.
\end{proof}

\begin{remark}\label{rem:flag-asym}
Let $f(n)$ be the number of flagpole partitions of size $n$.
It is known~\cite[(5.26)]{odlyzko} that
$\sum_{j\leq n} p(j)=\Theta\left(n^{-1/2}\exp(\pi\sqrt{2n/3})\right)$.
Using this and Theorem~\ref{thm:count-flag}, we get 
$f(n)=\Theta\left(n^{-1/2}\exp(\pi\sqrt{n/3})\right)$.  
Hardy and Ramanujan~\cite{HR1918} proved that
$p(n)=\Theta\left(n^{-1}\exp(\pi\sqrt{2n/3})\right)$.
So $f(n)=\Theta\left(p(n)^{1/\sqrt{2}}n^{(\sqrt{2}-1)/2}\right)$.
\end{remark}

The following variation of the bijection $\Phi$ will help us construct
global chains indexed by flagpole partitions.

\begin{lemma}\label{lem:flag-bij2}
Let $F$ be the set of flagpole partitions, and
let $H$ be the set of triples $(\lambda,L,\eta)$,
where $\lambda$ is an integer partition, 
$L$ is an integer with $L\geq |\lambda|+6$, and $\eta$ is $0$ or $1$.
There is a bijection $\Psi:F\rightarrow H$ given by 
\begin{equation}\label{eq:psi}
 \Psi(\mu)=(\ftype(\mu),\mind(\TI_2(\mu)),\dinv(\TI_2(\mu))\bmod 2). 
\end{equation}
\end{lemma}
\begin{proof}
Given a flagpole partition $\mu\in F$, we know
$\TI_2(\mu)=[v(\lambda,a,\epsilon)]$ for a unique partition 
$\lambda=\ftype(\mu)$, $a\geq a_0(\lambda)$, and $\epsilon\in\{0,1\}$,
namely for $(\lambda,a,\epsilon)=\Phi(\mu)$.
Since $v=v(\lambda,a,\epsilon)$ is a reduced Dyck vector, 
$\mind(\TI_2(\mu))=\len(v)$. By Remark~\ref{rem:flag-init}
and the definition of $a_0(\lambda)$,
$\len(v)=a+3+\lambda_1+\ell(\lambda)\geq |\lambda|+6$.
Thus, $\Psi$ is a well-defined function mapping into the codomain $H$.

To see that $\Psi$ is invertible, consider $(\lambda,L,\eta)\in H$.
Define $a=L-3-\lambda_1-\ell(\lambda)$, and note $L\geq |\lambda|+6$
implies $a\geq a_0(\lambda)$. Since $\dinv(v(\lambda,a,1))=
\dinv(v(\lambda,a,0))+1$, there is a unique $\epsilon\in\{0,1\}$
with $\dinv(v(\lambda,a,\epsilon))=\eta$. Now let 
$\mu=\Phi^{-1}(\lambda,a,\epsilon)$ be the unique flagpole partition 
with $\TI_2(\mu)=[v(\lambda,a,\epsilon)]$. It is routine to check
that the map $(\lambda,L,\eta)\mapsto \mu$ defined in this paragraph
is the two-sided inverse of $\Psi$.  
\end{proof}

\begin{example}\label{ex:Psi}
Let us find $\Psi(\mu)$ for $\mu=\ptn{322111})$. 
From~\eqref{eq:TI2-ex}, $\TI_2(\mu)=[001221222]=[v(\ptn{111},2,0)]$.
Since $\mind(\TI_2(\mu))=9$ and $\dinv(\TI_2(\mu))=14$ is even,
$\Psi(\mu)=(\ptn{111},9,0)$.

Next we compute $\Psi^{-1}(\ptn{0},L,0)$ for each $L\geq 6$.
For $a\geq a_0(\ptn{0})=3$, we have 
$v(\ptn{0},a,\epsilon)=0012^{a-\epsilon}1^{\epsilon}$.
By~\eqref{eq:loop-end}, $\Phi^{-1}(\ptn{0},a,\epsilon)$ is $\ptn{21^{a-1}}$ if
$a-\epsilon$ is odd and $\ptn{1^{a+1}}$ if $a-\epsilon$ is even.
Since we need $v(\ptn{0},a,\epsilon)$ to have length $L$,
we take $a=L-3$.  Thus, $\Psi^{-1}(\ptn{0},L,0)$ and $\Psi^{-1}(\ptn{0},L,1)$ 
  are $\ptn{1^{L-2}}$ and $\ptn{21^{L-4}}$ in some order.
Using~\eqref{eq:TI2-stats} to compute $\dinv(v(\ptn{0},a,\epsilon))$,
one readily checks that $\Psi^{-1}(\ptn{0},L,0)$ is $\ptn{1^{L-2}}$
when $L\bmod 4\in\{0,1\}$ and is $\ptn{21^{L-4}}$ when $L\bmod 4\in\{2,3\}$.  
\end{example}

\section{Review of the Local Chain Method}
\label{sec:review-local}

This section reviews the local chain method from~\cite{HLLL20},
which gives a convenient way to prove that two proposed global chains
$\C_{\mu}$ and $\C_{\mu^*}$ satisfy the \emph{opposite property}
$\Cat_{n,\mu^*}(t,q)=\Cat_{n,\mu}(q,t)$ for all $n>0$.
(Recall the definition of $\Cat_{n,\mu}$ from~\eqref{eq:Cat_n,mu}.)
The idea is to distill the $\mind$-profiles of $\C_{\mu}$ and
$\C_{\mu^*}$ into three finite lists of integers called the
$amh$-vectors for $\C_{\mu}$ and $\C_{\mu^*}$. The infinitely many
conditions in the opposite property can be verified through a simple
finite computation on the $amh$-vectors. First we define more carefully
what a ``proposed global chain'' must look like.

\begin{definition}\label{def:basic-str}
Suppose we are given specific partitions $\mu$ and $\mu^*$ of the same 
size $k>0$ (with $\mu^*=\mu$ allowed) and two sequences of partitions
$\C_{\mu}$ and $\C_{\mu^*}$. We say that $\C_{\mu}$ and $\C_{\mu^*}$
have \emph{basic required structure} iff the following
conditions hold:
\\ (a)~$\C_{\mu}$ is a sequence $(c_{\mu}(i):i\geq i_0(\mu))$
where $c_{\mu}(i)$ is a partition with deficit $k(=|\mu|)$ and dinv $i$.
\\ (b)~$\C_{\mu^*}$ is a sequence $(c_{\mu^*}(i):i\geq i_0(\mu^*))$
where $c_{\mu^*}(i)$ is a partition with deficit $k(=|\mu^*|)$ and dinv $i$.
\\ (c) $\C_{\mu}$ starts at dinv value $i_0(\mu)=\ell(\mu^*)$,
 and $\C_{\mu^*}$ starts at dinv value $i_0(\mu^*)=\ell(\mu)$.
\\ (d) $\C_{\mu}$ ends with the sequence $\tail(\mu)$, 
 and $\C_{\mu^*}$ ends with the sequence $\tail(\mu^*)$.
\\ (e) If $\mu\neq\mu^*$, then $\C_{\mu}$ and $\C_{\mu^*}$ are disjoint.
 If $\mu=\mu^*$, then $\C_{\mu}=\C_{\mu^*}$.
\end{definition}

Assume $\C_{\mu}$ and $\C_{\mu^*}$ have basic required structure.
We now review the definition of the $amh$-vectors for these chains.
Recall that the $\mind$-profile of $\C_{\mu}$ is the sequence of integers
$(p_i:i\geq i_0(\mu))$ where $p_i=\mind(c_{\mu}(i))$ for each $i$.
Define the \emph{descent set} $\Des(\mu)$ to be the set consisting of 
$i_0(\mu)$ and all $i>i_0(\mu)$ with $p_{i-1}>p_i$.
Since the $\mind$-values in $\tail(\mu)$ form
a weakly increasing sequence (Theorem~\ref{thm:tail-profile}),
$\Des(\mu)$ is a finite set. The \emph{$a$-vector for $\mu$}
is the list $(a_1,a_2,\ldots,a_N)$ of members of $\Des(\mu)$
written in increasing order.  The \emph{$h$-vector for $\mu$}
is $(h_1,h_2,\ldots,h_N)$, where $h_i=p_{a_i}$ for $1\leq i\leq N$.
The \emph{$m$-vector for $\mu$} is $(m_1,m_2,\ldots,m_N)$, 
where $m_i\geq 0$ is the largest integer such that 
$p_{a_i}=p_{a_i+1}=\cdots=p_{a_i+m_i}$. This definition
means that the $i$th ascending run of the $\mind$-profile of $\C_{\mu}$
starts at dinv index $a_i$ with $m_i+1$ copies of $h_i$ followed by
a different value. 
We similarly define the $a$-vector $(a_1^*,\ldots,a_{N^*}^*)$ for $\mu^*$,
the $h$-vector $(h_1^*,\ldots,h_{N^*}^*)$ for $\mu^*$, and
the $m$-vector $(m_1^*,\ldots,m_{N^*}^*)$ for $\mu^*$.
The next definition lists the conditions needed to decompose
$\C_{\mu}$ into local chains.

\begin{definition}\label{def:local-str}
Assume $\C_{\mu}$ and $\C_{\mu^*}$ have basic required structure.
We say that chain $\C_{\mu}$ has \emph{local required structure} iff
the following conditions hold:
\\ (a) $a_N=\dinv(\TI(\mu))$.
\\ (b) For $1\leq i<N$, the $i$th ascending run of the $\mind$-profile 
  of $\C_{\mu}$ is some prefix of the \emph{staircase sequence}
  $h_i^{m_i+1}(h_i+1)^{h_i}(h_i+2)^{h_i+1}(h_i+3)^{h_i+2}\cdots$
  that includes at least one copy of $h_i+1$.
\end{definition}

We use analogous conditions to define the local required structure for 
$\C_{\mu^*}$.  Condition~(a) means that the last ascending run of the 
$\mind$-profile for $\C_{\mu}$ corresponds to $\tail(\mu)$.  
In other words, $\TI(\mu)$ must have a smaller $\mind$ value than the 
preceding object (if any) in $\C_{\mu}$.
By Theorem~\ref{thm:tail-profile} and~\ref{def:basic-str}(d),
it is guaranteed that $m_N=0$ and condition (b) holds for $i=N$
 (with the prefix being the entire infinite staircase sequence).

\begin{definition}\label{def:amh-hyp}
Assume $\C_{\mu}$ and $\C_{\mu^*}$ have basic and local required structure.
We say $\C_{\mu}$ and $\C_{\mu^*}$ satisfy the \emph{amh-hypotheses} iff
the following conditions hold:
\\ (a) The $h$-vector for $\C_{\mu^*}$ is the reverse
of the $h$-vector for $\C_{\mu}$ (forcing $N^*=N$).
\\ (b) The $m$-vector for $\C_{\mu^*}$ is the reverse of 
the $m$-vector for $\C_{\mu}$.
\\ (c) For $1\leq i\leq N$, $a_i+m_i+k+a^*_{N+1-i}=\binom{h_i}{2}$,
 where $k=|\mu|=|\mu^*|$.
\end{definition}


\begin{theorem}\cite[Thm. 3.10 and Sec. 4]{HLLL20}\label{thm:local-opp}
Assume $\C_{\mu}$ and $\C_{\mu^*}$ have basic and local required structure
and satisfy the amh-hypotheses. Then for all $n>0$,
$\Cat_{n,\mu^*}(t,q)=\Cat_{n,\mu}(q,t)$.
\end{theorem}

All chains we have constructed previously 
 (see~\cite[Appendix]{HLLL20} and~\cite{HLLL20ext})
satisfy some additional conditions that we need for the recursive construction
in~\S\ref{sec:make-flag-chains}. We list these conditions next.

\begin{definition}\label{def:extra-str}
Assume $\C_{\mu}$ and $\C_{\mu^*}$ have basic required structure.
We say $\C_{\mu}$ has \emph{extra required structure} iff the 
following conditions hold:
\\ (a) The $h$-vector $(h_1,\ldots,h_N)$ is a weakly decreasing
 sequence followed by a weakly increasing sequence.
\\ (b) For $i<N$, all values in the $i$th ascending run 
of the $\mind$-profile for $\C_{\mu}$ are at most $1+\max(h_i,h_{i+1})$.
\\ (c) For $i\geq i_0(\mu)$, $\mind(c_{\mu}(i))<\mind(c_{\mu}(i+1))$
 iff the reduced Dyck vector for $c_{\mu}(i)$ is $0$ or begins with $00$.
\\ (d) $\C_{\mu}$ contains $\tail_2(\mu)$ (not just $\tail(\mu)$).
\end{definition}
We make an analogous definition for $\C_{\mu^*}$.
Condition~(c) is guaranteed for objects $c_{\mu}(i)$
in $\tail(\mu)$ by Theorem~\ref{thm:BDV-tail} (note $\mu\neq\ptn{0}$ here).
If $\TI_2(\mu)=[v]$ where $v$ has the form studied in
Theorem~\ref{thm:NU2-chains}, then condition~(c) is also guaranteed for objects
$c_{\mu}(i)$ in $\tail_2(\mu)$ by part~(e) of that theorem.

\begin{example}\label{ex:make1^5}
Let $\mu=\ptn{1^5}=\mu^*$, and define $\C_{\mu}=\C_{\mu^*}$
as the union of the $\NU_1$-segments starting at 
$[0012332]$, $[0012222]=\TI_2(\mu)$, $[0012211]$, and $[0011111]=\TI(\mu)$. 
Each Dyck class listed here is a $\NU_1$-initial object
with deficit $5=|\mu|$. By computing the $\NU_1$-segments,
we see that they assemble to give a chain of partitions
with consecutive dinv values starting at $5=\ell(\mu^*)=\ell(\mu)$, namely
\[ \C_{\ptn{1^5}}:\quad [0012332],\ [01234430],\ [0012222],\ [01233330],
 \ [0012211],\ [01233220],\ \mbox{ followed by }\tail(\ptn{1^5}). \]
Thus $\C_{\mu}$ has basic required structure.
$\C_{\mu}$ has $\mind$-profile $78787878^79^810^9\cdots$,
$a$-vector $(5,7,9,11)$, $m$-vector $(0,0,0,0)$, and $h$-vector $(7,7,7,7)$.  
We see that $\C_{\mu}$ has local required structure by inspection
of the $\mind$-profile (in particular, \ref{def:local-str}(a) holds 
since $a_N=11=\dinv(\TI(\mu))$). The first two $amh$-hypotheses are
true since the $m$-vector and $h$-vector are palindromes (equal to their 
own reversals). We check $amh$-hypothesis (c) for $i=1,2,3,4$ by computing
\[ 5+0+5+11=7+0+5+9=9+0+5+7=11+0+5+5=21=\binom{7}{2}. \]
By Theorem~\ref{thm:local-opp}, $\Cat_{n,\mu}(t,q)=\Cat_{n,\mu}(q,t)$
for all $n>0$. It is also routine to check that $\C_{\mu}$ has the
extra required structure.
\end{example}

\begin{example}
In Section~\ref{sec:make-flag-chains}, we perform an elaborate construction
to build proposed chains for $\mu=\ptn{531^4}$ and $\mu^*=\ptn{3321^4}$,
which are partitions of size $k=12$. As described in detail later,
$\C_{\mu}$ contains $\tail_2(\mu)$ and has $\mind$-profile given by
\[ 11,12,(10,11)^7,10,11^{10},12,11,12^{11},13,12,13^{12},14^{13},15^{14},
 \ldots, \]
which is the concatenation of~\eqref{eq:prof-A1},~\eqref{eq:prof-B1},
 and~\eqref{eq:prof-C1}.
$\C_{\mu^*}$ contains $\tail_2(\mu^*)$ and has $\mind$-profile given by
\[ 12,13,11,12,(10,11)^7,10,11^{10},12,11,12^{11},13^{12},14^{13},\ldots, \]
which is the concatenation of~\eqref{eq:prof-A2},~\eqref{eq:prof-B2}, 
and~\eqref{eq:prof-C2}.
$\C_{\mu}$ starts at dinv index $7=\ell(\mu^*)$,
and $\C_{\mu^*}$ starts at dinv index $6=\ell(\mu)$. 
The specific $\NU_1$-initial objects used to make $\C_{\mu}$ 
are different from the $\NU_1$-initial objects used to make $\C_{\mu^*}$,
so $\C_{\mu}$ and $\C_{\mu^*}$ are disjoint chains.
All the $\NU_1$-initial objects used have deficit $k=12$.
Thus, $\C_{\mu}$ and $\C_{\mu^*}$ have basic required structure and
extra structure condition~(d).

The $m$-vectors for $\C_{\mu}$ and $\C_{\mu^*}$ are identically $0$.
The $a$-vector and $h$-vector for $\C_{\mu}$ are
\[ a=(7,9,11,13,15,17,19,21,23,35,48), \quad
   h=(11,10,10,10,10,10,10,10,10,11,12). \]
The $a$-vector and $h$-vector for $\C_{\mu^*}$ are
\[ a^*=(6,8,10,12,14,16,18,20,22,24,36), \quad
   h^*=(12,11,10,10,10,10,10,10,10,10,11). \]
It is now routine to verify that $\C_{\mu}$ and $\C_{\mu^*}$
have local required structure and satisfy the $amh$-hypotheses. 
For instance, we check $amh$-hypothesis~(c) for $i=1,2,3$ and $i=11$ as follows:
\[ 7+0+12+36=55=\binom{11}{2};\quad
   9+0+12+24=45=\binom{10}{2}; \]
\[ 11+0+12+22=45=\binom{10}{2};\quad
   48+0+12+6=66=\binom{12}{2}. \]
By Theorem~\ref{thm:local-opp}, $\Cat_{n,\mu^*}(t,q)=\Cat_{n,\mu}(q,t)$
for all $n>0$. We can also verify the extra required structure for these
chains, using the specific objects constructed below to 
verify~\ref{def:extra-str}(c) for objects $c_{\mu}(i)$ preceding
the second-order tails. All of these verifications are special cases
of general results to be proved later.
\end{example}

\section{Constructing Global Chains Indexed by Certain Flagpole Partitions}
\label{sec:make-flag-chains}

\subsection{Statement of Results}
\label{subsec:state-result}

Our ultimate goal (not yet achieved in this paper)
is to define $\mu^*$ and chains $\C_{\mu}$ and $\C_{\mu^*}$ 
for \emph{every} integer partition $\mu$ and to prove that all 
structural conditions and $amh$-hypotheses in~\S\ref{sec:review-local} are true.
We hope to reach this goal recursively, constructing chains indexed
by partitions $\mu$ of a given size $k$ by referring to previously-built 
chains $\C_{\lambda}$ for various partitions $\lambda$ of size less than $k$.
This section achieves a limited version of this construction,
building $\C_{\mu}$ and $\C_{\mu^*}$ for all \emph{flagpole} partitions
$\mu$ of one particular size $k$, assuming that all chains $\C_{\lambda}$
indexed by \emph{all} partitions $\lambda$ of size less than $k$ are
already available. In fact, we prove a much sharper conditional
result by keeping track of exactly which smaller chains $\C_{\lambda}$
are needed to build $\C_{\mu}$ and $\C_{\mu^*}$, for each specific
flagpole partition $\mu$. To state our findings, we first define precisely
what we mean by a ``collection of previously-built chains.''

\begin{definition}\label{def:chain-coll}
A \emph{chain collection} is a triple $\CC=(\mcP,I,\C)$ satisfying
these conditions:
\\ (a)~$\mcP$ is a collection of integer partitions containing
      (at a minimum) all partitions of size at most $5$.
\\ (b)~$I:\mcP\rightarrow\mcP$ is a size-preserving involution 
 with domain $\mcP$, written $I(\lambda)=\lambda^*$ for $\lambda\in\mcP$.
\\ (c)~$\C$ is a function with domain $\mcP$ that maps each
 $\lambda$ in $\mcP$ to a sequence of partitions 
 $\C_{\lambda}=(c_{\lambda}(i):i\geq i_0(\lambda))$.
\\ (d)~For each $\lambda\in\mcP$, $\C_{\lambda}$ and $\C_{\lambda^*}$
 have basic required structure, local required structure, extra required
 structure, and satisfy the $amh$-hypotheses 
 (as defined in~\S\ref{sec:review-local}). 
\\ (e)~The chains $\C_{\lambda}$ (for $\lambda\in\mcP$) are pairwise
 disjoint.
\end{definition}

For example, our prior work~\cite{HLLL20,HLLL20ext} defines
$I(\lambda)=\lambda^*$ and constructs chains $\C_{\lambda}$ and 
$\C_{\lambda^*}$ (with all required properties) for $\lambda$ 
in the set $\mcP$ of all integer partitions of size at most $11$.
Thus we have a chain collection $(\mcP,I,\C)$ for this choice of $\mcP$,
which can be taken as a ``base case'' for the entire recursive construction.
The minimalist base case takes $\mcP$ to be the collection of all partitions
$\lambda$ of size at most $5$.  The chains for these $\lambda$
 are not hard to construct and verify, even without a computer.
 Example~\ref{ex:make1^5} illustrates this process for $\lambda=\ptn{1^5}$.

Let $\CC=(\mcP,I,\C)$ be a fixed chain collection. We can use this
collection to define $\mu^*$, $\C_{\mu}$, and $\C_{\mu^*}$ 
(with all required properties) for certain flagpole partitions $\mu$.
The construction for a specific $\mu$ only requires us to know
$\rho^*$, $\C_{\rho}$, and $\C_{\rho^*}$ for a specific 
list of smaller partitions $\rho$ (depending on $\mu$).
One of these partitions is the flag type $\lambda$ of $\mu$ 
(Definition~\ref{def:flag-type}), but there could be others.
Before we can describe these others, we must discuss how to compute $\mu^*$.

\begin{definition}\label{def:extend-I}
Let $\CC=(\mcP,I,\C)$ be a chain collection, 
and let $\mu$ be a flagpole partition.
We say \emph{$I$ extends to $\mu$} iff $|\mu|>|\rho|$ for all $\rho\in\mcP$,
and $\lambda=\ftype(\mu)$ belongs to $\mcP$.
In this case, define $\mu^*$ as follows.
Let the reduced Dyck vector for $\TI_2(\mu)$ have length $L$ and area $A$. 
Define $\mu^*=\Psi^{-1}(\lambda^*,L,A\bmod 2)$,
where $\Psi$ is the bijection~\eqref{eq:psi}.
\end{definition}

In~\S\ref{subsec:flag-invo}, we show this definition does give
a well-defined flagpole partition $\mu^*$ of the same size as $\mu$
with $\ftype(\mu^*)=\ftype(\mu)^*$ and $\mu^{**}=\mu$.

\begin{definition}\label{def:needed}
Let $\CC=(\mcP,I,\C)$ be a chain collection, and
let $\mu$ be a flagpole partition such that $I$ extends to $\mu$
(so $\mu^*$ is defined).
Let $\TI_2(\mu)=[V]$ and $\TI_2(\mu^*)=[V^*]$ where $V$ and $V^*$
are reduced. Recall the vectors $S_j(V)$ and $S_j(V^*)$
from Definition~\ref{def:Sj}, which (for $j>0$) are reduced TDVs of the
form $S_j(V)=0E_j1$ and $S_j(V^*)=0E_j^*1$ 
(Theorem~\ref{thm:NU2-chains}(e)).
The \emph{needed partitions} for $\mu$ are:
(a)~the flag type $\lambda$ of $\mu$; and
(b)~each partition $\rho$ such that $\tail(\rho)$ contains
 $[E_j]$ or $[E_j^*]$ for some $j>0$.
\end{definition}

Since $\lambda\in\mcP$, all objects mentioned in this definition
are explicitly computable. In particular, because each $E_j$ and $E_j^*$ 
is a TDV, $[E_j]$ and $[E_j^*]$ belong to some $\tail_2(\rho)$
(Theorem~\ref{thm:tail2}(a)),
and we know these tails for every partition $\rho$. In fact, 
Remark~\ref{rem:E-in-tail1} proves that each $[E_j]$ and $[E_j^*]$
is in some $\tail(\rho)$, so that $\rho$ itself is easily found
from Theorem~\ref{thm:BDV-tail}, as in Example~\ref{ex:plateau}.

We can now state the main results of this section.

\begin{theorem}\label{thm:extend-mu}
Let $\CC=(\mcP,I,\C)$ be a chain collection, and
let $\mu$ be a flagpole partition such that $I$ extends to $\mu$.
Assume all needed partitions for $\mu$ belong to $\mcP$.
Then we can explicitly construct chains $\C_{\mu}$ and $\C_{\mu^*}$
having basic required structure, local required structure, extra 
required structure, and satisfying the $amh$-hypotheses.
So, $\Cat_{n,\mu^*}(t,q)=\Cat_{n,\mu}(q,t)$ for all $n>0$.
\end{theorem}

By applying the construction to flagpole partitions in increasing
order of size, we can pass from a given initial chain collection 
$\CC$ to a larger one that, intuitively, consists of all chains 
that can be built from $\CC$ using the methods described here.
The next theorem makes this precise.

\begin{theorem}\label{thm:extend-all}
Let $\CC=(\mcP,I,\C)$ be a chain collection such
that $k_0=\max_{\rho\in\mcP} |\rho|$ is finite.
Recursively define $\CC^k=(\mcP^k,I^k,\C^k)$ for all $k\geq k_0$, as follows.
Let $\CC^{k_0}=\CC$. Fix $k>k_0$ and assume $\CC_{k-1}$ has been defined.
Enlarge $\mcP^{k-1}$ to $\mcP^k$ by adding all $\mu$ satisfying:
(i)~$\mu$ is a flagpole partition of size $k$;
(ii)~$I^{k-1}$ extends to $\mu$; and
(iii)~all partitions needed for $\mu$ are in $\mcP^{k-1}$. 
Extend $I^{k-1}$ to $I^k$ and $\C^{k-1}$ to $\C^k$
by using the construction in Theorem~\ref{thm:extend-mu} to define
$\mu^*$, $\C_{\mu}$, and $\C_{\mu^*}$ for each newly added $\mu$.
Let $\CC'=(\mcP',I',\C')$ be the union of the increasing sequence $(\CC^k)$.
Then every $\CC^k$ and $\CC'$ is a chain collection.
\end{theorem}

Section~\ref{sec:gen-flagpole} extends these theorems 
to generalized flagpole partitions.  

\subsection{Outline of Construction}
\label{subsec:outline}

In the rest of Section~\ref{sec:make-flag-chains}, 
we give the constructive proof of Theorem~\ref{thm:extend-mu}.
Some technical proofs are delayed until Section~\ref{sec:proofs}.
The construction is quite intricate, so we illustrate each step
with a running example where $\mu=\ptn{531^4}$.
Here is an outline of the main ingredients in the construction.
Throughout the rest of Section~\ref{sec:make-flag-chains},
we fix a chain collection $\CC=(\mcP,I,\C)$ and a flagpole
partition $\mu$ of size $k$ such that $I$ extends to $\mu$. 
Thus, $|\rho|<k$ for all $\rho\in\mcP$,
and the flag type $\lambda$ of $\mu$ is in $\mcP$.
Because $\lambda^*$ is known, we can compute $\mu^*$ 
(Definition~\ref{def:extend-I}); 
Section~\ref{subsec:flag-invo} supplies the required details.
At this point, we can explicitly compute the needed partitions 
for $\mu$ (Definition~\ref{def:needed}).
If they all belong to $\mcP$, we may proceed with the following construction.

The chain $\C_{\mu}$ consists of three parts, 
called the \emph{antipodal part}, the \emph{bridge part},
and the \emph{tail part}. We construct each part of the chain by
building specific $\NU_1$-initial objects that generate $\NU_1$-segments
comprising the chain.  The tail part of $\C_{\mu}$ is $\tail_2(\mu)$, 
which we already built in~\S\ref{subsec:TI2}.
The bridge part of $\C_{\mu}$ (introduced in~\S\ref{subsec:bridge})
consists of two-element $\NU_1$-segments
starting at partitions $[M_i(\mu)]$ made by adding a new leftmost 
column to particular partitions in the known chain $\C_{\lambda}$, 
where $\lambda=\ftype(\mu)$. The tail part and bridge part
of $\C_{\mu^*}$ are defined similarly, using $\lambda^*=\ftype(\mu^*)$.

The antipodal part of each chain is the trickiest piece to build. 
We must first identify the $\NU_1$-initial objects in $\tail_2(\mu)$
and $\tail_2(\mu^*)$ (Theorem~\ref{thm:NU2-chains} and 
\S\ref{subsec:init-tail2}), which have reduced Dyck vectors called $S_j(V)$  
and $S_j(V^*)$, respectively.
In~\S\ref{subsec:Ant}, we introduce the \emph{antipode map} $\Ant$;
this map interchanges area and dinv but only acts on a restricted class
of Dyck vectors. Applying $\Ant$ to the vectors $S_j(V^*)$ from the tail 
part of $\C_{\mu^*}$ produces $\NU_1$-initial objects $[A_j^*]$ that 
generate the antipodal part of $\C_{\mu}$ (see~\S\ref{subsec:antip}).
Similarly, the Dyck classes $[\Ant(S_j(V))]$ generate the
antipodal part of $\C_{\mu^*}$. Figure~\ref{fig:antip} may help
visualize the overall construction.

\begin{figure}[h]
\begin{center}
\epsfig{file=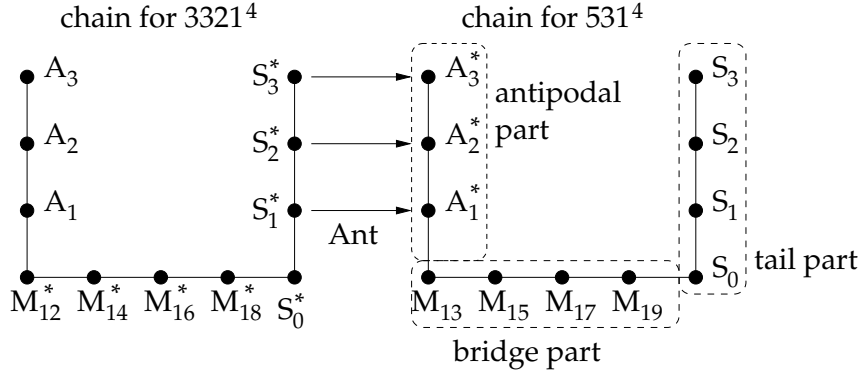,scale=0.7}
\end{center}
\caption{Structure of the chains $\C_{\mu}$ (right side) 
 and $\C_{\mu^*}$ (left side).}
\label{fig:antip}
\end{figure}
 
To finish, we must assemble all the $\NU_1$-segments and compute the
$\mind$-profiles and $amh$-vectors for the new chains $\C_{\mu}$ and
$\C_{\mu^*}$. Then we verify all the required structural properties
and $amh$-hypotheses from~\S\ref{sec:review-local}. The opposite
property stated in Theorem~\ref{thm:extend-mu} then follows from
Theorem~\ref{thm:local-opp}.

Our running example $\mu=\ptn{531^4}$ has size $k=12$, flag type
$\lambda=\ptn{31}$, $\lambda^*=\ptn{22}$, and $\mu^*=\ptn{3321^4}$
(see~\S\ref{subsec:flag-invo} for details). The needed
partitions for $\mu$ are the flag type $\lambda=\ptn{31}$
along with the following partitions (computed in Example~\ref{ex:Ant}): 
$\ptn{22}$, $\ptn{21}$, $\ptn{3}$, and $\ptn{2}$.
All needed partitions $\rho$ have size at most $4$ and are therefore in $\mcP$.
For later reference, we list $\rho$, $\rho^*$, $\C_{\rho}$, and
$\C_{\rho^*}$ for each needed $\rho$ here:

\begin{equation}\label{eq:small-chains}
\begin{array}{lll}
\rho=\ptn{31}: & \C_{\rho}=\NU_1^*(\ptn{2211})\cup\NU_1^*(\ptn{44311}),
 & \rho^*=\ptn{22}; \\
\rho=\ptn{22}: & \C_{\rho}=\NU_1^*(\ptn{21^4})\cup\NU_1^*(\ptn{3221}),
 & \rho^*=\ptn{31}; \\
\rho=\ptn{21}: & \C_{\rho}=\NU_1^*(\ptn{3111})\cup\NU_1^*(\ptn{3311}),
 & \rho^*=\ptn{111}; \\
\rho=\ptn{111}: & \C_{\rho}=\NU_1^*(\ptn{2111})\cup\NU_1^*(\ptn{3211}),
 & \rho^*=\ptn{21}; \\
\rho=\ptn{3}: & \C_{\rho}=\NU_1^*(\ptn{1111})\cup\NU_1^*(\ptn{222})
  \cup\NU_1^*(\ptn{3321}), & \rho^*=\ptn{3}; \\
\rho=\ptn{2}: & \C_{\rho}=\NU_1^*(\ptn{111})\cup\NU_1^*(\ptn{221}),
 & \rho^*=\ptn{2}.
\end{array} 
\end{equation}

We explain later how information from these existing chains is used
to help construct the new chains $\C_{\mu}$ and $\C_{\mu^*}$.

\subsection{Defining $\mu^*$ for Flagpole Partitions}
\label{subsec:flag-invo}

For the discussion of $\mu^*$ in this subsection, 
we assume $\CC=(\mcP,I,\C)$ is a fixed chain collection; $\mu$ is
a fixed flagpole partition of size $k$, where $k>|\rho|$ for all $\rho\in\mcP$
(hence $k\geq 6$ by~\ref{def:chain-coll}(a)); and the flag type
of $\mu$ is in $\mcP$.  We introduce notation that is used 
throughout Sections~\ref{sec:make-flag-chains} and~\ref{sec:proofs}.
Let $\lambda=\ftype(\mu)$, and
let $V=v(\lambda,a,\epsilon)$ be the reduced Dyck vector for $\TI_2(\mu)$.
Let $L=\len(V)=\mind(\TI_2(\mu))$, $D=\dinv(V)$, and $A=\area(V)$.
Recall (\S\ref{subsec:rep-flagpole}) that $\Psi(\mu)=(\lambda,L,D\bmod 2)$.
Note $\defc(V)=\defc(\TI_2(\mu))=\defc(\TI(\mu))=|\mu|=k$.
Remark~\ref{rem:flag-init} gives
\[ \defc(V)=k=|\lambda|+L-2, \] 
which is frequently used later.  Our proposed definition of $\mu^*$ is 
$\mu^*=\Psi^{-1}(\lambda^*,L,A\bmod 2)$. Lemma~\ref{lem:check-mu*} shows 
$\mu^*$ is well-defined. Let $V^*$ be the reduced Dyck vector for 
$\TI_2(\mu^*)$.  Then $\len(V^*)=\mind(\TI_2(\mu^*))=L$, and we let 
$D^*=\dinv(V^*)$ and $A^*=\area(V^*)$.

For our running example $\mu=\ptn{531^4}$, we compute
$V=0012212112=v(\lambda,2,0)$ where $\lambda=\ptn{31}$, 
$L=10$, $D=21$, $A=12$, and $\Psi(\mu)=(\ptn{31},10,1)$. 
From~\eqref{eq:small-chains}, we know $\lambda^*=\ptn{22}$.
(This is the only fact about $\lambda$ needed to compute $\mu^*$.)
Since $A\bmod 2=0$, $\mu^*=\Psi^{-1}(\ptn{22},10,0)=\ptn{3321^4}$.
Then $V^*=0012221122=v(\lambda^*,3,0)$, $D^*=20$, and $A^*=13$.

\begin{example}
We know $\ptn{0}^*=\ptn{0}$.  By Example~\ref{ex:Psi}, 
$\ptn{1^k}^*$ is either $\ptn{1^k}$ or $\ptn{21^{k-2}}$.
\end{example}

\begin{lemma}\label{lem:check-mu*}
Assume $\mu$ satisfies the hypothesis in Definition~\ref{def:extend-I}. 
Then $\mu^*$ is a well-defined flagpole partition with $|\mu^*|=|\mu|$
and $\mu^{**}=\mu$. Moreover, $D^*\equiv A\pmod{2}$ and $A^*\equiv D\pmod{2}$. 
\end{lemma}
\begin{proof}
From Lemma~\ref{lem:flag-bij2},
$\Psi(\mu)=(\lambda,L,D\bmod 2)$ where $L\geq |\lambda|+6\geq 6$.
Since $\lambda\in\mcP$, we know that
$\lambda^*$ is already defined, $\lambda^*\in\mcP$, $|\lambda^*|=|\lambda|$,
and $L\geq |\lambda^*|+6$. Thus, $(\lambda^*,L,A\bmod 2)$ does belong
to the codomain of the bijection $\Psi$, and so $\mu^*$ is a well-defined
flagpole partition. Since $V^*$ and $V$ both have length $L$, 
\[ |\mu^*|=\defc(V^*)=|\lambda^*|+L-2=|\lambda|+L-2=\defc(V)=|\mu|=k. \]
Next we check $D^*\equiv A\pmod 2$, $A^*\equiv D\pmod 2$, and $\mu^{**}=\mu$.
By definition of $\mu^*$, $\Psi(\mu^*)=(\lambda^*,L,A\bmod 2)$,
 so $D^*\equiv A\pmod 2$ by definition of $\Psi$. 
Now, since $\defc(V^*)=\defc(V)=k$ and $\dinv+\area+\defc=\binom{\len}{2}$, 
\begin{equation}\label{eq:DA-sum}
 D+A=\binom{L}{2}-k=D^*+A^*. 
\end{equation}
Because $D^*\equiv A\pmod{2}$, we also have $A^*\equiv D\pmod{2}$.
Finally, the definition of the involution 
gives $\Psi(\mu^{**})=(\lambda^{**},L,A^*\bmod 2)$.
Since $\lambda^{**}=\lambda$ (by~\ref{def:chain-coll}(b)) and 
$A^*\equiv D\pmod{2}$, $\Psi(\mu^{**})=(\lambda,L,D\bmod 2)=\Psi(\mu)$. 
Since $\Psi$ is one-to-one, $\mu^{**}=\mu$ follows.
\end{proof}

In the situation of the lemma, we really can extend the involution $I$
by setting $I(\mu)=\mu^*$ and $I(\mu^*)=\mu$ without conflicting with
any previously defined values $I(\rho)$, since $\mu$ and $\mu^*$ have
size strictly larger than $\rho$ and $\rho^*$ by hypothesis.

\begin{lemma}\label{lem:D>=A*}
Assume $I$ extends to a flagpole partition $\mu$ (necessarily of size
at least $6$). Then $D\geq A^*$.
\end{lemma}

This lemma is proved in~\S\ref{subsec:exceptional}.

\subsection{The Bridge Part of $\C_{\mu}$}
\label{subsec:bridge}

The \emph{bridge part} of $\C_{\mu}$ consists of two-element
$\NU_1$-segments $[M_i(\mu)]$, $\NU([M_i(\mu)])$ with
$\dinv(M_i(\mu))=i$, for all $i\in\{A^*,A^*+2,A^*+4,\ldots,D-4,D-2\}$.
The bridge part is empty if $D=A^*$. To define $M_i(\mu)$, we require
the following lemma, proved in~\S\ref{subsec:prove-bridge}.

\begin{lemma}\label{lem:pre-bridge}
Assume $I$ extends to $\mu$, so $\lambda=\ftype(\mu)$ is in $\mcP$.
For each $i\in\{A^*,A^*+2,\ldots,D-4,D-2\}$, there exists a unique
object $\gamma=c_{\lambda}(i-1)$ in the known chain $\C_{\lambda}$
such that $\dinv(\gamma)=i-1$ and $\defc(\gamma)=|\lambda|$. Moreover,
$\mind(\gamma)\leq L-2$, and $z=\qdvmap_{L-2}(\gamma)$ starts with $01$
and contains a $2$. 
\end{lemma}

For $i$ in $\{A^*,A^*+2,\ldots,D-2\}$, define $M_i(\mu)$ as follows.
Take $\gamma=c_{\lambda}(i-1)$ and $z=\qdvmap_{L-2}(\gamma)$ as in the lemma,
and let $M_i(\mu)=00z^+$. Visually, the Ferrers diagram for the 
partition $[M_i(\mu)]$ is obtained from the diagram for $\gamma$ by adding 
a new leftmost column containing $L-1$ boxes.
See Figure \ref{fig:Ferrers} for an example.

\begin{figure}[h]
\begin{tikzpicture}[scale=0.4, line width=1pt]
\fill[gray!50] (0,0) rectangle (6,-1);
\fill[gray!50](0,0) rectangle (4,-3);
\fill[gray!50] (0,0) rectangle (1,-5);
  \draw[thin] (0,0) grid (7,-1);
  \draw[thin] (0,-1) grid (6,-2);
  \draw[thin] (0,-2) grid (5,-3);
  \draw[thin] (0,-3) grid (4,-4);
  \draw[thin] (0,-4) grid (3,-5);
  \draw[thin] (0,-5) grid (2,-6);
  \draw[thin] (0,-6) grid (1,-7);
   \draw[thin] (0,0)--(8,0)--(0,-8)--(0,0);
\begin{scope}[shift={(15,0)}]
\fill[gray!50] (0,0) rectangle (7,-1);
\fill[gray!50](0,0) rectangle (5,-3);
\fill[gray!50] (0,0) rectangle (2,-5);
\fill[gray!100] (0,0) rectangle (1,-9);
  \draw[thin] (0,0) grid (9,-1);
  \draw[thin] (0,-1) grid (8,-2);
  \draw[thin] (0,-2) grid (7,-3);
  \draw[thin] (0,-3) grid (6,-4);
  \draw[thin] (0,-4) grid (5,-5);
  \draw[thin] (0,-5) grid (4,-6);
  \draw[thin] (0,-6) grid (3,-7);
  \draw[thin] (0,-7) grid (2,-8);
  \draw[thin] (0,-8) grid (1,-9);
   \draw[thin] (0,0)--(10,0)--(0,-10)--(0,0);
 \end{scope}  
   \end{tikzpicture}
\caption{Left: the gray region is the Ferrers diagram for $\gamma=[0112010]$. Right: the Ferrers diagram for $[M_i(\mu)]=[0012334232]$, which is obtained from the diagram for $\gamma$ (in light gray) by adding a leftmost column (in dark gray) containing $L-1=9$ boxes.}
\label{fig:Ferrers}
\end{figure}
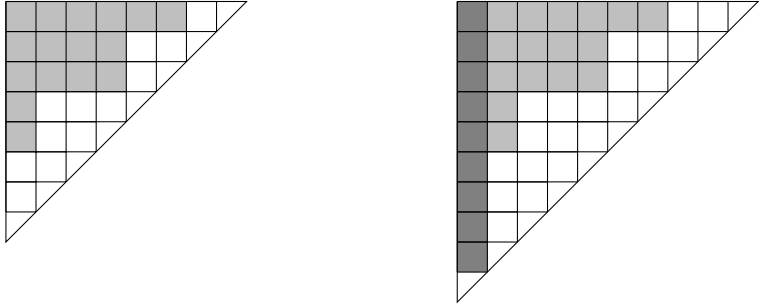

\begin{lemma}\label{lem:bridge-vec}
Assume $I$ extends to $\mu$.
Each $M_i(\mu)$ is a reduced Dyck vector of length $L$ 
starting with $0012$, ending with a positive symbol, and containing a $3$.  
$M_i(\mu)$ has dinv $i$ and deficit $k=|\mu|$.
$[M_i(\mu)]$ is a $\NU_1$-initial object with $\mind$ equal to $L$,
while $\NU([M_i(\mu)])$ is a $\NU_1$-final object with $\mind$ equal to $L+1$.  
\end{lemma}
\begin{proof}
By Lemma~\ref{lem:pre-bridge}, $z$ is a Dyck vector of length $L-2$ 
starting with $01$ and containing a $2$. Since $M_i(\mu)=00z^+$,
the first assertion in~\ref{lem:bridge-vec} holds.
Lemma~\ref{lem:semi-stats} shows that $\dinv(M_i(\mu))=\dinv(z)+1=i$ and
$\defc(M_i(\mu))=\defc(z)+\len(z)=|\lambda|+L-2=k$. 
The last claim in~\ref{lem:bridge-vec}
follows from Proposition~\ref{prop:NU-seg2}(b).
\end{proof}

For our example $\mu=\ptn{531^4}$, let us compute $M_i(\mu)$
for $i=13,15,17,19$ (this range comes from $A^*=13$ and $D=21$).
Here, $\lambda=\ftype(\mu)=\ptn{31}$.
We look up each $c_{\lambda}(i-1)$ from~\eqref{eq:small-chains},
find the representative $z$ of length $L-2=8$, and then form
$M_i(\mu)=00z^+$. The results appear in the following table.
\[ \begin{array}{l|l|l|l|l}
 i & 13 & 15 & 17 & 19 \\\hline
\gamma=c_{\ptn{31}}(i-1) & [0112010] & [0101001] & [01211210] & [01121010] \\ 
z=\qdvmap_8(\gamma) & 01223121 & 01212112 & 01211210 & 01121010 \\
M_i(\mu) & 0012334232 & 0012323223 & 0012322321 & 0012232121 
\end{array} \] 
The $\mind$-profile for the bridge part of $\C_{\mu}$ is 
\begin{equation}\label{eq:prof-B1}
 \left[\begin{array}{rrr|rr|rr|rr}
\dinv: & 13 & 14 & 15 & 16 & 17 & 18 & 19 & 20 \\
\mind: & 10 & 11 & 10 & 11 & 10 & 11 & 10 & 11 
\end{array}\right].  
\end{equation}
This part ends just before dinv index $D=21$, which is where
$\tail_2(\mu)$ begins at $[V]$. In fact, if we take $i=D$ in the definition
of $M_i(\mu)$, we find that $M_D(\mu)=V$ (see Remark~\ref{rem:MD-is-V}
for a proof).

For $\mu^*=\ptn{3321^4}$, we perform a similar calculation
using $\lambda^*=\ptn{22}$ and $i=12,14,16,18$ 
(since $A^{**}=A=12$ and $D^*=20$). The results are shown here:
\begin{equation}
 \begin{array}{l|l|l|l|l} 
 i & 12 & 14 & 16 & 18 \\\hline
\gamma=c_{\ptn{22}}(i-1) & [0122100] & [0110011] & [0001100] & [01221100] \\ 
z=\qdvmap_8(\gamma) & 01233211 & 01221122 & 01112211 & 01221100 \\
M_i(\mu^*) & 0012344322 & 0012332233 & 0012223322 & 0012332211 
\end{array} 
\end{equation}
The $\mind$-profile for the bridge part of $\C_{\mu^*}$ is
\begin{equation}\label{eq:prof-B2}
 \left[\begin{array}{rrr|rr|rr|rr}
\dinv: & 12 & 13 & 14 & 15 & 16 & 17 & 18 & 19 \\
\mind: & 10 & 11 & 10 & 11 & 10 & 11 & 10 & 11 
\end{array}\right].  
\end{equation}

The next proposition summarizes the key properties of the bridge part.

\begin{proposition}\label{prop:bridge}
Assume $I$ extends to $\mu$.
The bridge part of $\C_{\mu}$ (resp. $\C_{\mu^*}$) is a sequence of 
partitions in $\Def(k)$ indexed by consecutive dinv values from
$A^*$ to $D-1$ (resp. $A$ to $D^*-1$). The $\mind$-profile of the
bridge part of $\C_{\mu}$ (and $\C_{\mu^*}$) consists of
$(D-A^*)/2=(D^*-A)/2$ copies of $L,L+1$.
\end{proposition}

Note that $(D-A^*)/2=(D^*-A)/2$ follows from~\eqref{eq:DA-sum}.

\subsection{$\NU_1$-Initial Objects in the Tail Part}
\label{subsec:init-tail2}

The construction of the antipodal part of $\C_{\mu^*}$ 
begins with the $\NU_1$-initial objects in $\tail_2(\mu)$
(and similarly for $\C_{\mu}$ and $\tail_2(\mu^*)$,
 as shown in Figure~\ref{fig:antip}). Since $\mu$ is a flagpole
partition, Theorem~\ref{thm:NU2-chains} applies to $V=v(\lambda,a,\epsilon)$.
That theorem explicitly describes the $\NU$-chain from $[V]=\TI_2(\mu)$
to $\TI(\mu)$, including the $\mind$-profile of this chain and the
$\NU_1$-initial objects along the way. Recall that 
the reduced Dyck vectors for these $\NU_1$-initial Dyck classes,
listed in increasing order of dinv, are $S_0(V),S_1(V),\ldots,S_J(V)$. 
Here, $J$ depends on $\mu$. The corresponding vectors for $\mu^*$
are $S_0(V^*),S_1(V^*),\ldots,S_{J^*}(V^*)$.

Consider our running example $\mu=\ptn{531^4}$. We start at the reduced
Dyck vector for $\TI_2(\mu)$, which is $V=0012212112$ with
$\dinv(V)=D=21$ and $\len(V)=\mind([V])=L=10$. 
Apply Theorem~\ref{thm:NU2-chains} with $n_0=2$, $n_1=1$, $n_2=0$,
$n_3=1$, and $C=\emptyset$. The vectors $v^{(i)}$ in part (a) of the theorem
are $v^{(0)}=v$, $v^{(1)}=00121121110$, $v^{(2)}=00112111001$,
$v^{(3)}=001211110010$, and $v^{(4)}=001111001001=0B_{\mu}$.
These vectors have dinv values $21,33,35,46,48$ (in order).
Using part (b) of the theorem for $0\leq i\leq 4$, the complete
$\mind$-profile of $\tail_2(\mu)$ is
\begin{equation}\label{eq:prof-C1}
\left[\begin{array}{lll|lll|lll|llll}
\dinv: & 21 & 22 & 23 & \cdots & 34 & 35 & \cdots & 47 & 48 & \cdots 
 & \cdots & \cdots \\
\mind: & \underline{10} & 11 & \underline{10} & 11^{10} & 12 
  & \underline{11} & 12^{11} & 13 & \underline{12} & 13^{12}& 14^{13}& \cdots
\end{array}\right].  
\end{equation}
By part~(c) of the theorem, the underlined values correspond to the 
$\NU_1$-initial objects $[S_j(V)]$.  By part~(d) of the theorem, we have
$S_0(V)=V=0012212112$, $S_1(V)=0011211211$, $S_2(V)=v^{(2)}=00112111001$,
and $S_3(V)=v^{(4)}=001111001001$. We can see that the structural properties 
promised in Theorem~\ref{thm:NU2-chains}(e), (f), and (g) do hold for
the specific vectors we have computed. In particular,
$\area(S_0(V))=12$, $\area(S_1(V))=10$, $\area(S_2(V))=8$, and 
$\area(S_3(V))=6$.

Following the same procedure for $\mu^*=\ptn{3321^4}$
(using Theorem~\ref{thm:NU2-chains} or Example~\ref{ex:alg-v*}), we obtain:
\begin{equation}\label{eq:Sj*}
   S_0(V^*)=V^*=0012221122,\  
   S_1(V^*)=0012112211,\  
   S_2(V^*)=0011221101,\  
   S_3(V^*)=00111101011=0B_{\mu^*}. 
\end{equation}
Here, $\area(S_j(V^*))=A^*-2j$, which is $13,11,9,7$ for $j=0,1,2,3$.
The $\mind$-profile of $\tail_2(V^*)$ is 
\begin{equation}\label{eq:prof-C2}
\left[\begin{array}{lll|ll|lll|llll}
\dinv:& 20 & 21 & 22 & 23 & 24 & \cdots & 35 & 36 & \cdots & \cdots & \cdots \\
\mind: & \underline{10} & 11 & \underline{10} & 11 & \underline{10} & 11^{10}
 & 12 & \underline{11} & 12^{11} & 13^{12} & \cdots
\end{array}\right].  
\end{equation}

The next proposition summarizes the fundamental properties of the tail
part of $\C_{\mu}$, which follow from Theorem~\ref{thm:NU2-chains}. 
Here and below, let $L_j=\len(S_j(V))=\mind([S_j(V)])$ and $D_j=\dinv(S_j(V))$
for $0\leq j\leq J$. The analogous quantities for $\mu^*$ are
$L_j^*$ and $D_j^*$ for $0\leq j\leq J^*$. Note that we have
$L_0=L_1=L$ because $S_0(V)=V=v(\lambda,a,\epsilon)$ starts with $0012$.

\begin{proposition}\label{prop:tail2}
Assume $\mu$ is any flagpole partition.
\\ (a)~The tail part of $\C_{\mu}$, namely $\tail_2(\mu)$, is an infinite 
sequence of partitions in $\Def(k)$ indexed by consecutive dinv values 
starting at $D$.
\\ (b)~The $\mind$-profile for $\tail_2(\mu)$ consists
of weakly ascending runs starting at dinv indices $D_j$ (for $0\leq j\leq J$)
corresponding to the $\NU_1$-initial objects $[S_j(V)]$.  The $\mind$-profile
of the $j$th ascending run is a prefix of 
$L_j^1(L_j+1)^{L_j}(L_j+2)^{L_j+1}\cdots$,
where the prefix length is at least $2$ and (for $j<J$) the prefix
ends with one copy of $L_{j+1}+1$. 
The $L_j$ weakly increase from $L_0=L_1=L$ to $L_J=\mind(\TI(\mu))$. 
\\ (c)~We have $S_0(V)=V$ and $S_J(V)=0B_{\mu}$;
each $S_j(V)$ has the form $00X1$ with $X$ ternary;
and $\area(S_j(V))=A-2j$ for $0\leq j\leq J$.
\end{proposition}

When $\mu^*$ is defined, analogous results hold for the tail part of
$\C_{\mu^*}$ (using $D^*$, $L_j^*$, $D_j^*$, $J^*$, and $V^*$).

\subsection{The Antipode Map}
\label{subsec:Ant}

We now define the \emph{antipode map} $\Ant$, which acts on certain
ternary Dyck vectors that begin with $00$ and end with $1$. 
Intuitively, this map plays the following role in the overall construction.
As shown in Figure~\ref{fig:antip},
we will apply $\Ant$ to the reduced representatives $S_j(V^*)$ of
the $\NU_1$-initial objects in $\tail_2(\mu^*)$ to determine
the $\NU_1$-initial objects in the antipodal part of $\C_{\mu}$.
To compute $\Ant(S)$, we must already know two
chains $\C_{\rho}$ and $\C_{\rho^*}$ with the opposite property,
where $\rho$ depends on $S$.

Let $S$ be a TDV of deficit $k$ of the form $S=00X1$. 
Then $S=0E1$ where $E=0X$ is a ternary Dyck vector. 
Let $L'=\len(S)=\mind([S])$, $D'=\dinv(S)$, and $A'=\area(S)$.
We compute $\Ant(S)$ as follows. Since $E$ is a TDV,
$[E]$ belongs to $\tail_2(\rho)$ for a unique partition $\rho$
(Theorem~\ref{thm:tail2}(a)).
If $\rho\not\in\mcP$, then $\Ant(S)$ is not defined.
If $\rho\in\mcP$, look up $\gamma=c_{\rho^*}(A'-1)$ in $\C_{\rho^*}$ 
($\gamma$ is the unique object in $\C_{\rho^*}$ with 
 $\dinv(\gamma)=A'-1$). Let $z=\qdvmap_{L'-2}(\gamma)$, 
and define $\Ant(S)=00z^+$.  Lemma~\ref{lem:Ant-map}(a) shows this 
construction makes sense, but first we look at some examples.

\begin{example}\label{ex:Ant}
Continuing our running example, we now compute $\Ant(S_j(V^*))$ and 
$\Ant(S_j(V))$ for $j=1,2,3$. First, consider $S=S_1(V^*)=0012112211$.
We have $S=0E1$ with $E=01211221$.
The Dyck class $[E]=[01211221]=[0100110]$ belongs to plateau 2 of 
$\tail(\ptn{22})\subseteq \C_{\ptn{22}}$ by Example~\ref{ex:plateau}(b).  
The Dyck class $[E]$ has $\mind\leq 8=\len(E)$, $\area_8=10$, 
and $\dinv=14$. Here, $\rho=\ptn{22}$, and the induction 
hypothesis~\eqref{eq:small-chains} provides us with $\rho^*=\ptn{31}$ 
and the chain $\C_{\ptn{31}}$. To continue,
we find the unique object $\gamma=c_{\ptn{31}}(10)$
in $\C_{\ptn{31}}$ with $\dinv=10$, which 
(by the opposite property) is guaranteed to have
$\mind\leq 8$ and $\area_8=14$. From~\eqref{eq:small-chains}, we find 
$\gamma=\ptn{6332}=[0121120]$. Then $z=\qdvmap_8(\gamma)=01232231$,
and $\Ant(S)=00z^+=0012343342$.  $\Ant(S)$ is a reduced Dyck vector with
$\len=10$, $\area=22=\dinv(S)$, and $\dinv=11=\area(S)$.

Second, consider $S=S_2(V^*)=0011221101$. 
Here $E=01122110$, $[E]$ belongs to plateau 3 of
$\tail(\ptn{22})\subseteq\C_{\ptn{22}}$ (Example~\ref{ex:plateau}(c)),
and $[E]$ has $\area_8=8$, $\dinv=16$, and $\mind\leq 8$.
For this $S$, we again have $\rho=\ptn{22}$ and $\rho^*=\ptn{31}$.
From~\eqref{eq:small-chains},
we look up $\gamma=c_{\ptn{31}}(8)=\ptn{642}=[0123210]$,
so $z=\qdvmap_8(\gamma)=[01234321]$.
Note that $\mind(\gamma)\leq 8$, $\dinv(\gamma)=8$, and $\area_8(\gamma)=16$.
Here, $\Ant(S)=0012345432$, which is a reduced Dyck vector
with $\mind=10$, $\area=24=\dinv(S)$, and $\dinv=9=\area(S)$.

Third, consider $S=S_3(V^*)=00111101011$.
Here $E=011110101$, $[E]$ belongs to plateau 4
of $\tail(\ptn{21})\subseteq \C_{\ptn{21}}$ (Example~\ref{ex:plateau}(a)),
and $[E]$ has $\area_9=6$, $\dinv=27$, and $\mind\leq 9$. 
For this $S$, we have $\rho=\ptn{21}$ 
and $\rho^*=\ptn{111}$. From~\eqref{eq:small-chains}, we find
$\gamma=c_{\ptn{111}}(6)=\ptn{441}=[012201]$ and 
$z=\qdvmap_{9}(\gamma)=012345534$.
So $\Ant(S)=00123456645$, which is a reduced Dyck vector with
$\mind=11$, $\area=36=\dinv(S)$, and $\dinv=7=\area(S)$.

We compute each $\Ant(S_j(V))$ similarly.
For $S=S_1(V)=0011211211$, we have $E=01121121$,
$[E]=[0010010]\in\tail(\ptn{31})$, $\rho=\ptn{31}$, $\rho^*=\ptn{22}$,
$\gamma=c_{\ptn{22}}(9)=\ptn{53221}=[000110]$, $z=01222332$,
and $\Ant(S)=0012333443$.
For $S=S_2(V)=00112111001$, we have $E=011211100$,
$[E]\in\tail(\ptn{3})$, $\rho=\ptn{3}=\rho^*$,
$\gamma=c_{\ptn{3}}(7)=\ptn{5221}=[011120]$, $z=012344453$,
and $\Ant(S)=00123455564$.
For $S=S_3(V)=001111001001$, we have $E=0111100100$,
$[E]\in\tail(\ptn{2})$, $\rho=\ptn{2}=\rho^*$,
$\gamma=c_{\ptn{2}}(5)=\ptn{43}=[01200]$, $z=0123456755$,
and $\Ant(S)=001234567866$.
\end{example}

\begin{lemma}\label{lem:Ant-map}
Let $S=00X1=0E1$ be a TDV of deficit $k$,
 length $L'$, area $A'$, and dinv $D'$, 
 such that $[E]\in\tail_2(\rho)$ and $\rho\in\mcP$.
\\ (a)~There exists a unique partition
$\gamma\in\C_{\rho^*}$ with $\dinv(\gamma)=A'-1$
and $\mind(\gamma)\leq L'-2$. So $\Ant(S)$ is well-defined.
\\ (b)~$\Ant(S)$ is a reduced Dyck vector with deficit $k$,
 length $L'$, area $D'$, and dinv $A'$.
\\ (c) The Dyck class $[\Ant(S)]$ is a $\NU_1$-initial object
 with $\mind([\Ant(S)])=L'$.
\end{lemma}
\begin{proof}
We have $L'\geq 3$. By Lemma~\ref{lem:semi-stats}, 
$|\rho|=\defc(E)=k-(L'-2)<k$, $\dinv(E)=D'-(L'-2)$, and $\area(E)=A'-1$.
Because $\rho\in\mcP$, $\rho^*$ is defined, and we know
$\C_{\rho}$ and $\C_{\rho^*}$ satisfy the opposite property.
Now $[E]$ is an object in $\C_{\rho}$ with $\mind([E])\leq\len(E)=L'-2$,
$\dinv([E])=D'-(L'-2)$, and $\area_{L'-2}([E])=A'-1$.
So the opposite property guarantees the existence of a unique
$\gamma\in\C_{\rho^*}$, namely $\gamma=c_{\rho^*}(A'-1)$,
such that $\mind(\gamma)\leq L'-2$, $\dinv(\gamma)=A'-1=\area(E)$,
and $\area_{L'-2}(\gamma)=D'-(L'-2)=\dinv(E)$. 
So $z=\qdvmap_{L'-2}(\gamma)$ is a Dyck vector (not just a QDV)
with length $L'-2$, area $D'-(L'-2)$, and dinv $A'-1$.
Thus, $\Ant(S)=00z^+$ is a well-defined reduced Dyck vector
beginning with $00$. By Lemma~\ref{lem:semi-stats},
$\len(\Ant(S))=L'$, $\area(\Ant(S))=D'$, $\dinv(\Ant(S))=A'$,
and hence $\defc(\Ant(S))=k$.  Since $\Ant(S)=00z^+$ has leader $0$ 
and a positive final symbol, $[\Ant(S)]$ is a $\NU_1$-initial object by 
Proposition~\ref{prop:NU-QDV}(b).  
\end{proof}

In each computation from Example~\ref{ex:Ant},
the Dyck class $[E]$ always appeared in some $\tail(\rho)$,
not just in $\tail_2(\rho)$. Remark~\ref{rem:E-in-tail1} proves this
always happens for $S=S_j(V)$ or $S=S_j(V^*)$, which are the only
cases of interest below.  This fact lets us quickly compute $\rho$ 
from $S_j(V)$ using Theorem~\ref{thm:BDV-tail}, as illustrated
in Example~\ref{ex:Ant}. Each such $\rho$ is one of the needed
partitions for $\mu$.

\subsection{The Antipodal Parts of $\C_{\mu}$ and $\C_{\mu^*}$}
\label{subsec:antip}

Assume $I$ extends to $\mu$ and all needed partitions for $\mu$ are in $\mcP$.
Then we can define \emph{antipodal vectors} $A_j=\Ant(S_j(V))$ for 
$1\leq j\leq J$ and $A_j^*=\Ant(S_j(V^*))$ for $1\leq j\leq J^*$.
The \emph{antipodal part} of $\C_{\mu}$ consists of two-element
$\NU_1$-segments $[A_j^*],\NU_1([A_j^*])$, taken in order 
 from $j=J^*$ down to $j=1$ (see Figure~\ref{fig:antip}).
Similarly, the antipodal part of $\C_{\mu^*}$ consists of $\NU_1$-segments 
$[A_j],\NU_1([A_j])$, taken in order from $j=J$ down to $j=1$.
We check that everything works in Proposition~\ref{prop:antip-part}
after considering some examples.

\begin{example}\label{ex:antip-part}
Take $\mu=\ptn{531^4}$ and $\mu^*=\ptn{3321^4}$, so $J=J^*=3$.
In Example~\ref{ex:Ant}, we computed 
$A_1^*=0012343342$, $A_2^*=0012345432$, and $A_3^*=00123456645$.
The antipodal part of $\C_{\mu}$ is
\[ [A_3^*],\NU_1([A_3^*]),[A_2^*],\NU_1([A_2^*]),[A_1^*],\NU_1([A_1^*]), \]
where each $[A_j^*]$ is $\NU_1$-initial and
      each $\NU_1([A_j^*])$ is $\NU_1$-final.
The $\mind$-profile for the antipodal part of $\C_{\mu}$ is
\begin{equation}\label{eq:prof-A1}
\left[\begin{array}{rrr|rr|rr}
\dinv: & 7 & 8 & 9 & 10 & 11 & 12 \\
\mind: & 11 & 12 & 10 & 11 & 10 & 11
\end{array}\right].  
\end{equation}
This antipodal part ends at dinv index $12$,
while the bridge part of $\C_{\mu}$ starts at dinv index $A^*=13$.

From Example~\ref{ex:Ant}, we also have
$A_1=0012333443$, $A_2=00123455564$, and $A_3=001234567866$.
The antipodal part of $\C_{\mu^*}$ is
\[ [A_3],\NU_1([A_3]),[A_2],\NU_1([A_2]),[A_1],\NU_1([A_1]), \]
which consists of three two-element $\NU_1$-segments.
The $\mind$-profile for the antipodal part of $\C_{\mu^*}$ is
\begin{equation}\label{eq:prof-A2}
 \left[\begin{array}{rrr|rr|rr}
\dinv: & 6 & 7 & 8 & 9 & 10 & 11 \\
\mind: & 12 & 13 & 11 & 12 & 10 & 11
\end{array}\right].  
\end{equation}
This antipodal part ends at dinv index $11$,
while the bridge part of $\C_{\mu^*}$ starts at dinv index $A^{**}=A=12$.
\end{example}

\begin{proposition}\label{prop:antip-part}
Assume $I$ extends to $\mu$ and all needed partitions for $\mu$ are in $\mcP$.
\\ (a)~Each $A_j$ is a well-defined reduced Dyck vector
 with length $L_j$ that starts with $00$ and ends with a positive symbol.
 In fact, $A_j$ starts with $0012$ and contains a $3$.
\\ (b) Each $[A_j]$ is a $\NU_1$-initial object with $\mind=L_j$. 
Each $\NU_1([A_j])$ is a $\NU_1$-final object with $\mind=L_j+1$.
\\ (c) $\defc([A_j])=k$ and $\dinv([A_j])=A-2j$.
\\ (d) Parts (a), (b), and (c) are true replacing
 $A_j$ by $A_j^*$, $L_j$ by $L_j^*$, and $A$ by $A^*$.
\\ (e) The antipodal part of $\C_{\mu}$ is a sequence of
partitions in $\Def(k)$ indexed by consecutive dinv values from
$\ell(\mu^*)$ to $A^*-1$. The $\mind$-profile for the antipodal part of 
$\C_{\mu}$ is $L_{J^*}^*,L_{J^*}^*+1,\ldots,L_2^*,L_2^*+1,L_1^*,L_1^*+1$.
\\ (f) The antipodal part of $\C_{\mu^*}$ is a sequence of
partitions in $\Def(k)$ indexed by consecutive dinv values from
$\ell(\mu)$ to $A-1$.  The $\mind$-profile for the antipodal part of 
$\C_{\mu^*}$ is $L_J,L_J+1,\ldots,L_2,L_2+1,L_1,L_1+1$.  
\end{proposition}
\begin{proof}
Proposition~\ref{prop:tail2}(c) shows that each $S_j(V)$ is a valid
input to $\Ant$. The first sentence of part~(a) follows from the definition
of $\Ant$. The proof that $A_j$ must start with $0012$ and contain a $3$
is rather technical and is postponed to~\S\ref{subsec:prove-antip}.
Part~(b) follows from part~(a) and Proposition~\ref{prop:NU-seg2}(b).
For part~(c), Lemma~\ref{lem:Ant-map}(b) and Proposition~\ref{prop:tail2}(c) 
imply $\defc(A_j)=k$ and $\dinv(A_j)=\area(S_j)=A-2j$.
Part~(d) follows by replacing $\mu$ (and associated quantities) by $\mu^*$. 
By Proposition~\ref{prop:TI}(c), $S_{J^*}(V^*)=0B_{\mu^*}$ 
has area $\ell(\mu^*)$. Part~(e) follows from this observation and 
part (c). Similarly, part~(f) follows from~(d) and the formula
$\area(S_J(V))=\area(0B_{\mu})=\ell(\mu)$.
\end{proof}

\subsection{Proofs of the Main Theorems}
\label{subsec:prove-main-thm}

We can now complete the proof of Theorem~\ref{thm:extend-mu}.
Assume $\mu$ satisfies the hypotheses of that theorem.
Then we can build the chains $\C_{\mu}$ and $\C_{\mu^*}$
by combining the tail parts, bridge parts, and antipodal parts
described in the previous subsections. We need only verify that
these chains have basic required structure, local required structure,
extra required structure, and satisfy the $amh$-hypotheses.

The antipodal part of $\C_{\mu}$ consists of partitions with
deficit $k=|\mu|$ and consecutive dinv values from $\ell(\mu^*)$
to $A^*-1$ (Proposition~\ref{prop:antip-part}(c) and (e)).
The bridge part of $\C_{\mu}$ consists of partitions with deficit $k$ and 
consecutive dinv values from $A^*$ to $D-1$ (Proposition~\ref{prop:bridge}).
The tail part of $\C_{\mu}$ is $\tail_2(\mu)$, which consists of
partitions with deficit $k$ and dinv values from $D=\dinv(\TI_2(\mu))$ onward.
Analogous statements hold for $\C_{\mu^*}$. These comments prove
the basic structure conditions~\ref{def:basic-str}(a) through (d)
and extra condition~\ref{def:extra-str}(d). In~\S\ref{subsec:prove-disj},
we prove that $\C_{\mu}$ and $\C_{\mu^*}$ are disjoint (have no terms in 
common) when $\mu\neq\mu^*$.

Next, we show that the $h$-vector for $\C_{\mu}$ is
\begin{equation}\label{eq:hvec}
 (L^*_{J^*},\ldots,L_2^*,L_1^*,L^{(D-A^*)/2},L_0,L_1,L_2,\ldots,L_J). 
\end{equation}
The ascending runs for the $\mind$-profile within the three parts of
$\C_{\mu}$ are given in Propositions~\ref{prop:bridge},~\ref{prop:tail2}, 
and~\ref{prop:antip-part},
but we must still check that a new ascending run begins at the start
of the bridge part and the tail part. Since $L_1^*=L=L_0$, the last 
objects in the antipodal part and the bridge part (if nonempty)
both have $\mind=L+1$, while the first objects in the bridge part
(if nonempty) and the tail part have $\mind=L$. So~\eqref{eq:hvec} holds.
The last entry in the $h$-vector comes from $[S_J(V)]=\TI(\mu)$,
so the last entry in the $a$-vector is $a_N=\dinv(\TI(\mu))$, verifying 
local condition~\ref{def:local-str}(a). Condition~\ref{def:local-str}(b)
follows from the $\mind$-profiles computed in
Propositions~\ref{prop:bridge},~\ref{prop:tail2}, and~\ref{prop:antip-part},
which also show that the $m$-vector for $\C_{\mu}$ has all entries $0$.

By analogous reasoning, the $h$-vector for $\C_{\mu^*}$ is
\[ (L_J,\ldots,L_2,L_1,L^{(D^*-A)/2},L_0^*,L_1^*,L_2^*,\ldots,L^*_{J^*}). \]
Since $(D-A^*)/2=(D^*-A)/2$ and $L_0=L=L_0^*$, the two $h$-vectors
are reversals of each other and $amh$-hypothesis~\ref{def:amh-hyp}(a) holds.
Now that we know the $amh$-vectors for $\C_{\mu}$ and $\C_{\mu^*}$ all
have the same length $N$, \ref{def:amh-hyp}(b) follows 
since both $m$-vectors are $0^N$.

To check~\ref{def:amh-hyp}(c),
we look at cases based on the Dyck class $[v]$ in $\C_{\mu}$
with $\dinv(v)=a_i$. By our determination of the $h$-vector of $\C_{\mu}$,
$[v]$ is one of the $\NU_1$-initial objects in $\C_{\mu}$. 
Recall that $\dinv(v)+\defc(v)+\area(v)=\binom{\len(v)}{2}$ 
holds for all Dyck vectors $v$.  
First consider the tail case where $v=S_j(\mu)$ for some $j$ between
$1$ and $J$.  Then $a_i=\dinv(v)$, $m_i=0$, $k=\defc(v)$,
$a^*_{N+1-i}=\dinv(\Ant(S_j(\mu)))=\area(S_j(\mu))=\area(v)$,
and $h_i=\mind(v)=\len(v)$. So~\ref{def:amh-hyp}(c) holds.
In the special case where $v=S_0(\mu)=V$, 
we have $a_i=\dinv(V)=D$, $m_i=0$, $k=\defc(V)$, and $h_i=L=\len(V)$.
If the bridge parts are nonempty, then $a_{N+1-i}^*$ is the dinv index
of the first object in the bridge of $\C_{\mu^*}$, namely $A=\area(V)$.
If the bridge parts are empty, then $a_{N+1-i}^*$ is the dinv index
of $S_0(\mu^*)=V^*$, namely $D^*$, but $D^*=A$ since the bridge is empty.
In all these situations, \ref{def:amh-hyp}(c) holds since $D+k+A=\binom{L}{2}$
by definition of $\defc(V)$.

Next consider the bridge case where $v=M_{D-2j}(\mu)$ for some $j>0$.
Then $a_i=\dinv(v)=D-2j$, $m_i=0$, $k=\defc(v)$, and $h_i=L$.
We find $a^*_{N+1-i}=A+2j$ by counting up from the beginning
of the bridge part of $\C_{\mu^*}$ (Proposition~\ref{prop:bridge}).
Again,~\ref{def:amh-hyp}(c) holds since $(D-2j)+0+k+(A+2j)=D+k+A=\binom{L}{2}$.

Next consider the antipodal case where 
$v=\Ant(S_j(\mu^*))$ for some $j$ between $1$ and $J^*$.
Then $a_i=\dinv(v)=\area(S_j(\mu^*))$, $m_i=0$, $k=\defc(v)=\defc(S_j(\mu^*))$, 
$a_{N+1-i}^*=\dinv(S_j(\mu^*))$, and $h_i=\mind(v)=L_j^*=\len(S_j(\mu^*))$.  
So~\ref{def:amh-hyp}(c) holds here, as well.
The opposite property for $\C_{\mu}$ and $\C_{\mu^*}$
now follows from Theorem~\ref{thm:local-opp}.

To finish the proof, 
we check the first three extra structural conditions for $\C_{\mu}$.
Condition~\ref{def:extra-str}(a) follows from~\eqref{eq:hvec} and the 
fact (Proposition~\ref{prop:tail2}(b)) that the sequences $(L_j)$
and $(L_j^*)$ are weakly increasing with $L_0=L_0^*=L$.
Similarly, Proposition~\ref{prop:tail2}(b) implies
condition~\ref{def:extra-str}(b) for the ascending runs in $\tail_2(\mu)$.
Condition~(b) is immediate for the runs in the
antipodal and bridge parts, which all have length $2$ and $\mind$-profiles
of the form $L',L'+1$. Finally, condition~\ref{def:extra-str}(c) 
follows from Theorem~\ref{thm:BDV-tail} for $\tail(\mu)$,
Theorem~\ref{thm:NU2-chains}(e) for the rest of $\tail_2(\mu)$,
and Proposition~\ref{prop:NU-seg2}(b) for the antipodal and bridge parts.  

Now, Theorem~\ref{thm:extend-all} follows readily from
Theorem~\ref{thm:extend-mu} using induction on $k$. 
The only condition not already checked is~\ref{def:chain-coll}(e),
which requires the chains in all the new chain collections
to be pairwise disjoint. We prove this (for each value of $k$
in the inductive construction) in~\S\ref{subsec:prove-disj}.
Since objects in chains for different values of $k$ have different deficits,
the final collection~$\CC'$ also satisfies~\ref{def:chain-coll}(e).

\section{The Remaining Proofs}
\label{sec:proofs}

\subsection{Proof of Lemma~\ref{lem:D>=A*}.}
\label{subsec:exceptional}

We use the notation $\mu$, $\lambda$, $L$, $D$, $A^*$ introduced 
in~\S\ref{subsec:flag-invo}.  
Assume $D<A^*$, so $D-A^*$ is a negative even integer by 
Lemma~\ref{lem:check-mu*}.
We first show that $L\in\{6,7,8\}$ and $\lambda=\ptn{0}$.
Using~\eqref{eq:TI2-stats}, we have
\[ D-A^*\geq \binom{L}{2}-3L-|\lambda|+\lambda_1+7-(2L-\lambda_1^*-5). \]
Now $L\geq |\lambda|+6\geq 6$ since $\mu$ is a flagpole partition. 
Using this inequality to eliminate $-|\lambda|$, we find
\[ D-A^*\geq \binom{L}{2}-5L-L+6+\lambda_1+12+\lambda_1^*
          =p(L)+\lambda_1+\lambda_1^*, \]
where $p(L)=(L^2-13L+36)/2$. Now $p(6)=p(7)=-3$, $p(8)=-2$,
and $p(L)\geq 0$ for all $L\geq 9$. So $L$ must be $6$, $7$, or $8$.
Suppose, to get a contradiction, that $\lambda\neq\ptn{0}$.
Then $\lambda^*\neq\ptn{0}$, so $\lambda_1+\lambda_1^*\geq 2$,
so $D-A^*\geq -3+2=-1$. This is impossible, since $D-A^*$ is negative and even.
Thus, $\lambda=\ptn{0}$.  

We now know that the assumption $D<A^*$ is possible only if
$\TI_2(\mu)=[V]$, where $V$ is one of the six vectors
$v(\ptn{0},a,\epsilon)=0012^{a-\epsilon}1^{\epsilon}$ 
with $a\in\{3,4,5\}$ and $\epsilon\in\{0,1\}$. The following
table computes $\mu$, $D$, $A$, $\mu^*$, and $A^*$ 
for each such $V$.

\begin{center}
\begin{tabular}{l|c|c|c|c|c}
$V$ & $\mu$ & $D$ & $A$ & $\mu^*$ & $A^*$ \\\hline
$001222$ & $\ptn{21^2}$   & 4 & 7 & $\ptn{1^4}$ & 6 \\
$001221$ & $\ptn{1^4}$    & 5 & 6 & $\ptn{21^2}$& 7 \\\hline
$0012222$ & $\ptn{1^5}$   & 7 & 9 & $\ptn{1^5}$ & 9 \\
$0012221$ & $\ptn{21^3}$  & 8 & 8 & $\ptn{21^3}$ & 8 \\\hline
$00122222$ & $\ptn{21^4}$ & 11& 11& $\ptn{21^4}$ & 11 \\
$00122221$ & $\ptn{1^6}$  & 12& 10& $\ptn{1^6}$ & 10 \\
\end{tabular}
\end{center}

We see that $D<A^*$ occurs in the first three rows only.
But these cases are ruled out because $|\mu|\geq 6$
by Definition~\ref{def:chain-coll}(a).

\subsection{Proof of Lemma~\ref{lem:pre-bridge}.}
\label{subsec:prove-bridge}

We must prove that for all $i\in\{A^*,A^*+2,\ldots,D-4,D-2\}$, 
$\gamma=c_{\lambda}(i-1)$ exists, $\mind(\gamma)\leq L-2$,
and $z=\qdvmap_{L-2}(\gamma)$ starts with $01$ and contains a $2$. 
When $\mind(\gamma)\leq L-3$, the conclusion about $z$ follows
if $\qdvmap_{L-3}(\gamma)$ contains a $1$, 
or equivalently $\gamma\neq [0^{L-3}]$.

Since $\lambda\in\mcP$,
the chain $\C_{\lambda}$ already exists and starts at dinv index
$\ell(\lambda^*)$. To show that $c_{\lambda}(i-1)$ exists
for all $i$ in the given range, it suffices to prove 
$A^*-1\geq\ell(\lambda^*)$. Using~\eqref{eq:TI2-stats}
and the bounds $L\geq |\lambda|+6=|\lambda^*|+6$,
$|\lambda^*|\geq \lambda_1^*$, and $|\lambda^*|\geq\ell(\lambda^*)$, 
we have the stronger bound
\[ A^*-1\geq 2L-\lambda_1^*-7\geq 2|\lambda^*|+5-\lambda_1^*
 \geq \ell(\lambda^*)+5. \]

Recall that $\TI_2(\mu)=[v(\lambda,a,\epsilon)]$
where $v(\lambda,a,\epsilon)=0012^{a-\epsilon}B_{\lambda}^+1^{\epsilon}$
has length $L$ and dinv $D$. By Lemma~\ref{lem:semi-stats} and
Theorem~\ref{thm:BDV-tail}(b), $w=01^{a-\epsilon}B_{\lambda}0^{\epsilon}$
has dinv $D-1$, has length $L-2$, and belongs to plateau $a$
of $\tail(\lambda)\subseteq\C_{\lambda}$. 
This means that $c_{\lambda}(D-1)=[w]$.
Every object $c_{\lambda}(i-1)$ considered here is a
partition appearing in the chain $\C_{\lambda}$ 
an even positive number of steps before $[w]$. 
So the needed conclusion follows from the next lemma.

\begin{lemma}\label{lem:mind-before-w}
Assume $I$ extends to $\mu$ and $\lambda=\ftype(\mu)$.
Let $\gamma$ be any partition in the chain $\C_{\lambda}$ at least
two steps before $[w]=c_{\lambda}(D-1)$, 
where $w=01^{a-\epsilon}B_{\lambda}0^{\epsilon}$.
Then $\gamma$ satisfies one of these conditions: 
(a)~$\mind(\gamma)\leq L-3$ and $\gamma\neq[0^{L-3}]$; 
(b)~$\mind(\gamma)=L-2$ and $z=\qdvmap_{L-2}(\gamma)$ starts with $01$
 and contains a $2$.
\end{lemma}
\begin{proof}
First consider the case where $\gamma$ is not in $\tail(\lambda)$.
Let $\lambda$ have $h$-vector $(h_1,\ldots,h_N)$.
For some $j<N$, $\mind(\gamma)$ is a value in the $j$th ascending run
of the $\mind$-profile for $\C_{\lambda}$.
By extra structure conditions~\ref{def:extra-str}(a) and~(b), 
we deduce $\mind(\gamma)\leq 1+\max(h_1,h_N)$. 
Since the last ascending run of the $\mind$-profile corresponds
to $\tail(\lambda)$, we have $h_N=\mind(\TI(\lambda))$.
By hypothesis~\ref{def:amh-hyp}(a), $h_1$ is the last entry in the $h$-vector
for $\lambda^*$, so that $h_1=\mind(\TI(\lambda^*))$.
Proposition~\ref{prop:TI}(a) shows $\mind(\TI(\lambda))
 =\lambda_1+\ell(\lambda)+1\leq |\lambda|+2$, and similarly 
 $\mind(\TI(\lambda^*))\leq |\lambda^*|+2=|\lambda|+2$.
Since we know $|\lambda|+6\leq L$, we finally get
\begin{equation}\label{eq:mind-est}
 \mind(\gamma)\leq 1+\max(\mind(\TI(\lambda^*)),\mind(\TI(\lambda)))
 \leq |\lambda|+3\leq L-3. 
\end{equation}
Here, $\gamma$ cannot be $[0^{L-3}]$, since $[0^{L-3}]$ appears only
in $\tail(\ptn{0})$ by Theorem~\ref{thm:BDV-tail},
and $\gamma$ is in $\C_{\lambda}$ outside $\tail(\lambda)$.
So (a) holds.
 
Next consider the case where $\gamma$ is in $\tail(\lambda)$ 
and $\lambda\neq\ptn{0}$ 
(so that $\gamma\neq [0^{L-3}]$ and $w$ must be reduced).
Since $\mind$ values weakly increase as we move forward through the tail
(Theorem~\ref{thm:tail-profile}), we have
$\mind(\gamma)\leq \mind([w])=\len(w)=L-2$.  
If $\gamma$ appears in the tail before plateau $a$, 
then $\mind(\gamma)\leq L-3$, so (a) holds. Suppose $\gamma$ appears
in plateau $a$ before $[w]$, so that $\mind(\gamma)=L-2$. 
Because $\epsilon$ is $0$ or $1$, $[w]$ is the first or second Dyck
class listed in Theorem~\ref{thm:BDV-tail}(b). Since $\gamma$
precedes $[w]$ in the tail by at least $2$, $\gamma$ must be
one of the Dyck classes listed in Theorem~\ref{thm:BDV-tail}(a).
Then $z$ is one of the reduced vectors listed there,
which all begin with $01$ and contain a $2$ since
$B_{\lambda}$ begins with $0$ and ends with $1$. So (b) holds.

Finally, consider the two special cases where $\gamma\in\tail(\ptn{0})$.
If $\epsilon=0$, then $[w]=[01^a]=[0^a]=[0^{L-3}]$. Since $\gamma$
appears before $[w]$ in the tail, $\mind(\gamma)\leq L-3$ and
$\gamma$ is not $[0^{L-3}]=[w]$. So (a) holds.
If $\epsilon=1$, then $w=01^{a-1}0$ has length $L-2$, 
$[w]$ is the first object in plateau $a$ of $\tail(\ptn{0})$,
and the immediate predecessor of $[w]$ is $[01^a]=[0^a]=[0^{L-3}]$
(Theorem~\ref{thm:BDV-tail}). 
Because $\gamma$ precedes $[w]$ by at least $2$, $\mind(\gamma)\leq L-3$
and $\gamma\neq [0^{L-3}]$. So (a) holds.
\end{proof}

\begin{remark}\label{rem:MD-is-V}
We now show that $M_D(\mu)=V$ using the definition of $M_i(\mu)$
from~\S\ref{subsec:bridge}. We saw above that $c_{\lambda}(D-1)=[w]$
where $w=01^{a-\epsilon}B_{\lambda}0^{\epsilon}$ has length $L-2$.
So $M_D(\mu)=00w^+=v(\lambda,a,\epsilon)=V$.  
\end{remark}

\subsection{Proof of Proposition~\ref{prop:antip-part}(a).}
\label{subsec:prove-antip}

We must prove that each $A_j=\Ant(S_j(V))$ 
starts with $0012$ and contains a $3$. Fix $j$ between $1$ and $J$,
and write $S_j(V)=0E1$. As in~\S\ref{subsec:Ant},
let $\rho$ be the unique partition with $[E]$ in $\tail_2(\rho)$, 
let $\gamma$ be the unique object in $\C_{\rho^*}$ with
 $\dinv(\gamma)=\area(E)=A-2j-1$ (see Propostion~\ref{prop:tail2}(c)),
and let $z=\qdvmap_{L_j-2}(\gamma)$. Since $A_j=00z^+$, it suffices to prove: 
either $\mind(\gamma)\leq L_j-3$ and $\gamma\neq [0^{L_j-3}]$; 
or else $\mind(\gamma)=L_j-2$, $z$ starts with $01$, and $z$ contains a $2$.
Recall $S_0(\mu)=V=0012^{n_0}B_{\lambda}^+1^{\epsilon}$
where $n_0=a-\epsilon\geq 1$, $\lambda=\ftype(\mu)$, 
and $\epsilon$ is $0$ or $1$.

\emph{Case~1:} Assume $1\leq j\leq \lfloor n_0/2\rfloor$.
Here $S_j(V)=0012^{n_0-2j}B_{\lambda}^+1^{\epsilon+2j}$ 
by~\eqref{eq:tail2-init}, so $L_j=\len(S_j(V))=\len(V)=L$.
Also, $E=012^{n_0-2j}B_{\lambda}^+1^{\epsilon+2j-1}$ 
has length $L-2$ and area $A-2j-1$.
We see that $[E]=[01^{n_0-2j}B_{\lambda}0^{\epsilon+2j-1}]$
belongs to $\tail(\lambda)$ by Theorem~\ref{thm:BDV-tail}(b).
Now, $\gamma=c_{\lambda^*}(A-2j-1)$ is an object in $\C_{\lambda^*}$
that is at least two steps before $c_{\lambda^*}(D^*-1)$,
since $j>0$ and $A\leq D^*$ (by Lemma~\ref{lem:D>=A*} for $\mu^*$).
The required conclusions now follow
from Lemma~\ref{lem:mind-before-w} (applied to $\mu^*$
and $\lambda^*=\ftype(\mu^*)$), recalling from~\S\ref{subsec:flag-invo}
that $V$ and $V^*$ both have length $L$.  

\emph{Case~2:}
Assume $\lfloor n_0/2\rfloor < j < J$. 
The description of $S_j(\mu)$ in Theorem~\ref{thm:NU2-chains}(d) shows
that $E=01X^+2W$ for some binary vectors $X$ and $W$.  
$W$ must contain a $0$ since for these $j$, the value of $i$
in~\eqref{eq:loop-iter} and~\eqref{eq:tail2-init} must be at least $1$.
So $E$ is reduced. We further claim that $W$ starts with $1^{c-1}0$
for some $c\geq 2$. Formulas~\eqref{eq:loop-iter} and~\eqref{eq:tail2-init}
show that the last $2$ in $E$ is followed by 
$1^{p_i+\epsilon+2\lfloor n_0/2\rfloor}$.
If $\epsilon=1$, then this string of $1$s is nonempty.
If $\epsilon=0$, then $n_0/2=a/2\geq 1$ (since $a\geq 2$),
and again the string of $1$s is nonempty. Theorem~\ref{thm:BDV-tail}(a)
now shows that $[E]$ belongs to the $c$th plateau of $\tail(\rho)$ 
and that $\rho\neq\ptn{0}$. Let $s$ be the number of objects
in this plateau weakly following $[E]$, and let $n_0=\mind(\TI(\rho))
=\rho_1+\ell(\rho)+1$. By Theorem~\ref{thm:tail-profile}, $s\leq n_0+c-1$.
Since $E$ is reduced, $L_j-2=\len(E)=\mind([E])=n_0+c$.

Because $\rho$ is a needed partition for $\mu$, we know $\rho\in\mcP$,
so $\C_{\rho}$ and $\C_{\rho^*}$ satisfy the opposite property. 
For each $n>0$, let $\C_{\rho}^{\leq n}$ be the finite set of $\gamma$ in 
$\C_{\rho}$ with $\mind(\gamma)\leq n$; define $\C_{\rho^*}^{\leq n}$ 
similarly. We use dinv to order these sets, so ``the second largest
object in $\C_{\rho}^{\leq n}$'' refers to the object with the
second largest dinv value.

Take $n=n_0+c$. The $c$th plateau of $\tail(\rho)$ consists 
of objects with $\mind=n$, while objects in all later plateaus
have $\mind>n$. So $[E]$ is the $s$th largest object in
$\C_{\rho}^{\leq n}$, and the $s$ largest objects have consecutive
dinv values. Recall from the proof of Lemma~\ref{lem:Ant-map}
that $\gamma$ is obtained from $[E]$ by invoking the opposite property 
$\Cat_{n,\rho^*}(t,q)=\Cat_{n,\rho}(q,t)$ for this value of $n$
(namely $n=L_j-2=n_0+c$).
Thus $\gamma$ must be the $s$th smallest object in $\C_{\rho^*}^{\leq n}$,
where the $s$ smallest objects have consecutive dinv values.
Hypothesis~\ref{def:amh-hyp}(a) shows
that the smallest object in $\C_{\rho^*}$, namely $\delta
=c_{\rho^*}(\ell(\rho))$, has $\mind(\delta)=\mind(\TI(\rho))=n_0\leq n$. 
Therefore $\gamma$ must be $c_{\rho^*}(\ell(\rho)+s-1)$.

Now we prove $\mind(\gamma)\leq L_j-3$.
Apply the opposite property again, with $n-1$ instead of $n$. 
The largest objects in $\C_{\rho}^{\leq n-1}$ are the objects in plateaus $0$ 
through $c-1$ of $\tail(\rho)$, which have consecutive dinv values. 
Because $c\geq 2$ and $s\leq n_0+c-1$, there are at least $s$ such objects 
(plateau $0$ contributes $1$ and plateau $c-1$ contributes $n_0+c-2$).
So the $s$ smallest objects in $\C_{\rho^*}^{\leq n-1}$ have
consecutive dinv values. Once again, the smallest object in $\C_{\rho^*}$,
namely $\delta$, has $\mind(\delta)=n_0\leq n-1$. So the $s$th smallest object 
in $\C_{\rho^*}^{\leq n-1}$ is $c_{\rho^*}(\ell(\rho)+s-1)=\gamma$. 
Thus $\mind(\gamma)\leq n-1=L_j-3$, as needed.
Now $\rho^*\neq\ptn{0}$ since $\rho\neq\ptn{0}$,
and $[0^{L_j-3}]\in\tail(\ptn{0})\subseteq \C_{\ptn{0}}$.
So $\gamma\neq [0^{L_j-3}]$ since these objects belong to different 
$\NU_1$-tails, which are pairwise disjoint.

\emph{Case~3:}
Assume $j=J$, so $E$ is $B_{\mu}$ with its final $1$ removed.  
First assume $\mu\neq\ptn{1^k}$, so $B_{\mu}$ contains two $0$s and
$E$ is reduced. By~\eqref{eq:loop-end}, $E$ is a binary Dyck
vector beginning with $01^c0$, where 
$c=p_{r+1}+\epsilon+2\lfloor n_0/2\rfloor\geq 1$.
We claim that either $c\geq 2$, or $c=1$ and $E$ ends in $0$. 
This holds because $a\geq 2$ and $c=1$ imply $n_0=a-\epsilon<2$,
$a=2$, $\epsilon=1$, $p_{r+1}=0$, $n_r$ is odd, $0B_{\mu}$
given in~\eqref{eq:loop-end} ends in $01$, and so $E$ ends in $0$. 
By the claim and Theorem~\ref{thm:BDV-tail}(b), 
$[E]$ is in $\tail(\rho)$ in plateau $2$ or higher.
We can now repeat the proof from Case~2 to see that $\mind(\gamma)\leq L_J-3$.
To see $\gamma\neq [0^{L_J-3}]$, note that $E$ is a reduced BDV
of length $L_J-2$, so $\dinv(\gamma)=\area(E)\leq L_J-4$.
But $\dinv([0^{L_J-3}])=\binom{L_J-3}{2}>L_J-4$ since $L_J\geq L\geq 6$.
(To see why $L\geq 6$, note from Example~\ref{ex:TI2} that the 
only vectors $V=v(\lambda,a,\epsilon)$ with $\len(V)<6$ represent
$\TI_2(\mu)$ for $\mu=\ptn{111}$ and $\mu=\ptn{21}$, but we
are assuming $|\mu|\geq 6$.)

To finish Case~3, consider $\mu=\ptn{1^k}$. Here, $E=01^{k-1}$ is not reduced
and has length $k=L_J-2$, area $k-1$, and dinv $\binom{k-1}{2}$.
So $\gamma$ is the unique object in $\C_{\ptn{0}}=\tail(\ptn{0})$
having dinv $k-1$. 
By Theorem~\ref{thm:BDV-tail}, $[E]=[00^{k-2}]$ has $\mind=k-1$
and is the last object in plateau $k-2$ of $\tail(\ptn{0})$.
Now $k-1<\binom{k-1}{2}$ for all $k\geq 5$, so $\gamma$ appears
strictly before $[E]$ in $\tail(\ptn{0})$. 
This means $\mind(\gamma)\leq\mind([E]=k-1=L_J-3$,
and moreover $\gamma\neq [0^{L_J-3}]=[E]$.  

\begin{remark}\label{rem:E-in-tail1}
Letting $S_j(V)=0E1$, we proved in all three cases that
$[E]$ must belong to $\tail(\rho)$, not just $\tail_2(\rho)$.
\end{remark}

\subsection{Proof that Chains are Disjoint}
\label{subsec:prove-disj}

This section completes the proofs of Theorems~\ref{thm:extend-mu}
and~\ref{thm:extend-all} by verifying that conditions~\ref{def:basic-str}(e)
and~\ref{def:chain-coll}(e) hold for all newly constructed chains $\C_{\mu}$
indexed by $\mu$ of size $k$. In fact, we prove the stronger result
that all such new chains $\C_{\mu}$, along with the partial chains
$\tail_2(\xi)$ for all other $\xi$ of size $k$, are pairwise disjoint.

Because the maps $\NU$
and $\TI_2$ are one-to-one and each $\tail_2(\xi)$ is the $\NU$-segment
starting at $\TI_2(\xi)$, all second-order tails are pairwise disjoint.
We must show that the bridge parts
and antipodal parts of the various chains $\C_{\mu}$ do not overlap with 
each other or any second-order tail. Since $\NU_1$ is a bijection and
all parts are unions of $\NU_1$-segments, it suffices to analyze the
$\NU_1$-initial objects $[M_i]$ and $[A_j^*]$ in the bridge part
and antipodal part of $\C_{\mu}$.

\emph{Step~1.} We show that $[M_i]$ and $[A_j^*]$ do not
belong to any set $\tail_2(\xi)$. By Lemma~\ref{lem:bridge-vec}
and Proposition~\ref{prop:antip-part}, each $M_i$
and $A_j^*$ is a reduced Dyck vector starting with $00$ and containing
a $3$. Examining Definition~\ref{def:cyc-TDV}, we see that 
the reduction of a cycled ternary Dyck vector cannot have this form.
Step~1 now follows from Theorem~\ref{thm:tail2}(a).

\emph{Step~2.} We show that the bridge parts of the chains $\C_{\mu}$
do not overlap. It suffices to show that $\mu$ can be recovered uniquely
from any $\NU_1$-initial object $\gamma=[M_i]$ defined in the construction
of the bridge part of $\C_{\mu}$. Given such a $\gamma$, we first obtain
$M_i$ as the unique reduced Dyck vector representing $\gamma$.
By Lemma~\ref{lem:bridge-vec}, $M_i$ must have length $L$ and dinv $i$ 
for some $i\equiv D\pmod{2}$.  By definition, $M_i=00z^+$ where
$[0z^+]=[z]=c_{\lambda}(i-1)$ with $\lambda=\ftype(\mu)$.
We deduce $\lambda$ by finding the unique chain $\C_{\lambda}$ 
containing $[z]$. (This chain must already be known, given that
$\C_{\mu}$ was successfully constructed.)
Finally, since $\Psi(\mu)=(\lambda,L,D\bmod 2)$,
we recover $\mu$ by computing $\mu=\Psi^{-1}(\lambda,L,i\bmod 2)$.

\emph{Step~3.} We show that the antipodal parts of the chains $\C_{\mu}$
do not overlap. It suffices to show that $\mu$ can be recovered uniquely
from any $\NU_1$-initial object $\delta=[A_j^*]$ used in the construction
of the antipodal part of $\C_{\mu}$.  We recall the definition
of $A_j^*$ from~\S\ref{subsec:Ant}. Starting with $S=S_j(V^*)$, we write
$S=0E1$, find the unique $\rho$ with $[E]\in\C_{\rho}$, compute
$\gamma=c_{\rho^*}(\area(S)-1)$, let $z$ be the representative of $\gamma$
of length $\len(E)$, and set $A_j^*=00z^+$. We recover all these quantities
from $\delta$ as follows. First, $A_j^*$ is the unique reduced representative
of the Dyck class $\delta$. Then $L_j^*$ (the length of $S$) is
$\len(A_j^*)$. Dropping the first $0$ from $A_j^*$, we can find the
Dyck class $[0z^+]=[z]=\gamma$. We recover $\rho^*$ by finding the
unique chain $\C_{\rho^*}$ containing $\gamma$ (this chain and its
partner $\C_{\rho}$ must already be known, given that $\C_{\mu}$ 
was successfully constructed). Now $[E]$ must be the unique object
in $\C_{\rho}$ with $\dinv([E])=\area(z)$, since the proof of
Lemma~\ref{lem:Ant-map} shows both sides must equal $\dinv(S)-(L_j^*-2)$.
Next, $E$ itself is the representative of $[E]$ of length $L_j^*-2$,
and then $S_j(V^*)=0E1$. Finally, we recover $\mu^*$ (and hence $\mu$)
by finding the unique second-order tail $\tail_2(\mu^*)$ containing
$[S_j(V^*)]$. One way to do this is to follow the $\NU$-chain from $S_j(V^*)$
until the Dyck class $\TI(\mu^*)=[0B_{\mu^*}]$ is reached.
Then $\mu^*$ can be read off from the binary representative $0B_{\mu^*}$,
and $\mu$ is $\mu^{**}$.

\emph{Step~4.} We introduce a variation of the antipode map, called
$\Ant'$, that is an involution interchanging area and dinv and
preserving length, deficit, and $\mind$ (compare to Lemma~\ref{lem:Ant-map}(b)).
Inputs to $\Ant'$ are certain vectors $00z^+$ where
$z$ is a Dyck vector. Suppose $00z^+$ has length $\ell$, area $a$, dinv $d$,
and deficit $k$. By Lemma~\ref{lem:semi-stats}, $z$ must have length
$\ell-2$, area $a-(\ell-2)$, dinv $d-1$, and deficit $k-(\ell-2)<k$.
Suppose $[z]$ belongs to a known chain $\C_{\rho}$ with known opposite
chain $\C_{\rho^*}$, $\gamma=c_{\rho^*}(a-1)$ exists, and $\gamma$
is represented by a Dyck vector $w$ of length $\ell-2$. In this situation,
define $\Ant'(00z^+)=00w^+$; otherwise $\Ant'(00z^+)$ is not defined.
Note that $\len(w)=\ell-2=\len(z)$, $\defc(w)=|\rho^*|=|\rho|=\defc(z)$,
and $\dinv(w)=\dinv(\gamma)=a-1$. It follows that $\area(w)=d-(\ell-2)$
since $\area(w)+\dinv(w)=\binom{\len(w)}{2}-\defc(w)=\area(z)+\dinv(z)$.
Now Lemma~\ref{lem:semi-stats} shows $00w^+$ has length $\ell$,
area $d$, dinv $a$, and deficit $k$, as needed. Applying $\Ant'$
to input $00w^+$, we see (using $\rho^{**}=\rho$ and $[z]=c_{\rho}(d-1)$) 
that $\Ant'(00w^+)=00z^+$.  So $\Ant'$ is an involution on its domain.

\emph{Step~5.} We show that for $i\in\{A^*,A^*+2,\ldots,D-2,D\}$,
 $\Ant'$ interchanges $M_i(\mu)$ and $M_{A+D-i}(\mu^*)$. 
 From~\eqref{eq:DA-sum}, $A+D=\binom{L}{2}-k=A^*+D^*$, 
 so $i$ is in the given range if and only if
 $A+D-i\in\{A,A+2,\ldots,D^*-2,D^*\}$.  By Lemma~\ref{lem:bridge-vec} 
 and Remark~\ref{rem:MD-is-V}, $M_i(\mu)$ and $M_{A+D-i}(\mu^*)$
 are well-defined Dyck vectors of length $L$. Also,
 for $\lambda=\ftype(\mu)$, $M_i(\mu)=00z^+$ where $z$ is the length $L-2$ 
 representative of $c_{\lambda}(i-1)$; 
 and $M_{A+D-i}(\mu^*)=00w^+$ where $w$ is the length $L-2$ representative of
  $c_{\lambda^*}(A+D-i-1)$. Comparing these expressions to the
 definition of $\Ant'(00z^+)$ in Step~4, we need only check that
 $\area(M_i(\mu))=A+D-i$. This holds since Lemma~\ref{lem:bridge-vec}
 and~\eqref{eq:DA-sum} give $\area(M_i(\mu))
 =\binom{L}{2}-\defc(M_i(\mu))-\dinv(M_i(\mu)) =\binom{L}{2}-k-i=A+D-i$.

\emph{Step~6.} Let $A_j=\Ant(S_j(V))$.  We show that if $\Ant'(A_j)$ 
is defined, then $\Ant'(A_j)=S_j(V)$. Recall how $A_j$ is computed.
Let $S_j(V)$ have length $L'$, area $A'(=A-2j)$, and dinv $D'$.
Write $S_j(V)=0E1$, where $E$ is a TDV,
$\len(E)=L'-2$, $\dinv(E)=D'-(L'-2)$, and (by 
Remark~\ref{rem:E-in-tail1}) $[E]$ belongs to $\tail(\rho)$ for some $\rho$.
Then $A_j=00z^+$ where $\len(z)=L'-2$, $\dinv(z)=A'-1$,
and $[z]\in\C_{\rho^*}$.  We proved $\area(A_j)=\dinv(S_j(V))=D'$.
Assume $\Ant'(A_j)$ is defined and equals $00w^+$.
This means that $w$ is a Dyck vector of length $L'-2$ such that
$[w]\in\C_{\rho}$ and $\dinv([w])=D'-1=\dinv([E])+L'-3$.
Because $[E]$ is in $\tail(\rho)$, $[w]$ must be $\NU_1^{L'-3}([E])$.

We claim that $E$ is not reduced. Otherwise, $[E]$ belongs to a 
plateau of $\tail(\rho)$ consisting of $L'-3$ (or fewer) objects
with $\mind$ equal to $\mind(E)=\len(E)=L'-2$ (Theorem~\ref{thm:tail-profile}).
But then $[w]=\NU_1^{L'-3}([E])$ has $\mind([w])>L'-2=\len(w)$,
which contradicts $w$ being a Dyck vector. So the TDV $E$ is not reduced,
say $E=0Y^+$ where $Y$ is a BDV of length $L'-3$.  By 
Proposition~\ref{prop:bin-NU}, 
$[w]=\NU_1^{L'-3}([Y])=[Y0]$. Since $\len(w)=L'-2=\len(Y0)$,
we get $w=Y0$ and $\Ant'(A_j)=00w^+=00Y^+1=0E1=S_j(V)$.

\emph{Step~7.} We show that for any two flagpole partitions $\mu\neq\nu$
 such that $\C_{\mu}$ and $\C_{\nu}$ have been constructed,
 the bridge part of $\C_{\mu}$ does not overlap the antipodal part of 
$\C_{\nu}$. To get a contradiction, assume there exist $i\in\{A^*,A^*+2,\ldots,
 D-2\}$ and $j>0$ with $v=M_i(\mu)=A_j'$, where 
 $A_j'=\Ant(S_j(V'))$ for $V'$ the reduced representative of $\TI_2(\nu^*)$.
 By Step~5, $\Ant'(v)=M_{A+D-i}(\mu^*)$. Since $\Ant'(v)$ is defined,
 Step~6 shows that $\Ant'(v)=S_j(V')$. We have now contradicted
 Step~1, since $S_j(V')\in\tail_2(\nu^*)$, while
 $M_{A+D-i}(\mu^*)\not\in\tail_2(\nu^*)$. The index $i=A^*$ is special.
 For this $i$,~\eqref{eq:DA-sum} and Remark~\ref{rem:MD-is-V} 
 imply $M_{A+D-i}(\mu^*)=M_{D^*}(\mu^*)=V^*$. But $V^*$ is in
 $\tail_2(\mu^*)$ and thus not in $\tail_2(\nu^*)$, since $\mu^*\neq\nu^*$.
 These contradictions prove Step~7.

\section{Generalized Flagpole Partitions}
\label{sec:gen-flagpole}

We know that $\mu$ is a flagpole partition if 
$\TI_2(\mu)=[v(\lambda,a,\epsilon)]$ for sufficiently large $a$
(namely, $a\geq a_0(\lambda)$, which is equivalent to 
$v(\lambda,a,\epsilon)$ having length $L\geq |\lambda|+6$).  Examining the 
constructions of Sections~\ref{sec:make-flag-chains} and~\ref{sec:proofs}, 
we see that the condition $L\geq |\lambda|+6$ was used only three times: 
showing that $\mu^*$ is well-defined in Lemma~\ref{lem:check-mu*}; 
proving $D\geq A^*$ in~\S\ref{subsec:exceptional};
and checking our claims about bridge generators in~\S\ref{subsec:prove-bridge}.
By making minor modifications to these three proofs, we can extend
the chain constructions for flagpole partitions to a larger class
of partitions called generalized flagpole partitions.
Informally, these new partitions arise by replacing the lower bound
$a_0(\lambda)$ by a smaller number (often as small as $2$).
We give the formal definition next, then discuss the changes needed
for the three proofs.  

\subsection{Definition of Generalized Flagpole Partitions}
\label{subsec:def-gen-flag}

\begin{definition}\label{def:gen-flag}
Let $\rho\mapsto\rho^*$ be a size-preserving involution
defined on some collection $\mcP$ of partitions.
Suppose $\mu$ is a partition of size $k$ such that
$\TI_2(\mu)=[V]$ where $V=v(\lambda,a,\epsilon)$ has length $L$.
We say $\mu$ is a \emph{generalized flagpole partition} 
(relative to the given involution) if and only if $\lambda\in\mcP$ and
\begin{equation}\label{eq:bound-L}
 L\geq 5+\lambda_1+\ell(\lambda)\quad\mbox{and}\quad
 L\geq 5+\lambda_1^*+\ell(\lambda^*).
\end{equation}
Since $L=a+3+\lambda_1+\ell(\lambda)$ (Remark~\ref{rem:flag-init}), 
$\mu$ is a generalized flagpole partition iff $\lambda\in\mcP$ and
$a\geq 2$ and $a\geq 2+\lambda_1^*+\ell(\lambda^*)
 -\lambda_1-\ell(\lambda)$. Note this condition reduces to $a\geq 2$
when $\lambda=\lambda^*$.
\end{definition}

\begin{example}
For any chain collection $(\mcP,I,\C)$,
we found that $\lambda=\ptn{3321^4}$ has $\lambda^*=\ptn{531^4}$
in~\S\ref{subsec:flag-invo}. Since $\lambda_1+\ell(\lambda)=10$
and $\lambda_1^*+\ell(\lambda^*)=11$, every $\mu$ such that
$\TI_2(\mu)=[v(\ptn{3321^4},a,\epsilon)]$ for some $a\geq 3$ 
(equivalently, $L\geq 16$) is a generalized flagpole partition. 
For $\mu$ to be a flagpole partition,
we would need $a\geq a_0(\lambda)=5$ (equivalently, $L\geq 18$). 
Similarly, every $\mu$ such that $\TI_2(\mu)=[v(\ptn{531^4},a,\epsilon)]$ 
for some $a\geq 2$ (equivalently, $L\geq 16$) is a generalized flagpole 
partition. For $\mu$ to be a flagpole partition,
we would need $a\geq 4$ (equivalently, $L\geq 18$).  
The difference in the bounds on $a$ becomes more dramatic when the diagrams
of $\lambda$ and $\lambda^*$ have many cells outside the first row 
and column.
\end{example}

For all generalized flagpole partitions $\mu$, $V$ starts with $0012$.
Thus the analysis of $\tail_2(\mu)$ in~\S\ref{subsec:init-tail2} still applies.
The next theorem is an unconditional result not relying on the
assumed existence of any chain collections.

\begin{theorem}\label{thm:gflag-asym}
Let $\lambda\mapsto\lambda^*$ be any size-preserving involution 
on the set of all integer partitions. Let $g(k)$ be the number of generalized
flagpole partitions of size $k$ (defined using this involution).
(a)~For all $k$,
\begin{equation}\label{eq:gf-bound}
 g(k)\geq \sum_{j=1}^{k-1}
  2\max\left\{0,p(j)-2\sum_{i=k-3-j}^j q_i(j)\right\}. 
\end{equation}
(b)~For any real $c<1$ and all sufficiently large $k$, $g(k)>p(k)^c$.
\end{theorem}

See~\S\ref{subsec:gflag-proof} for the proof.  Remark~\ref{rem:flag-asym}
gives the corresponding asymptotic enumeration of flagpole partitions.

\begin{theorem}\label{thm:extend-genflag}
Theorems~\ref{thm:extend-mu} and~\ref{thm:extend-all} hold with
flagpole partitions replaced by generalized flagpole partitions throughout.
\end{theorem}

For Theorem~\ref{thm:extend-mu}, generalized flagpole partitions
are defined relative to the involution $I$ in the given chain collection.
For Theorem~\ref{thm:extend-all}, generalized flagpole partitions
of size $k$ are defined relative to the involution $I^{k-1}$ on $\mcP^{k-1}$.
Both theorems follow from the constructions and proofs already given,
after making the modifications in the next three subsections.

\subsection{Defining $\mu^*$ for Generalized Flagpole Partitions}
\label{subsec:gen-mu*}

We modify the bijection $\Psi$ from Lemma~\ref{lem:flag-bij2} as follows.
Let $F_k$ be the set of generalized flagpole partitions of size $k$.
Let $H_k$ be the set of triples $(\lambda,L,\eta)$ such that
$\lambda$ is an integer partition in $\mcP$ of size less than $k$,
$L=k+2-|\lambda|$, $L$ satisfies~\eqref{eq:bound-L}, and $\eta\in\{0,1\}$.
Given $\mu\in F_k$, say $\TI_2(\mu)=[V]$ where $V=v(\lambda,a,\epsilon)$
has length $L$, dinv $D$, and area $A$.
Recall (Remark~\ref{rem:flag-init}) that $k=\defc(V)=|\lambda|+L-2$.
Define $\Psi_k:F_k\rightarrow H_k$ by $\Psi_k(\mu)=(\lambda,L,D\bmod 2)$.
The proof of Lemma~\ref{lem:flag-bij2} shows that
$\Psi_k$ is a bijection; we need only replace the old condition
$L\geq |\lambda|+6$ by~\eqref{eq:bound-L}.

Furthermore,~\eqref{eq:bound-L} ensures that $(\lambda^*,L,A\bmod 2)$
also belongs to the codomain $H_k$. So we may define $\mu^*$ to
be the unique object in $F_k$ with $\Psi_k(\mu^*)=(\lambda^*,L,A\bmod 2)$.
The rest of the proof of Lemma~\ref{lem:check-mu*} goes through
with no changes.

\subsection{Proving $D\geq A^*$ for Generalized Flagpole Partitions}
\label{subsec:exceptional2}

In~\S\ref{subsec:exceptional}, we used $L\geq |\lambda|+6$ to eliminate
$|\lambda|$ in the estimate 
\begin{equation}\label{eq:est1}
D-A^*\geq \binom{L}{2}-5L+12+\lambda_1 +\lambda_1^*-|\lambda|. 
\end{equation}
We give a modified estimate here 
using~\eqref{eq:bound-L}. First consider the case $\lambda\neq\ptn{0}$.
The diagram of $\lambda$ fits in a rectangle with $\ell(\lambda)$ rows
and $\lambda_1$ columns, so $|\lambda|\leq \lambda_1\ell(\lambda)$.
Moreover, $\lambda_1\ell(\lambda)\leq\max(\lambda_1^2,\ell(\lambda)^2)
 \leq\lambda_1^2+\ell(\lambda)^2$.
Since $L\geq\lambda_1+\ell(\lambda)+5$, we get
\[ (L-5)^2\geq \lambda_1^2+2\lambda_1\ell(\lambda)+\ell(\lambda)^2
  \geq 3|\lambda|. \]
Thus, $-|\lambda|\geq -(L-5)^2/3$. We also have $\lambda_1+\lambda_1^*\geq 2$
since $\lambda\neq\ptn{0}$. Using these estimates in~\eqref{eq:est1}
and simplifying, we get $D-A^*\geq (L^2-13L+34)/6$. This polynomial in $L$
exceeds $-2$ for all $L$, so the even integer $D-A^*$ must be nonnegative.

If $\lambda=\ptn{0}$, then~\eqref{eq:est1} becomes
$D-A^*\geq\binom{L}{2}-5L+12$. Here $D-A^*\leq -2$ is possible only
for $5\leq L\leq 8$. The exceptional cases $L=6,7,8$ were already
examined in the table in~\S\ref{subsec:exceptional}. 
If $L=5$, then $|\mu|=|\lambda|+L-2=3$, but we know 
from Definition~\ref{def:chain-coll}(a) that $|\mu|\geq 6$.

\subsection{Modified Bridge Analysis}
\label{subsec:prove-bridge2}

We modify two calculations in~\S\ref{subsec:prove-bridge} where
the old assumption $L\geq |\lambda|+6$ was used. To prove 
$A^*-1\geq\ell(\lambda^*)$, use~\eqref{eq:TI2-stats}
and the second part of~\eqref{eq:bound-L} to get
\[ A^*-1\geq 2L-\lambda_1^*-7\geq 3+\lambda_1^*+2\ell(\lambda^*)
 >\ell(\lambda^*). \]
Since $\mind(\TI(\lambda))=\lambda_1+\ell(\lambda)+1$
  and $\mind(\TI(\lambda^*))=\lambda_1^*+\ell(\lambda^*)+1$,
the bound~\eqref{eq:mind-est} becomes
\[ \mind(\gamma)\leq\max(2+\lambda_1+\ell(\lambda),
  2+\lambda_1^*+\ell(\lambda^*))\leq L-3. \]

\subsection{Proof of Theorem~\ref{thm:gflag-asym}}
\label{subsec:gflag-proof}

Let $g(k)$ be the number of generalized flagpole partitions of size $k$,
 relative to a size-preserving involution $\lambda\mapsto\lambda^*$ on
 the set of all integer partitions. Fix a real $c<1$.  We prove $g(k)>p(k)^c$ 
 for all sufficiently large $k$.  It suffices to consider $c>1/2$.
For any nonzero partition $\lambda$, let 
$h(\lambda)=\lambda_1+\ell(\lambda)-1$, which is the longest hook-length
in the diagram of $\lambda$. For $0<i\leq j$, let $q_i(j)$ be the
number of partitions $\lambda$ of size $j$ with $h(\lambda)=i$. We begin by
proving the bound~\eqref{eq:gf-bound}.

Fix $j$ between $1$ and $k-1$, and consider a fixed partition 
$\lambda$ of size $j$. The Dyck vector $v(\lambda,a,\epsilon)$
has deficit $k$ iff the length $L$ of this vector (which is a constant
plus $a$) satisfies $L=k+2-|\lambda|=k+2-j$. The corresponding partition
$\mu=\TI_2^{-1}([v(\lambda,a,\epsilon)])$ is a generalized flagpole
partition iff $k+2-j\geq h(\lambda)+6$ and $k+2-j\geq h(\lambda^*)+6$
(by~\eqref{eq:bound-L}). So for each partition $\lambda$ of size $j$
such that $h(\lambda)\leq k-4-j$ and $h(\lambda^*)\leq k-4-j$,
we obtain two generalized flagpole partitions of size $k$ (since $\epsilon$
can be $0$ or $1$).

Let $P$ be the set of partitions of size $j$. $P$ is the
disjoint union of the sets $A=\{\lambda\in P:h(\lambda)\leq k-4-j\}$ and
$B=\{\lambda\in P:h(\lambda)\geq k-3-j\}$.
The partitions in $P$ are paired by the involution 
$\lambda\mapsto\lambda^*$. For each $\lambda\in A$ that
is paired with some $\lambda^*\in A$, we obtain $2$ generalized
flagpole partitions of size $k$. In the worst case, every partition
in $B$ pairs with something in $A$. Then we would still have at least
$|A|-|B|=|P|-2|B|$ partitions in $A$ that pair with something in $A$.
So we get at least $2\max(0,|P|-2|B|)$ generalized flagpole partitions
of size $k$ from this choice of $j$.
Now $|P|=p(j)$ and $|B|=\sum_{i=k-3-j}^j q_i(j)$.
Summing over $j$ gives~\eqref{eq:gf-bound}.

Next we estimate $q_i(j)$. To build the diagram of a partition
counted by $q_i(j)$, first select a corner hook of size $i$
(consisting of the first row and column of the diagram) in any 
of $i$ ways. Then fill in the remaining cells of the diagram
with some partition of $j-i$. Not every such partition fits inside
the chosen hook, but we get the bound $q_i(j)\leq ip(j-i)$. 
So~\eqref{eq:gf-bound} becomes
$ g(k)\geq \sum_{j=1}^{k-1}
  \max\left\{0,2p(j)-\sum_{i=k-3-j}^j 4ip(j-i)\right\}$. 
We prove $g(k)>p(k)^c$ (if $k$ is large enough) 
by finding a single index $j$ such that
\[ 2p(j)-\sum_{i=k-3-j}^j 4ip(j-i) > p(j) > p(k)^c. \] 
We claim $j=\lceil kc\rceil$ will work. Recall the Hardy--Ramanujan estimate
$p(k)=\Theta\left(k^{-1}\exp(\pi\sqrt{2k/3})\right)$. We have
$p(j)=\Theta\left((ck)^{-1}\exp(\pi\sqrt{2ck/3})\right)$ and
$p(k)^c=\Theta\left(k^{-c}\exp(\pi\sqrt{2c^2k/3})\right)$. 
Since $c>c^2$, $p(j) > p(k)^c$ for large enough $k$.  
Next we show $\sum_{i=k-3-j}^j 4ip(j-i)<p(j)$ for large $k$.
There are $j-(k-4-j)=2j-k+4\leq k+4$ summands, and each summand
is at most $4jp(j-(k-3-j))\leq 4kp(2j-k+3)$.  
So it suffices to show $(k+4)4kp(2j-k+3)<p(j)$ for large $k$.
Using $j=\lceil kc\rceil$, we compute
\[p(j)/p(2j-k+3) =\Theta\Big(\exp\left[\pi\sqrt{2kc/3}
  -\pi\sqrt{(2c-1)2k/3+2}\right]\Big).\]
Now, it is routine to check that for $A>B>0$, any $C$, and any polynomial
$f(k)$, $\exp(\sqrt{Ak}-\sqrt{Bk+C})>f(k)$ for large enough $k$.
This follows by taking logs, dividing by $\sqrt{k}$, and using
L'Hopital's Rule to take the limit as $k$ goes to infinity.
Since $0<2c-1<c$, we get $p(j)/p(2j-k+3)>4k(k+4)$ for large $k$, as needed.


\end{document}